\documentclass[twocolumn]{svjour3}          

\smartqed  
\usepackage{amssymb,mathtools, dsfont, amsfonts}
\usepackage{graphicx}
\usepackage[lofdepth,lotdepth]{subfig}
\usepackage{tikz}
\usepackage{pgfplots}
\usetikzlibrary{plotmarks}
\usepackage{algorithm}
\usepackage[noend]{algpseudocode}
\usepackage{varwidth}
\usepackage{booktabs,multirow, rotating}
\usepackage{tabularx}
\usepackage{subfig}

\allowdisplaybreaks[1]

\newcommand{\egomotion}{ego-motion }
\DeclareTextFontCommand{\emph}{\em}

\begin{document}
\sloppy



\newcommand{\bitem}{\begin{itemize}}
\newcommand{\eitem}{\end{itemize}}
\newcommand{\mc}[1]{\mathcal{#1}}
\newcommand{\mb}[1]{\mathbb{#1}}
\newcommand{\mf}[1]{\mathfrak{#1}}
\newcommand{\ms}[1]{\mathscr{#1}}
\newcommand{\on}[1]{\operatorname{#1}}
\newcommand{\II}{\mathbb{I}}
\newcommand{\N}{\mathbb{N}}
\newcommand{\R}{\mathbb{R}}
\newcommand{\C}{\mathbb{C}}
\newcommand{\F}{\mathcal{F}}
\newcommand{\B}{\mathbb{B}}
\newcommand{\U}{\mathbb{U}}
\newcommand{\EE}{\mathbb{E}}
\newcommand{\V}{\mathbb{V}}
\newcommand{\Q}{\mathbb{Q}}
\newcommand{\Z}{\mathbb{Z}}
\newcommand{\PP}{\mathbb{P}}
\newcommand{\TT}{\mathbb{T}}
\newcommand{\bpm}{\begin{pmatrix}}
\newcommand{\epm}{\end{pmatrix}}
\newcommand{\bsm}{\left(\begin{smallmatrix}}
\newcommand{\esm}{\end{smallmatrix}\right)}
\newcommand{\T}{\top}
\newcommand{\ul}[1]{\underline{#1}}
\newcommand{\ol}[1]{\overline{#1}}
\newcommand{\la}{\langle}
\newcommand{\ra}{\rangle}
\newcommand{\si}{\sigma}
\newcommand{\SI}{\Sigma}
\newcommand{\mrm}[1]{\mathrm{#1}}
\newcommand{\msf}[1]{\mathsf{#1}}
\newcommand{\mfk}[1]{\mathfrak{#1}}
\newcommand{\row}[2]{{#1}_{#2,\bullet}}
\newcommand{\col}[2]{{#1}_{\bullet,#2}}
\newcommand{\df}[2]{\frac{\partial #1}{\partial #2}}
\newcommand{\p}{\partial}
\newcommand{\veps}{\varepsilon}
\newcommand{\toset}{\rightrightarrows}
\newcommand{\w}{\omega}
\newcommand{\gdw}{\Leftrightarrow}
\newcommand{\vphi}{\varphi}
\newcommand{\ora}[1]{\overrightarrow{#1}}
\newcommand{\ola}[1]{\overleftarrow{#1}}
\newcommand{\oset}[2]{\overset{#1}{#2}}
\newcommand{\uset}[2]{\underset{#1}{#2}}
\newcommand{\SE}{\operatorname{SE}_{3}}
\newcommand{\se}{\mathfrak{se}_{3}}
\newcommand{\SO}{\operatorname{SO}_{3}}
\newcommand{\so}{\mathfrak{so}_{3}}
\newcommand{\Ad}{\operatorname{Ad}}
\newcommand{\dist}{\operatorname{dist}}
\newcommand{\etr}{\operatorname{etr}}
\newcommand{\vex}{\operatorname{vex}}
\newcommand{\Psym}{\mathbb{P}_{s}}
\newcommand{\Pskew}{\mathbb{P}_{a}}
\newcommand{\vecso}{\operatorname{vec}_{\mathfrak{so}}}
\newcommand{\vecse}{\operatorname{vec}_{\mathfrak{se}}}
\newcommand{\matso}{\operatorname{mat}_{\mathfrak{so}}}
\newcommand{\matse}{\operatorname{mat}_{\mathfrak{se}}}
\newcommand{\Hess}{\operatorname{Hess}}
\newcommand{\grad}{\operatorname{grad}}
\newcommand{\tr}{\operatorname{tr}}
\newcommand{\argmin}{\operatorname{arg\,min}}
\newcommand{\kronse}{\otimes_{\mathfrak{se}}}
\newcommand{\diag}{\operatorname{diag}}
\newcommand{\Exp}{\operatorname{Exp}}
\newcommand{\D}{\mathbf{d}}
\newcommand{\Id}{\operatorname{Id}}
\newcommand{\G}{\mathcal{G}}
\newcommand{\g}{\mathfrak{g}}
\newcommand{\vecg}{\operatorname{vec}_{\mathfrak{g}}}
\newcommand{\matg}{\operatorname{mat}_{\mathfrak{g}}}
\newcommand{\ad}{\operatorname{ad}}
\newcommand{\Log}{\operatorname{Log}}

\newcommand{\eins}{\mathds{1}}

\newcommand{\cperp}[3]{{#1} \perp\negthickspace\negthinspace\negthickspace\perp {#2} \,|\, {#3}}

\newcommand{\LG}[1]{\mathrm{#1}}
\newcommand{\Lg}[1]{\mathfrak{#1}}

\newcommand{\st}[1]{{\scriptstyle #1}}
\newcommand{\sst}[1]{{\scriptscriptstyle #1}}


\title{Second-Order Recursive Filtering on the Rigid-Motion Lie Group $\SE$ Based on Nonlinear Observations}


\titlerunning{Second-Order Recursive Filtering on the Rigid-Motion Lie Group $\SE$ Based on Nonlinear Observations}        

\author{Johannes Berger \and Frank Lenzen \and Florian Becker \and Andreas Neufeld \and
        Christoph Schn\"orr 
}

\authorrunning{J.~Berger {\em et al.}} 

\institute{J.~Berger \at
              Research Training Group 1653, Image and Pattern Analysis Group, Heidelberg University, Speyerer Str.~6, 69115 Heidelberg, Germany   \\
              \email{johannes.berger@iwr.uni-heidelberg.de}           
\and
           F.~Becker, F.~Lenzen\at
              Heidelberg Collaboratory for Image Processing, Heidelberg University (HCI), Speyerer Str.~6, 69115 Heidelberg, Germany \\
              \email{becker@math.uni-heidelberg.de,Frank.Lenzen@iwr.uni-heidelberg.de}
           \and
           A.~Neufeld, C.~Schn\"orr\at
              Image and Pattern Analysis Group, Heidelberg University, Speyerer Str.~6, 69115 Heidelberg, Germany \\
              \email{\{neufeld,schnoerr\}@math.uni-heidelberg.de}
}

\date{}

\maketitle

\begin{abstract}

Camera motion estimation from observed scene features is an important task in image processing to increase the accuracy of many methods, e.g.\ optical flow and structure-from-motion. Due to the curved geometry of the state space $\SE$ and the non-linear relation to the observed optical flow,
many recent filtering approaches use a first-order approximation and assume a Gaussian {\em a posteriori }distribution or restrict the state to Euclidean geometry. The physical model is usually also limited to uniform motions.

We propose a second-order minimum energy filter with a {\em generalized kinematic model} that copes with the full geometry of $\SE$ as well as with the nonlinear dependencies between the state space and observations. The derived filter enables reconstructing motions correctly for synthetic and real scenes, e.g.\ from the KITTI benchmark.
Our experiments confirm that the derived minimum energy filter with higher-order state differential equation copes with higher-order kinematics and is also able to minimize model noise. We also show that the proposed filter is superior to state-of-the-art extended Kalman filters on Lie groups in the case of linear observations and that our method reaches the accuracy of modern visual odometry methods.

\keywords{Minimum Energy Filter \and Lie Group \and Recursive Filtering \and Constant Acceleration Model \and Optimal Control \and Visual Odometry}

\end{abstract}

\section{Introduction}
\label{intro}

\subsection{Overview and Motivation}
Camera motion estimation is a fundamental task in many important applications (e.g.\ autonomous driving, robotics) in computer vision.
It is an essential component of structure-from-motion, simultaneous localization and mapping (SLAM) and odometry tasks. Furthermore it aids as additional prior for e.g.\ optical flow methods.
In the proposed approach, the \egomotion of the camera is fully determined solely by the apparent motion of visual features (optical flow) as recorded by the camera  without needing additional sensors such as GPS or acceleration sensors.

Although the camera motions can be reconstructed correctly from only two consecutive frames~\cite{nister2004efficient,hartley1997defense},
the best performing methods take into account multiple frames. They are more robust against the influence of erroneous correspondence estimates. 
Two approaches to making use of the temporal context can be distinguished:
batch approaches -- such as 
bundle adjustment methods~\cite{triggs2000bundle} -- first record all the frames and fit in a smooth camera path afterwards. They sometimes also incorporate loop closure constraints~\cite{whelan2013deformation} to further improve camera motion accuracy.
Factorization methods~\cite{Sturm-et-al-ECCV1996,nasihatkon2015generalized} create the problem of jointly estimating camera poses and scene points as a matrix decomposition problem. 
These batch approaches have the potential of working exactly as they make use of all available information.
On the other hand, they hardly work in real-time applications, as the volume of incorporated information increases linearly with time.

In contrast, {\em online} approaches apply
sliding window techniques~\cite{badino2013visual,bellavia2013robust,bourmaud2015robust} that track features on multiple frames to increase robustness and compute the best fitting motion. 

A mathematical description of (online) temporal smoothing is given by the notion of (stochastic) filtering~\cite{Stochastic-Filtering-09}: in case of on the one hand, an ODE describing the behavior of a latent variable, and on the other hand, observations that depend on the latent variable, the goal is to estimate the most likely value of the unknowns. However, stochastic filters suffer from non-linear dependencies of latent variables and observations as well as geometric constraints and unknown probability distributions.

That is the reason why we chose deterministic {\em Minimum Energy Filters} that do not need information about distributions and cope with the non-linearities of the observer equation and the geometry of the state space~$\SE$ in~\cite{berger2015second}. Since the state equation of the \egomotion in~\cite{berger2015second} is simple and requires small weights on the penalty term for the model noise, however, this approach is sensitive against noise and requires good observation data. 

Therefore, in this paper, 
we extend our previous work~\cite{berger2015second}  to a second-order model with constant {\em acceleration} assumption which is more stable and shows better convergence.
In our experiments, we demonstrate significantly improved performance both
on synthetic data with higher-order kinematic scenarios and on the challenging KITTI benchmark~\cite{Geiger2012CVPR}. Comparison with novel {\em continuous/discrete extended Kalman filters on Lie Groups}~\cite{bourmaud2015continuous} shows that our approach -- although being less general than~\cite{bourmaud2015continuous} -- leads to better results and is robust against imperfect initializations.

\subsection{Related Work}
Incorporation of temporal context -- in terms of (partial) differential equations -- into the estimation of latent variables has a long tradition in many common applications, e.g.\ robotics, aviation and astronautics.
Starting from the seminal work of Kalman~\cite{kalman1960new} considering Gaussian noise and linear filtering equations, stochastic filters had have great success 
in many important areas of mathematics, computer sciences and engineering
during the last fifty years. The filtering methods have been improved during the last decades to cope with nonlinearities of state and observation equations, such as extended Kalman filters~\cite{jazwinski2007stochastic}, unscented Kalman filters~\cite{julier1997new} and particle filters~\cite{arulampalam2002tutorial}. For a detailed overview of these methods we refer to~\cite{Stochastic-Filtering-09,Daum2005}.

However, one strong limitation of stochastic filters represents the fact that the {\em a posteriori} distribution is usually unknown and, in general, is infinite dimensional due to the nonlinear dependencies. To cover a large bandwidth of {\em a posteriori} distributions Brigo {\em et al.} approximated them by distributions of the exponential family~\cite{NonlFilterExpFamily-99}. In contrast, particle filters try to sample from them~\cite{arulampalam2002tutorial}. Extended and unscented Kalman filters, on the other hand, only allow distributions that are Gaussian. 

Although these methods work successfully for many real-valued problems, they cannot be easily transferred to filtering problems which are constrained to manifolds, appearing in many modern engineering and robotic applications. Therefore, in the last decade, several strategies have been developed to adapt classical unconstrained filters to filtering problems on specific Lie groups and Riemannian manifolds: Kalman filters were transferred to the manifold of symmetric positive definite matrices~\cite{tyagi2008}. Extended Kalman filters on~$\SO$~\cite{markley2003attitude} with symmetry preserving observers~\cite{bonnabel2007left} were elaborated. 
Particle filters on~$\SO$ and~$\SE$ were proposed in~\cite{kwon2007} as well on Stiefel~\cite{tompkins2007bayesian} and on Grassman manifolds~\cite{rentmeesters2010efficient}. An application of particle filters to monocular SLAM is reported in~\cite{kwon2010monocular}. 

Recently, unscented Kalman filters were generalized to Riemannian manifolds~\cite{hauberg2013}. Since then, extended Kalman filters for constrained model and observation equations were developed~\cite{bourmaud2015continuous} for general Lie groups based on the idea of the Bayesian fusion~\cite{Wolfe2011}.

However, although stochastic filters have been adapted to curved spaces and non-linear measurement equations, they still require assumptions about the {\em a posteriori} distributions, e.g.\ to be Gaussian. Furthermore, while transferring related concepts of probability theory and stochastic analysis to Riemannian manifolds is mathematically feasible \cite{hsu2002stochastic,chirikjian2011stochastic,chikuse2012statistics}, exploiting them computationally for stochastic filtering seems involved. 
The widely applied particle filters also have limitations in connection with manifolds since the sampling requirements of particles become expensive \cite{kwon2010monocular}.

A different way to approach a solution to the filtering problem was proposed by Mortensen \cite{mortensen1968}. Rather than trying to cope with the probabilistic setting of the filtering problem, he investigated the filtering problem from the viewpoint of optimal control. By using the control parameter to model noise and by integrating a quadratic penalty function over the time, he found a first-order optimal {\em Minimum Energy Filter}. The advantage of this method is that it does not rely on assumptions about, or approximations of, the {\em a posteriori} distribution and that {\em Hamilton-Jacobi-Bellman equation} provides a well-defined optimality criterion.
It was shown theoretically in~\cite{krener2003convergence} that the minimum energy estimator converges with {\em exponential} speed for control systems on~$\R^{n}$ that are uniformly observable. 

The first article applying the minimum energy filters to geometrically constrained problems used perspective projections in the case of vectorial measurements~\cite{aguiar2006minimum} .
The minimum energy filters were generalized to {\em second-order} filters on specific Lie groups
with the help of geometric control theory in~\cite{jurdjevic1997,agrachev2004,saccon2011lie}.
The Minimum Energy Filter, as introduced by Mortensen~\cite{mortensen1968}, was generalized to the Lie group~$\SO$ for the case of linear observation equations~\cite{zamani2012} and for attitude estimation~\cite{zamani2013minimum}.
Further follow-up work \cite{saccon2013optimal} generalized the filter to non-compact Lie groups~\cite{saccon2013second}.

In this article, we greatly elaborate our initial work on camera estimation using nonlinear measurement equations, especially by moving from a {\em constant velocity} assumption~\cite{berger2015second} to a {\em second-order state equation} with {\em constant acceleration} model. In addition, we investigate generalized kinematic models of arbitrary order.

\subsection{Contribution and Organization}
Our contributions reported in this paper amount
\begin{itemize}
\item to generalize the constant camera {\em velocity} model from~\cite{berger2015second} 
(non-linear measurement model)
to polynomial models, in particular the constant {\em acceleration} model;

\item to provide a complete derivation of the second-order minimum energy filter~\cite{saccon2013second} as applied to camera motion estimation together with robust numerics that are consistent with the geometry and the structure of matrix Riccati equations;

\item to report experiments demonstrating that {\em higher-order kinematic models} are more accurate than the {\em constant velocity model}~\cite{berger2015second} on synthetic (with kinematic camera tracks) and real world data and that they enable to reconstruct higher-order information;

\item to report experiments comparing our approach to state-of-the-art extended Kalman Filters on Lie groups~\cite{bourmaud2015continuous}, 
indicating that our method is superior in coping with non-linearities of the observation function as well as 
in being more robust against imperfect initializations.

\end{itemize}

In the next section, we introduce the filtering equations related to our problem of camera motion reconstruction. Next, we describe the basics of minimum energy filters and detail how to apply the (operator-valued) minimum energy filter derived from~\cite{saccon2013second} to our scenario. The numerical integration schemes of the ODEs for the optimal state will be given in Section~\ref{sec:numInt}. We will confirm the theoretical results in Section~\ref{sec:exp} by experiments on synthetic and real world data and thus underline the applicability of our approach.

\subsection{Notation} \label{sec:notation}
{\small
\begin{tabular}{ll}
$\on{GL}_{4}$ & General Linear group \\
$\SO$ & Special Orthogonal group \\
$\SE$ & Special Euclidean group \\
$\se$ & Lie algebra of $\SE$ \\
$\vecse: \se \rightarrow \R^{6}$ & vectorization operator \\
$\matse = \vecse^{-1}$ & inverse of $\vecse$ \\
$\G$ & (product) Lie group $\SE \times \R^{6}$ \\
$\g$ & Lie algebra of $\G$ \\
$T_{G}\G$ & tangent space of $\G$ at $G$ \\
$\vecg: \g \rightarrow \R^{12}$ & vectorization operator \\
$\matg =: \vecg^{-1}$ & inverse of $\vecg$ \\
$\Exp_{\G}$ & exponential map on $\G$ \\
$\Log_{\G}$ & logarithmic map on $\G$ \\
$\on{Pr}: \R^{4\times 4} \rightarrow \se $ & projection onto Lie algebra $\se$ \\
$L_{G}H:=GH$ & left translation \\
$T_{H}L_{G}$ & tangent map of left translation at $H$ \\
$G \eta := T_{\Id}L_{G}\eta $ & shorthand for tangent map \\
$G^{-1} \eta := T_{\Id}L_{G}^{\ast}\eta $ & shorthand for dual of tangent map\\
$\Id$ & identity element of Lie group \\
$\la \xi,\eta \ra_{G}$ & Riemannian metric at $G \in \G$ \\
$\la \xi, \eta \ra = \la \xi, \eta \ra_{\Id}$ & Riemannian metric on Lie algebra $\g$ \\
$\la x,y \ra$ & scalar product on $\R^{n}$ \\
$\nabla_{\cdot}\cdot$ & Levi-Civita connection on $T\G$ \\
$\omega_{\chi}\eta:=\omega(\chi, \eta) := \nabla_{\chi}\eta$ & connection function for $\chi, \eta \in \g$  \\
$\omega_{\chi}^{\leftrightharpoons}\eta:= \omega_{\eta}\chi$ & swap operator\\
$\la \omega_{\chi}^{\ast}\eta, \xi \ra:= \la  \eta, \omega_{\chi} \xi\ra$ & dual of connection function \\
$ \la \omega_{\chi}^{\leftrightharpoons \ast} \eta, \xi \ra := \la \eta, \omega_{\xi} \chi \ra$ & dual of swap operator \\
$[\cdot, \cdot]$ & Lie bracket on considered  \\ & Lie group, matrix commutator \\
$\D f(G)$ & differential/Riemannian gradient \\ & of $f$ at $G$ \\
$\D f(G)[\eta]$ & directional derivative of $f$  \\ & in direction $\eta$ \\
$\Hess f(G)$ & Hessian of a twice differentiable \\ & function $f:\G \rightarrow \R$ \\
$\D_{i} f$ & differential resp.~ \\ & $i$-th component of $f$ \\
$\D_{G}$ & differential of an expression resp.~$G$  \\
$[n]:=\{1,\dots,n\}$ & set of integer numbers from $1$ to $n$ \\
$\eta, \chi, \xi$ & tangent vectors\\
$x_{i:j}$ & $i$-th to $j$-th component of $x$\\
$A_{i:j,k:l}$ & block matrix with rows from $i$ to $j$ \\
& and columns from $k$ to $l$ from $A$\\
$\eins_{n}$ & $n\times n$ identity matrix \\
$\lVert x  \rVert_{Q}^{2}:=\la x, Qx \ra$ & quadratic form regarding Q \\
$e_{i}^{n}$ & $i$-th unit vector in $\R^{n}$ \\
\end{tabular}
}

\begin{figure*}
\begingroup%
  \makeatletter%
  \providecommand\color[2][]{%
    \errmessage{(Inkscape) Color is used for the text in Inkscape, but the package 'color.sty' is not loaded}%
    \renewcommand\color[2][]{}%
  }%
  \providecommand\transparent[1]{%
    \errmessage{(Inkscape) Transparency is used (non-zero) for the text in Inkscape, but the package 'transparent.sty' is not loaded}%
    \renewcommand\transparent[1]{}%
  }%
  \providecommand\rotatebox[2]{#2}%
  \ifx\svgwidth\undefined%
    \setlength{\unitlength}{453.54330709bp}%
    \ifx\svgscale\undefined%
      \relax%
    \else%
      \setlength{\unitlength}{\unitlength * \real{\svgscale}}%
    \fi%
  \else%
    \setlength{\unitlength}{\svgwidth}%
  \fi%
  \global\let\svgwidth\undefined%
  \global\let\svgscale\undefined%
  \makeatother%
  \begin{picture}(1,0.375)%
    \put(0,0){\includegraphics[width=\unitlength,page=1]{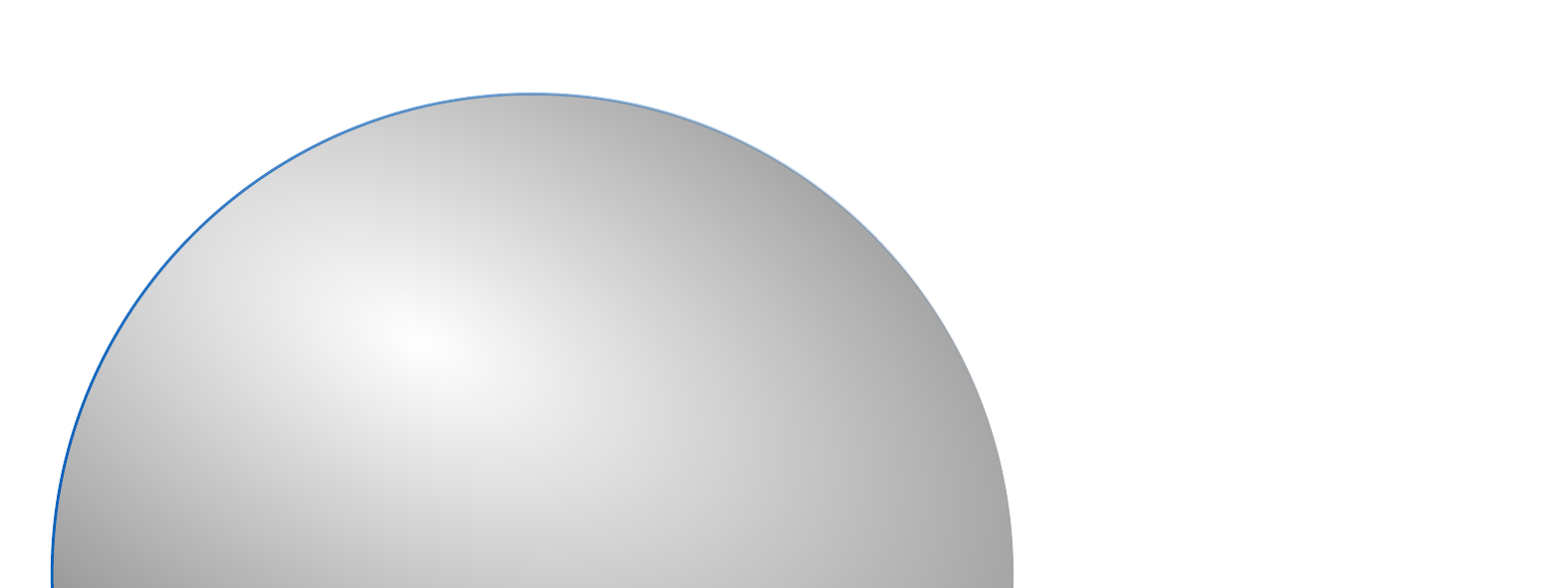}}%
    \put(0.32505947,0.29482141){\color[rgb]{0,0,0}\makebox(0,0)[lb]{\smash{}}}%
    \put(0.3293432,0.24631448){\color[rgb]{0,0,0}\makebox(0,0)[lb]{\smash{$\Id \in \G$}}}%
    \put(0,0){\includegraphics[width=\unitlength,page=2]{manifold.pdf}}%
    \put(0.52338966,0.14482384){\color[rgb]{0,0,0}\makebox(0,0)[lb]{\smash{$G = \Exp_{\g}(\eta)$}}}%
    \put(0.56466906,0.05021087){\color[rgb]{0,0,0}\makebox(0,0)[lb]{\smash{$\chi = T_{\Id}L_{G}\eta =:G\eta$}}}%
    \put(0.07483545,0.1104714){\color[rgb]{0,0,0}\makebox(0,0)[lb]{\smash{$H$}}}%
    \put(0.12732682,0.16816493){\color[rgb]{0,0,0}\makebox(0,0)[lb]{\smash{$\xi = T_{\Id}L_{H}\eta = T_{\Id}L_{H}T_{\Id}^{\ast}L_{G}\chi$}}}%
    \put(0.44634006,0.10601692){\color[rgb]{0,0,0}\makebox(0,0)[lb]{\smash{$T_{G}\G$}}}%
    \put(0.33426915,0.31262707){\color[rgb]{0,0,0}\makebox(0,0)[lb]{\smash{$\g = T_{\Id}\G=\se\times\R^{6}$}}}%
    \put(0.62897424,0.33230511){\color[rgb]{0,0,0}\makebox(0,0)[lb]{\smash{$\vecg$}}}%
    \put(0.62719242,0.25034247){\color[rgb]{0,0,0}\makebox(0,0)[lb]{\smash{$\matg$}}}%
    \put(0.85173586,0.2521134){\color[rgb]{0,0,0}\makebox(0,0)[lb]{\smash{$\R^{12}$}}}%
    \put(0.52749051,0.19142751){\color[rgb]{0,0,0}\makebox(0,0)[lb]{\smash{$\Exp_{\g}$}}}%
    \put(0.46806047,0.19087799){\color[rgb]{0,0,0}\makebox(0,0)[lb]{\smash{$\Log_{\g}$}}}%
    \put(0.17907056,0.05167213){\color[rgb]{0,0,0}\makebox(0,0)[lb]{\smash{$\G = \SE \times \R^{6}$}}}%
    \put(0.46567442,0.27691657){\color[rgb]{0,0,0}\makebox(0,0)[lb]{\smash{$\eta$}}}%
    \put(0.32883922,0.27718254){\color[rgb]{0,0,0}\makebox(0,0)[lb]{\smash{$\mathbf{0}\in \g$}}}%
  \end{picture}%
\endgroup%

\caption{Illustration of the Lie group $\G$ (represented as sphere) with its Lie algebra $\g$ and tangent spaces at different points. A tangent vector $\chi$ at a point $G$ can be expressed as tangent map at identity of the left translation at $G$ of a vector $\eta \in \g$, i.e. $\chi = G\eta$. Since the Lie algebra $\g$ can be identified by the $\vecg$ mapping with $\R^{12}$, we can express each tangent vector at a point $G$ as a pair $(G,\vecg(\eta))$. Each tangent vector on any tangent space may be mapped to the manifold using the exponential map $\Exp$.}
\end{figure*}
Moreover, we will employ the following concepts from differential geometry:

\paragraph{Riemannian metric on product Lie group.} On $\SE$ as submanifold of $\on{GL}_{4}$, the Riemannian metric at $E\in \SE$ for $\xi, \eta \in T_{E}\SE$ is given by $\la \xi, \eta \ra_{E}:=\la E^{-1}\xi, E^{-1}\eta\ra_{\eins_{4}}$ where $\la A,B \ra_{\eins_{4}}:=\tr(A^{\T}B)$ is the usual inner matrix product.
 
\paragraph{Riemannian Gradient.} For a real-valued function $f:\G \rightarrow \R$, the Riemannian gradient $\D f(G)$ is defined through the relation $\la \D f(G), \eta \ra_{G} := \D f(G)[\eta]$ for all $\eta \in T_{G}\G.$ For the product Lie group $\G = \SE \times \R^{6}$ and $G = (E,v)\in \G, \eta = (E\eta_{1}, \eta_{2}) \in T_{G}\G$ we calculate the Riemannian gradient as follows:
\begin{align*}
\D & f(G)[\eta] =  \la \D f(G), \eta \ra_{G} \\
= & \la  E^{-1} \D_{E} f((E,v)), \eta_{1} \ra_{\eins_{4}} + \la \D_{v} f((E,v)), \eta_{2} \ra\,,
\end{align*}
where  $\D_{E}f((E,v))$ is the partial Riemannian gradient on $\SE$ and $\D_{v} f((E,v))$ is the Euclidean partial gradient on $\R^{6}$.

\paragraph{Levi-Civita connection and connection function.}
For $G\in \G$ we denote by $\nabla$ the Levi-Civita connection of the Lie group $\G$ given through $\nabla: T_{G}\G \times T_{G}\G \rightarrow T_{G}\G$, with the properties {\em symmetry}, i.e. $[\eta,\chi] = \nabla_{\eta}\chi - \nabla_{\chi}\eta$, where $[\cdot, \cdot]$ denotes the Lie bracket, and {\em compatibility with the Riemannian metric.} The Levi-Civita connection is characterized by its connection function $\omega: \g \times \g \rightarrow \g$, $\omega(\xi,\eta):=\omega_{\xi}\eta:=\nabla_{\xi}\eta$ with the property $\nabla_{G\xi}G\eta = G\omega_{\xi}\eta$ for $\xi,\eta \in \g$.
 
\paragraph{Riemannian Hessian.} The Riemannian Hessian is defined through
$\la \Hess f(G)[\xi], \eta \ra := \D (\D f(G)[\xi])[\eta] - \D f(G) [\nabla_{\eta}\xi].$ On the product Lie group $\G = \SE \times \R^{6}$ which we consider in this paper, we set $G=(E,v)\in \G$ and $\xi = (E\xi_{1},\xi_{2})\in T_{G}\G,$ $\eta = (E\eta_{1},\eta_{2})\in T_{G}\G.$

\section{Minimum Energy Filtering Approach}
\label{sec:1}

\subsection{State Model with Constant Acceleration Assumption}
In the following we will denote by $E(t)\in \SE$ the time-dependent (external) camera parameter that can be expressed in terms of a rotation matrix $R(t) \in \SO$ and a translation vector $w(t) \in \R^{3}$ as a $4 \times 4$ matrix
\begin{equation}
E(t) = \bpm R(t) & w(t) \\ \mathbf{0}_{1\times 3} & 1\epm\,,
\end{equation}
for which we also use the shorthand $E(t) = (R(t), w(t)).$
Since the  \egomotion of a camera is generally not constant, the model $\dot{E}=0$ assumed in previous work  \cite{berger2015second} does not hold in real-world problems, where a camera fixed to a car rotates and accelerates in different directions. 
The \emph{constant acceleration} assumption, however, is more suited in this cases. 
It can be described by the second-order differential equation $\ddot{E}(t) = 0$ for all $t$ with initial pose $E(t_{0}) = E_{0}$ and velocity $\dot{E}(t_{0}) = V_{0}$.
In general, one can consider a polynomial model of even higher-order for $E(t)$. In the following, we will focus on the assumption that $E(t)$ is quadratic in $t$. We will comment on generalizations at the end of Section~\ref{sec:MEF-Derivation}. 

The equation $\ddot{E}(t) = 0$  can be prescribed as a system of first-order differential equations
\begin{align}
\begin{split}
\dot{E}(t) = &  V(t)\,,  \\
\dot{V}(t) = & 0\,,
\end{split}
\end{align}
where $V(t) \in T_{E(t)}\SE$ and $\dot{V}(t) \in T_{V(t)}T_{E(t)}\SE = T_{E(t)}\SE$. However, since the tangent bundle of a Lie Group can be expressed in terms of the product $T\SE \sim \SE \times \se$, we obtain a more compact expression, i.e.
\begin{align}
\label{eq:const-acc-ass-2states}
\begin{split}
\dot{E}(t) = & E(t) \matse(v(t)), \\
\dot{v}(t) = & \mathbf{0}_{6} \in \R^{6},
\end{split}
\end{align}
where the operator $\matse: \R^{6} \rightarrow \se$ is defined by
\begin{equation}
(\eta_{1}, \eta_{2}, \eta_{3}, \eta_{4}, \eta_{5}, \eta_{6})^{\T} \mapsto \left(\begin{smallmatrix} 
0 & -\tfrac{\eta_{3}}{\sqrt{2}} & \tfrac{\eta_{2}}{\sqrt{2}} & \eta_{4}\\
\tfrac{\eta_{3}}{\sqrt{2}} & 0 & -\tfrac{\eta_{1}}{\sqrt{2}} & \eta_{5}\\
-\tfrac{\eta_{2}}{\sqrt{2}} & \tfrac{\eta_{1}}{\sqrt{2}} & 0 & \eta_{6}\\
0 & 0 & 0 & 0
\end{smallmatrix}\right).
\end{equation}
The inverse operation is denoted by $\vecse: \se \rightarrow \R^{6}.$ Note that this operation is consistent with the usual scalar product, i.e. for $\chi, \eta \in \se$ it holds 
\begin{equation}
\la \chi ,\eta \ra_{\Id}:= \tr(\chi^{\T}\eta) = \la \vecse(\chi), \vecse(\eta) \ra.
\end{equation}
Since $\SE$ is a Lie Groups regarding the matrix multiplication and $\R^{6}$ is a Lie Group regarding addition, we can understand the system \eqref{eq:const-acc-ass-2states} as a first-order differential equation on a product Lie Group 
\begin{equation}
\G := \SE \times \R^{6}.
\end{equation}
For two elements $G_{1}=(E_{1},v_{1}),G_{2} = (E_{2},v_{2})\in \G$ we define the left translation $L_{G_{1}}$ by $L_{G_{1}}G_{2}:=(E_{1}E_{2},v_{1}+v_{2}) \in \G.$  Since the tangent bundle $T\R^{6}$ can be identified with $\R^{6}$, we obtain the Lie algebra 
\begin{equation}
\g = \se \times \R^{6}.
\end{equation}
In turn, we can take down \eqref{eq:const-acc-ass-2states} compactly as
\begin{equation}\label{eq:const-acc-ass-G}
\dot{G}(t) = (E(t)\matse(v(t)), \mathbf{0}_{6})\,,
\end{equation}
where $E$ and $v$ will denote the first and second element of $G \in \G$, respectively.
On matrix Lie groups, one can express kinematics directly as matrix multiplication (cf.~\cite{zamani2012}), i.e.\ $\dot{E} = E\Gamma$ for $\Gamma\in \se, E \in \SE,$ which is not valid for general Lie groups. The rigorous way to describe kinematics is to use the {\em tangent map} (cf.~\cite{saccon2013second}) of the left translation which is given by the following proposition:
\begin{proposition} \label{prop1}
The tangent map of the left translation regarding $G=(E,v) \in \G$ at identity, i.e. $T_{\Id}L_{G}: \g \rightarrow T_{G}\G$, can be computed for  $\eta = (\eta_{1},\eta_{2})\in \g$ as
\begin{equation}
T_{\Id}L_{G}\eta = (E\eta_{1}, \eta_{2}) = L_{(E,0)}\eta =: G \eta.
\end{equation}
\end{proposition}

With Proposition~\ref{prop1} we can write down \eqref{eq:const-acc-ass-G} as
\begin{equation}\label{eq:const-acc-ass-dotG}
\dot{G}(t) = T_{\Id}L_{G(t)} f(G(t)) = G(t) f(G(t)),
\end{equation}
where $f: \G \rightarrow \g$ is given by 
\begin{equation}\label{eq:const-acc-ass-def-f}
 f(G)=f((E,v))= (\matse(v), \mathbf{0}_{6}).
\end{equation}

\begin{remark}
During the further development, the notation $G\eta$ for a Lie group element $G \in \G$ and $\eta \in \g$ must always be understood as the tangent map of the left translation at identity. Similarly, we denote $G^{-1}\eta:=T_{\Id}L_{G}^{\ast}\eta$ for the dual of the tangent map of $L_{G}$  at identity. 
\end{remark}

\subsection{Optical Flow Induced by Ego-Motion}

The optical flow $u: \Omega \times T \rightarrow \R^{2}$ on an image sequence $\{I(t),t \in T\}$ can be computed in terms of the underlying scene structure as given by a depth map $d: \Omega \times T$ and the camera motion $E:T \rightarrow \SE,$ i.e.\ $E(t)= (R(t),w(t))$,  
where $R(t)$ and $w(t)$ denote the camera rotation and translation, respectively, by the following relation:
\begin{align}
\begin{split}
u (x, t; d(x,t), & (R(t),w(t)))  \\
 = &  \pi( R(t)^{\T}( \left(\begin{smallmatrix} x \\ 1 \end{smallmatrix}\right) d(x,t) - w(t)) ) - x,
 \end{split} \label{eq:flow_Rv}
\end{align}
whereas $\pi:\R^{3} \rightarrow \R^{2}$ is the projection $(x_{1},x_{2},x_{3})^{\T} \mapsto x_{3}^{-1}(x_{1},x_{2})^{\T}$ as depicted in Fig.~\ref{fig:camera-model}. Note that $x \in \R^{3}$ indicates
{\em inhomogenous} coordinates rather than {\em homogenous} coordinates on the projective space.

We can also express \eqref{eq:flow_Rv} directly in terms of $E(t)$. By adding the superscript $k$ for (discretized) pixels $x^{k}\in \Omega$, we obtain
\begin{equation}\label{eq:flow_E}
u(x^{k},t; d(x,t),E(t)) + x^{k} = \pi((E^{-1}(t)g_{k}(t))_{1:3}),
\end{equation}
where $g_{k}(t) := (d(x^{k},t) (x^{k})^{\T},  d(x^{k},t), 1)^{\T}$ 
denotes the data vector containing depth information of pixel~$x^{k}$ below. 

\begin{remark}
In the equation \eqref{eq:flow_E} we assumed a {\em static scene}, since we set the scene point $X$ constant in time.
\end{remark}

\begin{figure}[t]
\begin{center}
\begingroup%
  \makeatletter%
  \providecommand\color[2][]{%
    \errmessage{(Inkscape) Color is used for the text in Inkscape, but the package 'color.sty' is not loaded}%
    \renewcommand\color[2][]{}%
  }%
  \providecommand\transparent[1]{%
    \errmessage{(Inkscape) Transparency is used (non-zero) for the text in Inkscape, but the package 'transparent.sty' is not loaded}%
    \renewcommand\transparent[1]{}%
  }%
  \providecommand\rotatebox[2]{#2}%
  \ifx\svgwidth\undefined%
    \setlength{\unitlength}{226.77165354bp}%
    \ifx\svgscale\undefined%
      \relax%
    \else%
      \setlength{\unitlength}{\unitlength * \real{\svgscale}}%
    \fi%
  \else%
    \setlength{\unitlength}{\svgwidth}%
  \fi%
  \global\let\svgwidth\undefined%
  \global\let\svgscale\undefined%
  \makeatother%
  \begin{picture}(1,1)%
    \put(0,0){\includegraphics[width=\unitlength,page=1]{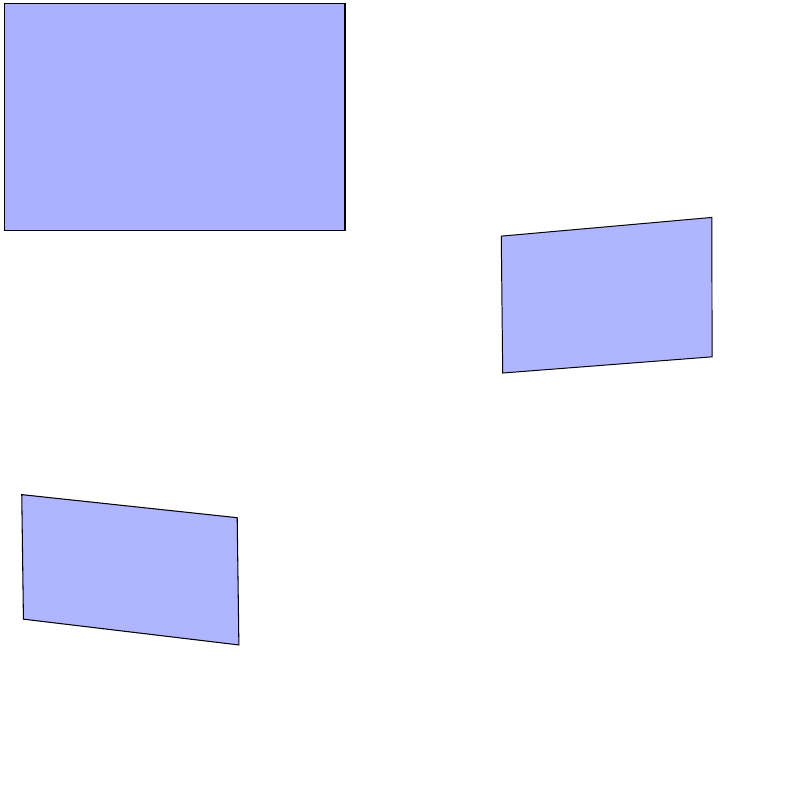}}%
    \put(0.57877297,0.92924718){\color[rgb]{0,0,0}\makebox(0,0)[lb]{\smash{$X=d(x(t))x(t)$}}}%
    \put(0,0){\includegraphics[width=\unitlength,page=2]{camera-model.pdf}}%
    \put(0.0604762,0.07811524){\color[rgb]{0,0,0}\makebox(0,0)[lb]{\smash{$E(t-1) = (I_3,\mathbf{0})$}}}%
    \put(0,0){\includegraphics[width=\unitlength,page=3]{camera-model.pdf}}%
    \put(0.44021244,0.25987137){\color[rgb]{0,0,0}\makebox(0,0)[lb]{\smash{$w(t)$}}}%
    \put(0.69845904,0.380212){\color[rgb]{0,0,0}\makebox(0,0)[lb]{\smash{$E(t) = (R(t),w(t))$}}}%
    \put(0.16945932,0.22174601){\color[rgb]{0,0,0}\makebox(0,0)[lb]{\smash{$x(t)$}}}%
    \put(0,0){\includegraphics[width=\unitlength,page=4]{camera-model.pdf}}%
    \put(0.6795404,0.6188313){\color[rgb]{0,0,0}\makebox(0,0)[lb]{\smash{$\tilde{x}(t+1)$}}}%
    \put(0,0){\includegraphics[width=\unitlength,page=5]{camera-model.pdf}}%
    \put(0.24355993,0.82833707){\color[rgb]{0,0,0}\makebox(0,0)[lb]{\smash{$x(t)$}}}%
    \put(0,0){\includegraphics[width=\unitlength,page=6]{camera-model.pdf}}%
    \put(0.07358641,0.74172439){\color[rgb]{0,0,0}\makebox(0,0)[lb]{\smash{$x(t+1)$}}}%
    \put(0.05669643,0.82163209){\color[rgb]{0,0,0}\makebox(0,0)[lb]{\smash{${\color{red}u(x(t),t)}$}}}%
    \put(0,0){\includegraphics[width=\unitlength,page=7]{camera-model.pdf}}%
    \put(0.56444444,1.16416673){\color[rgb]{0,0,0}\makebox(0,0)[lb]{\smash{}}}%
  \end{picture}%
\endgroup%
\caption{Camera model for the monocular approach: A static scene point $X$ is projected onto the plane at $x(t)$ of the first camera $E(t-1)$ which is mounted at the origin with rotation $\eins_{3}$ such that $X = d(x(t))x(t)$. By moving the camera into position $E(t)=(R(t),w(t))$, the scene point is projected onto $\pi(R^{\T}(t) (X - w(t) )) = x(t+1)$ 
which is at the same (relative) image position as $\tilde{x}(t+1)$ on the second image plane. The induced optical flow is given by the difference $u(x(t),t) = x(t+1)-x(t)$.
}
\label{fig:camera-model}
\end{center}
\end{figure}

\section{Minimum Energy Filter Derivation}
\label{sec:MEF-Derivation}

In this section, we will determine the problem of camera motion estimation with filtering equations, and we will summarize the most important steps for the derivation of the minimum energy filter.

By denoting the left hand side of \eqref{eq:flow_E} by $y_{k} \in \R^{2}$ which is the observation, i.e.
 \begin{equation}\label{eq:def-yk}
 y_{k}(t):= u(x^{k},t; d(x,t),E(t)) + x^{k} \,,
 \end{equation}
 and defining 
\begin{equation} \label{eq:def-hk}
h_{k}(E,t):=\pi((E^{-1}(t)g_{k}(t))_{1:3})
\end{equation}
as the right hand side of~\eqref{eq:flow_E}, together with~\eqref{eq:const-acc-ass-2states} and~\eqref{eq:const-acc-ass-dotG}, we obtain the following state- and observation system by setting $G  = (E,v) \in \G$:

\begin{align}
\dot{G}(t) = & G(t) (f(G(t)) + \delta(t)), G(t_{0}) = G_{0}\,, \,\,\,\mbox{(state)} \label{eq:state-eq}\\
y_{k} = & h_{k}(E(t),t) + \epsilon_{k}(t), k \in [n]\,, \,\,\mbox{(observation)} \label{eq:observation-eq}
\end{align}
where $f(G)$ is defined as in \eqref{eq:const-acc-ass-def-f} and $n$ denotes a (fixed) number of specific image pixels. The functions $\delta: T \rightarrow \mf g$ and $\epsilon_{k}: T \rightarrow \R^{2}$, $k\in [n]$ are noise processes that model deviations from state and observations, respectively. Here, $T$ denotes a continuous time interval, e.g.\ $T = \R_{\geq 0}$.

\subsection{Energy Function}

Given a depth map, which is contained in the function $g_{k}(t)$ in \eqref{eq:def-hk} and the optical flow $u_{k}$ in terms of the observations $y_{k}$ in \eqref{eq:def-yk}, we want to find the camera motion and its velocity in terms of $G(t) \in \G$ such that the observation error $\epsilon_{k}$ in \eqref{eq:observation-eq} is minimal and such that \eqref{eq:state-eq} is fulfilled with minimal deviations $\delta(t)$ for all $t \in T.$ 

To this end, we consider the penalization of $\delta = (\delta_{1},\delta_{2})\in \g$ and $\epsilon = \{\epsilon_{k}\}_{k=1}^n$ by a quadratic function $c:\g \times \R^{2n}\times T \times T \rightarrow \R$ given as
\begin{equation}\label{eq:penalty-function_no_decay}
\begin{split}
c(\delta, \epsilon, \tau, t):= &\tfrac{1}{2}\Bigl( \lVert \vecse(\delta_{1}(\tau)) \rVert_{S_{1}}^{2}  \\
& \quad + \lVert \delta_{2}(\tau) \rVert_{S_{2}}^{2} + \sum_{k=1}^{n} \lVert \epsilon_{k}(\tau)\rVert_{Q}^{2} \Bigr),
\end{split}
\end{equation}
where $S_{1},S_{2} \in \R^{6\times 6}$ and $Q\in \R^{2\times 2}$ are symmetric, positive definite weighting matrices. 
From \cite{saccon2013second} we adopt the idea of a decay rate $\alpha>0$, 
and thus we introduce the weighting factor $e^{-\alpha(t-t_{0})}$ on the right-hand side of~\eqref{eq:penalty-function_no_decay}:
\begin{align}\label{eq:penalty-function}
\begin{split}
c(\delta, \epsilon, \tau, t):= &\tfrac{1}{2}e^{-\alpha(t-t_{0})}\Bigl( \lVert \vecse(\delta_{1}(\tau)) \rVert_{S_{1}}^{2}  \\
& \quad + \lVert \delta_{2}(\tau) \rVert_{S_{2}}^{2} + \sum_{k=1}^{n} \lVert \epsilon_{k}(\tau)\rVert_{Q}^{2} \Bigr).
\end{split}
\end{align}
Based on the penalty function \eqref{eq:penalty-function}, we define the energy:
\begin{equation}\label{eq:energy-function}
\mc{J}(\delta, \epsilon ,t_{0},t) := m_{0}(G(t),t,t_{0}) + \int_{t_{0}}^{t} c(\delta, \epsilon, \tau, t) \, d\tau\,,
\end{equation}
where $m_{0}$ is a quadratic penalty function for the initial state. For our model we set
\begin{equation}
m_{0}(G,t,t_{0}) := \tfrac{1}{2} e^{-\alpha(t-t_{0})} \la G-\Id, G-\Id \ra_{\Id},
\end{equation}
where the difference is canonical, i.e. $G-\Id = (E-\eins_{4},v)$ for $G=(E,v).$

\begin{remark}
Instead of using two quadratic forms with matrices $S_{1}, S_{2}$, we can use more generally a symmetric and positive weighting matrix $S \in \R^{12\times 12}$ if we want to couple $\delta_{1}$ and $\delta_{2}$. In  the upper case we find that $S = \bsm S_{1} & 0 \\ 0 & S_{2} \esm.$
\end{remark}

\subsection{Optimal Control Problem}
The optimal control theory allows us to determine the optimal control input $\delta:T \rightarrow \mf g$ that minimizes the energy 
$\mc{J}(\delta,\epsilon(G(t),t),t_{0},t)$
for each $t\in T$ subject to the state constraints \eqref{eq:state-eq}. To be precise, we want to find for all $t \in T$ and fixed $G(t)$ the control input $\delta \vert_{[t_{0},t]}$ defining
\begin{equation}\label{eq:acc-const-val-fun}
\mc{V}(G(t),t):= \min_{\delta \vert_{[t_{0},t]}} \mc{J}(\delta,\epsilon(G(t),t),t_{0},t), \,\mbox{s.t. \eqref{eq:state-eq}\,.} 
\end{equation}
The optimal trajectory is 
\begin{equation}\label{eq:optimality_cond}
{G^{\ast}}(t) := \argmin_{G(t) \in \G} \mathcal{V}(G(t),t)\,,
\end{equation}
 for all $t \in T$ and $\mathcal{V}(G,t_{0}) = m_{0}(G_{0},t_{0},t_{0}).$ This problem is a classical optimal control problem, for which the standard Hamilton-Jacobi theory \cite{jurdjevic1997,athans1966optimal} under appropriate conditions results in the well-known {\em Hamilton-Jacobi-Bellman} equation. Pontryagin \cite{athans1966optimal} proved that the minimization of the Hamiltonian provides a solution to the corresponding optimal control problem ({\em Pontryagin's Minimum Principle}).

However, since $\mc G$ is a non-compact Riemannian manifold, we cannot apply the classical Hamilton-Jacobi theory for real-valued problems (cf.~\cite{athans1966optimal}).
Instead we follow the approach of Saccon {\em et al.}~\cite{saccon2013second} who derived a {\em left-trivialized optimal Hamiltonian} based on control theory on Lie groups \cite{jurdjevic1997}. 
This left-trivialized optimal Hamiltonian is defined by $\tilde{\mathcal H}^{-}: \mc G \times \mf g \times \mf g \times T \rightarrow \R,$
\begin{equation}\label{eq:preHamiltonian}
\tilde{\mathcal{H}}^{-}(G,\mu, \delta,  t) := c(\delta, \epsilon(G,t), t_0,t) - \la \mu, F(G(t))+\delta \ra_{\Id} .
\end{equation}
The minimization of \eqref{eq:preHamiltonian} w.r.t.\ the variable $\delta = (\delta_{1},\delta_{2})$ leads \cite[Proposition~4.2]{saccon2013second} to the optimal Hamiltonian 
\begin{equation}\label{eq:optimal-hamiltonian}
\mathcal{H}^{-}(G,\mu,t):= \tilde{\mathcal{H}}^{-}(G,\mu,\delta^{\ast},t)\,,
\end{equation}
where $\delta^{\ast}=(\delta_{1}^\ast,\delta_{2}^\ast)$ is given by 
\begin{equation}
\begin{aligned}
  \vecse(\delta_{1}^{\ast}) &= e^{\alpha(t-\tau)} S_{1}^{-1} \vecse(\mu_{1}),\text{ and}\\ 
  \delta_{2}^{\ast} &= e^{\alpha(t-t_{0})} S_{2}^{-1} \mu_{2}\,.
\end{aligned}
\end{equation}
Examining the right-hand side of \eqref{eq:optimal-hamiltonian} in detail, we obtain
\begin{align}
\mathcal{H}^{-}& ((E,V), \mu, t) =  \tfrac{1}{2}e^{-\alpha(t-t_{0})} \Bigl( \sum_{k=1}^{n} \lVert y_{k}-h_{k}(E)\rVert_{Q}^{2} \Bigr) \nonumber \\
& - \tfrac{1}{2}e^{\alpha(t-t_{0})}\Bigl( \la  \mu_{1}, \matse(S_{1}^{-1}\vecse(\mu_{1})) \ra_{\Id}  \label{eq:left-trivial-hamilton}\\
& + \la \mu_{2}, S_{2}^{-1} \mu_{2} \ra\Bigr) - \la \mu_{1}, \matse(V)\ra_{\Id} \,,\nonumber
\end{align}
where we used $\epsilon(G(t),t)=\{y_{k}-h_{k}(E(t),t)\}_{k=1}^n$. Here we introduced on the left hand side the variable $G$ since the right hand side depends on $G = (E,v)$. 

In the next section, we will compute explicit ordinary differential equations regarding the optimal state ${E^{\ast}}(t)$ for each $t\in T$ that consists of different derivatives of the left trivialized Hamilton function \eqref{eq:left-trivial-hamilton}.

\subsection{Recursive Filtering Principle by Mortensen}

In order to find a recursive filter, we compute the total time derivative of the optimality condition on the value function, which is
\begin{equation}
\D_{1} \mathcal V({G^{\ast}},t) = 0\,,
\end{equation}
 for each $t \in T$. This equation must be fulfilled by an optimal solution ${G^{\ast}} \in \G$ of the filtering problem. Unfortunately, because the filtering problem is in general infinite dimensional, this leads to an expression containing derivatives of every order. In practice (cf. \cite{zamani2012,saccon2013second}), derivatives of third order and higher are neglected, since they require tensor calculus. Omitting these leads to a second-order approximation of the optimal filter. The following theorem is an adaption of \cite[Theorem~4.1]{saccon2013second}:
 
 \begin{theorem}\label{thm:minEnFilter_Acc}
The differential equations of the second-order Minimum Energy Filter for state~\eqref{eq:state-eq} and nonlinear observer model~\eqref{eq:observation-eq} are given by
\begin{align}
&(G^{\ast})^{-1}\dot{G}^{\ast} =   \bigl(f(G^{\ast}) - \matg(P(t)\vecg(r_{t}(G^{\ast})))\bigr), \nonumber \\
& G^{\ast}(t_{0}) =  \Id \,, \label{eq:optimalStateG} \\
 \begin{split}
& \dot{P}(t) =   -\alpha \cdot P + S^{-1} + C P + P C^{\T}\\
& - P 
\begin{pmatrix}
\sum_{k=1}^{n} (\tilde{\Gamma}_{\vecse(\on{Pr}(A_{k}(E^{\ast})))} + D_{k}(E^{\ast})) & \mathbf{0}_{6\times 6} \\
\mathbf{0}_{6\times 6} & \mathbf{0}_{6\times 6}
\end{pmatrix} P \,, \\
& P(t_{0})  =  \eins_{12}\,,
\end{split}\label{eq:optimalStateP}
\end{align}
where 
$r_{t}(G^{\ast}) :=\bigl(\sum_{k=1}^{n} \on{Pr}(A_{k}(E^{\ast}), \mathbf{0}_{6}\bigr) $
and 
\begin{align*} \label{eq:defA}
C(G^{\ast},t) := & \begin{pmatrix}
-\Psi(G^{\ast},t) & \eins_{6} \\
 \mathbf{0}_{6\times 6} & \mathbf{0}_{6\times 6}
 \end{pmatrix}\,,
 \end{align*}
 with 
 \begin{align}
 \begin{split}
 \Psi(G^{\ast},t):=&  \ad_{\se}^{\on{vec}} (f(G^{\ast})) + \tilde{\Gamma}^{\ast}_{\vecse(P\, r_{t}(G^{\ast}))} \,.
 \end{split}
 \end{align}
 The function $A_{k}:\SE \rightarrow \R^{4\times 4}$ is given by
 \begin{align}\label{eq:def-Ak}
 \begin{split}
 A_{k}(E)& = A_{k}(E,g_{k}):= 
\bigl(\kappa_{k}^{-1}\hat{I}  -  \kappa_{k}^{-2}  \hat{I}   E^{-1}  e_{3}^{4} g_{k}^\T   \hat{I} \bigr)^{\T} \\
& \quad \cdot  Q (y_{k} - h_{k}(E))  g_{k}^{\T}  E^{-\T},
 \end{split}
 \end{align}
 where $\kappa_{k}:=\kappa_{k}(E):= (e_{3}^{4})^{\T}E^{-1}g_{k}.$ The second-order operator $D_{k}:\SE \rightarrow \R^{6\times 6}$ is given by \eqref{eq:def_Dk}, see Appendix~\ref{appendix:proofs}.
 
The matrix valued functions $\tilde{\Gamma}_{z},\tilde{\Gamma}_{z}^{\ast}: \R^{6} \rightarrow  \R^{6\times 6}$ are obtained from the vectorization of the connection functions. Their components are given by $(\tilde \Gamma_{z})_{ij} := \sum_{k=1}^{6}\Gamma_{jk}^{i} z^{k}$ and $(\tilde \Gamma_{z}^{\ast})_{ik} := \sum_{j=1}^{6}\Gamma_{jk}^{i} z^{j}$ with $z \in \R^{6}$ and the Christoffel-Symbols $\Gamma_{jk}^{i}$ are given in Appendix~\ref{app:christoffel}.
\end{theorem}
This theorem will be proven at the end of the section.

\begin{remark}
A generalization of this theorem is published in Saccon {\em et al.} \cite{saccon2013second} for a larger class of filtering problems. However, the application of the theorem is not straightforward since the appearing expressions, e.g.\ exponential functor, cannot be evaluated directly. Furthermore, the adaption to nonlinear filtering problems has not been considered in the literature yet. Besides, we show how to find explicit expressions in terms of matrices for the general operators in \cite{saccon2013second}.
\end{remark}

In our previous work~\cite{berger2015second} we presented a theory regarding the case of constant \emph{velocity}. This theory can be derived directly from Theorem~\ref{thm:minEnFilter_Acc}
by neglecting the velocity $v$ i.e.\ the second component of $G=(E,v)\in \G$ (thus changing from Lie group $\SE\times\R^6 $ to $\SE$) and by setting $f(G)\equiv 0$.
In this case, the state and observation equations are reduced to
\begin{align}
\dot{E}(t) = & E(t)\delta(t), \, E(t_{0}) = E_{0}, \quad  \mbox{(state)}\,, \label{eq:state-eq-first}\\
y_{k} = & h_{k}(E(t),t) + \epsilon_{k}(t), \, k \in [n].  \quad  \mbox{(observation)} \label{eq:observation-eq-first}.
\end{align}
For the reader's convenience, we state the theory under the assumption of constant \emph{velocity} as a corollary:

\begin{corollary}\label{cor:first-order-filter}
The differential equations of the second-order Minimum Energy Filter for our state \eqref{eq:state-eq-first} and nonlinear observer model \eqref{eq:observation-eq-first} are given by
\begin{align}
& (E^{\ast})^{-1} \dot{E}^{\ast} = - \matse(P(t)\vecse(\sum_{k=1}^{n} \on{Pr}(A_{k}(E^{\ast})))),  \nonumber \\
& E^{\ast}(t_{0}) =  \Id\,, \label{eq:optimalStateE}\\
\begin{split}
&\dot{P}(t) =  -\alpha \cdot P + S_{1}^{-1}  \\
& -  \tilde{\Gamma}^{\ast}_{\vecse((E^{\ast})^{-1} \dot{E}^{\ast} )}  P - P( \tilde{\Gamma}^{\ast}_{\vecse((E^{\ast})^{-1} \dot{E}^{\ast} )} )^{\T} \\
& - P 
\bigl(\sum_{k=1}^{n} (\tilde{\Gamma}_{\vecse(\on{Pr}(A_{k}(E^{\ast})))} + D_{k}(E^{\ast})) \bigr)P,\\
& P(t_{0}) =  \eins_{6}.
\end{split}\label{eq:optimalStateP1}
\end{align}
\end{corollary}

\begin{remark}
We compare the computational complexity for the cases of constant \emph{velocity} and constant \emph{acceleration}.
By considering the difference between Theorem \ref{thm:minEnFilter_Acc} and Corollary \ref{cor:first-order-filter}, we see that the only differences are a larger state space and the occurrence of the additional operator
$f(G^{\ast})$ in \eqref{eq:defA}. However, this does not change the computational effort significantly. Thus, we suggest to use the second-order minimum energy filter since it is more robust but computational only slightly more complex as we will see in the experiments.
\end{remark}

Before we will turn to proving Theorem \ref{thm:minEnFilter_Acc}, we first provide some lemmas that are based on the general approach of \cite{saccon2013second}. However, we cannot use the main result of \cite{saccon2013second} directly, since the appearing general operators are complicated to evaluate. Instead, we provide the corresponding expressions in such a way that they can be easily implemented. Thus, following \cite[Eq. (37)]{saccon2013second} the estimate of the optimal state $G^{\ast}$ is given by
\begin{equation}\label{eq:ode-optimal-state}
\begin{split}
(G^{\ast})^{-1}&\dot{G}^{\ast} = -\D_{2} \mc H^{-}(G^{\ast},0,t) \\
& - Z(G^{\ast},t)^{-1} \circ (G^{\ast})^{-1} \D_{1} \mathcal H^{-}(G^{\ast}, 0, t)\,.
\end{split}
\end{equation}
This expression contains the second-order information matrix $Z(G,t): \g \rightarrow \g$ of the value function $\mc{V}$ as defined in \eqref{eq:acc-const-val-fun}, defined through
\begin{equation}\label{eq:defZ}
Z(G,t) \circ\eta = G^{-1} \circ \Hess_{1} V(G,t)[ G \eta ] \,.
\end{equation}
 An explicit expression for the gradient of the Hamiltonian in \eqref{eq:ode-optimal-state} is provided in the following lemma:

\begin{lemma} \label{lemma1}
The Riemannian gradient $\D_{1} \mc H^{-}(G,\mu,t)$ on $T_{G}\G$ for $G = (E,v)$ can be calculated as
\begin{align}
\begin{split} \label{eq:grad1H}
\D_{1} & \mc H^{-}(G,\mu,t)  \\
= & G \Bigl(e^{-\alpha(t-t_{0})} \sum_{k=1}^{n} \on{Pr}(A_{k}(E)), -\vecse(\mu_{1})\Bigr), 
\end{split}
\end{align}
where the function $A_{k}(E)=A_{k}(E,g_{k}):\SE \times \R^{4} \rightarrow \on{GL}_{4}$ is defined in \eqref{eq:def-Ak}.
\end{lemma}
By insertion of \eqref{eq:grad1H} in \eqref{eq:ode-optimal-state} and usage of the definition of $r_{t}(G^{\ast})$ from Theorem~\ref{thm:minEnFilter_Acc} we obtain
\begin{align}\label{eq:ode-optimal-state1}
\begin{split}
& (G^{\ast})^{-1} \dot{G}^{\ast}  =  -\D_{2} \mc H^{-}(G^{\ast},0,t) \\
&  \quad - e^{-\alpha(t-t_{0})} Z(G^{\ast},t)^{-1} \circ r_{t}(G^{\ast}).
\end{split}
\end{align}

Following the calculus in \cite{saccon2013second}, the evolution equation for the trivialized Hessian $Z(G,t): \mf g \rightarrow \mf g^{\ast}$ is given by
\begin{equation}\label{eq:Zall}
\begin{aligned}
\frac{d}{dt} & Z( G^{\ast}(t),t)\\ 
\approx &  Z(G^{\ast},t) \circ \omega_{(G^{\ast})^{-1} \dot{G^{\ast}}} \\
& + Z(G^{\ast},t) \circ  \omega_{\D_{2}\mc H^{-}(G^{\ast},0,t)}^{\leftrightharpoons}\\
& +  \omega_{(G^{\ast})^{-1}\dot{G^{\ast}}}^{\ast} \circ Z(G^{\ast},t) \\
& + \omega_{\D_{2}\mc H^{-}(G^{\ast},0,t)}^{\leftrightharpoons \ast}\circ Z(G^{\ast},t) \\
& +T_{\on{Id}} L_{G^{\ast}}^{\ast} \circ \Hess_{1} \mc H^{-}(G^{\ast}, 0, t) \circ T_{\on{Id}}L_{G^{\ast}} \\
& + T_{\on{Id}} L_{G^{\ast}}^{\ast} \circ \D_{2}(\D_{1} \mc H^{-})(G^{\ast},0,t) \circ Z(G^{\ast},t) \\
& + Z(G^{\ast},t) \circ \D_{1} (\D_{2} \mc H^{-})(G^{\ast},0,t) \circ T_{\on{Id}}L_{G^{\ast}} \\
& + Z(G^{\ast},t) \circ  \Hess_{2} \mc H^{-}(G^{\ast},0,t) \circ Z(G^{\ast},t)\,  . 
\end{aligned}
\end{equation}
(cf. \cite[Eq. (51)]{saccon2013second}).

The ``swap''-operators $\omega_{\cdot}^{\leftrightharpoons}\cdot, \omega_{\cdot}^{\leftrightharpoons \ast}\cdot$ in this expression are defined in Section~\ref{sec:notation}, i.e.~$\omega^{\leftrightharpoons}_{\eta}\xi:= \omega_{\xi}\eta$ and $\la \omega_{\eta}^{\leftrightharpoons \ast}\xi, \chi \ra_{\Id} := \la \xi, \omega_{\eta}^{\leftrightharpoons} \chi \ra_{\Id} = \la \xi, \omega_{\chi}\eta \ra_{\Id}$.
By considering the standard basis of $\g$, there exists a matrix representation $K\in \R^{12\times12}$, such that for all $\eta = (\eta_{1}, \eta_{2}) \in \mf g$ we receive
\begin{equation} \label{eq:vectorization-Z}
\vecg(Z(G^{\ast},t) \circ \eta) = K(t) \vecg(\eta)\,.
\end{equation} 
Similarly to \cite{berger2015second} we need to evaluate the right-hand side of the evolution equation at $\eta \in \g$ and to vectorize it. The single expressions are shown in the following lemma.

 \begin{lemma}[Matrix representations of $Z$] \label{lemma2} \\ Let $Z({G^{\ast}},t):\g \rightarrow \g$ be the operator \eqref{eq:defZ}. Then there exists a matrix $K = K(t)\in \R^{12\times 12}$ yielding
 \begin{equation}\label{eq:vectorization_Z}
 \vecg(Z({G^{\ast}},t)(\eta)) = K(t) \vecg(\eta),
 \end{equation}
 and thus 
 \begin{align}
 \vecg(d/dt Z({G^{\ast}},t)(\eta)) = &  \dot{K}(t)\vecg(\eta)\, , \label{eq:vec-temp-der-Z} \\
 \vecg(Z^{-1}(G^{\ast},t)(\eta)) = &  K^{-1}(t)\vecg(\eta) \label{eq:vec-inv-Z}\, ,
 \end{align}
  as well as
\begin{enumerate}
	\item $\vecg(Z(G^{\ast},t) \circ \omega_{(G^{\ast})^{-1} \dot{G^{\ast}}}\eta  \\
	 + Z(G^{\ast},t) \circ  \omega_{\D_{2}\mc H^{-}(G^{\ast},0,t)}^{\leftrightharpoons}\eta) = K(t)B \vecg(\eta)$
	\item $\vecg(  \omega_{(G^{\ast})^{-1}\dot{G^{\ast}}}^{\ast} \circ Z(G^{\ast},t)\circ \eta  \\
	 + \omega_{\D_{2}\mc H^{-}(G^{\ast},0,t)}^{\leftrightharpoons \ast}\circ Z(G^{\ast},t)\circ \eta)= B^{\T}K(t) \vecg(\eta)$
	\item $\vecg(T_{\on{Id}} L_{G^{\ast}}^{\ast} \circ \Hess_{1} \mc H^{-}(G^{\ast}, 0, t)[T_{\on{Id}}L_{G^{\ast}}\eta]) =  e^{-\alpha(t-t_{0})} \\ \cdot \begin{pmatrix}
\sum_{k=1}^{n} (\tilde{\Gamma}_{\vecse(\on{Pr}(A_{k}(E)))} + D_{k}(E)) & \mathbf{0}_{6\times 6} \\
\mathbf{0}_{6\times 6} & \mathbf{0}_{6\times 6}
\end{pmatrix} \vecg(\eta)$
	\item $\vecg(Z(G^{\ast},t) \circ \D_{1} (\D_{2} \mc H^{-})(G^{\ast},0,t) \circ T_{\on{Id}}L_{G^{\ast}}\eta) \\ = -K(t) \begin{pmatrix}  
\mathbf{0}_{6\times 6} & \eins_{6} \\
\mathbf{0}_{6 \times 6} & \mathbf{0}_{6\times 6} 
\end{pmatrix}\vecg(\eta)$
	\item $\vecg(T_{\on{Id}} L_{G^{\ast}}^{\ast} \circ \D_{2}(\D_{1} \mc H^{-})(G^{\ast},0,t) \circ Z(G^{\ast},t)\circ \eta) \\  = -\begin{pmatrix}  
\mathbf{0}_{6\times 6} & \mathbf{0}_{6 \times 6}\\
 \eins_{6} &  \mathbf{0}_{6\times 6}
\end{pmatrix} K(t) \vecg(\eta)$
		\item $\vecg( Z({G^{\ast}},t)  (\Hess_{2}  \mathcal H^{-}({G^{\ast}},0,t)[ Z({G^{\ast}},t) (\eta)])) \\ =  - e^{\alpha(t-t_{0})} K(t)S^{-1}K(t) \vecg(\eta)\, ,$ \label{lem2-3}
 \end{enumerate}
 with $\tilde \Gamma_{\cdot}$, $\tilde \Gamma_{\cdot}^{\ast},$  and functions $A_{k},D_{k}$ from Theorem~\ref{thm:minEnFilter_Acc} and
\begin{equation}\label{eq:defB}
B := \bsm \Psi(G^{\ast},t) & \mathbf{0}_{6\times 6} \\
 \mathbf{0}_{6\times 6}  &  \mathbf{0}_{6\times 6} 
\esm \,,
\end{equation}
with $\Psi$ from Theorem~\ref{thm:minEnFilter_Acc}.
\end{lemma}

With these lemmas we are able to prove our main result in Theorem~\ref{thm:minEnFilter_Acc}:

\begin{proof}[of Theorem \ref{thm:minEnFilter_Acc}]
We can easily compute the differential of Hamiltonian in \eqref{eq:left-trivial-hamilton} which is
\begin{equation}\label{eq:D2hamiltonian}
- \D_{2} \mc H^{-}(G^{\ast},0,t) =  \bigl(\matse(v^{\ast}), \mathbf{0} \bigr) = f(G^{\ast}) \,.
\end{equation}
By inserting expression~\eqref{eq:D2hamiltonian} into the optimal state equation~\eqref{eq:ode-optimal-state1} together with the definition of the operator 
$\vecg(Z(G^{\ast},t)^{-1}\circ G^{\ast}\eta) = K^{-1}(t)\vecg(\eta)$, 
we find that
\begin{align}
(& G^{\ast})^{-1} \dot{G}^{\ast} = f(G^{\ast}) \nonumber \\
& - e^{-\alpha(t-t_{0})}\matg\bigl(\vecg(Z(G^{\ast},t)^{-1}\circ r_{t}(G^{\ast})) \bigr) \nonumber \\
= & f(G^{\ast}) - e^{-\alpha(t-t_{0})} \matg\bigl(K^{-1}(t) \vecg(r_{t}(G^{\ast}))\bigr)\,. \label{eq:op-state-withK}
\end{align}
The application of the $\vecg-$operation onto the equation \eqref{eq:Zall}\, evaluated for a direction $\eta$, together with Lemma~\ref{lemma2} results in
\begin{align}\label{eq:odeK}
\begin{split}
\dot{K}&(t) \vecg(\eta) = \Bigl[ K(t)B + B^{\T}K(t) \\
&+ e^{-\alpha(t-t_{0})} \bsm
\sum_{k=1}^{n} (\tilde{\Gamma}_{\vecse(\on{Pr}(A_{k}(E)))} + D_{k}(E)) & \mathbf{0}_{6\times 6} \\
\mathbf{0}_{6\times 6} & \mathbf{0}_{6\times 6}
\esm \\
& -K(t) \bsm \mathbf{0}_{6\times 6} & \eins_{6} \\ \mathbf{0}_{6\times 6} & \mathbf{0}_{6\times 6}   \esm
- \bsm \mathbf{0}_{6\times 6} & \mathbf{0}_{6\times 6}  \\ \eins_{6} & \mathbf{0}_{6\times 6}   \esm K(t) \\
&- e^{\alpha (t-t_{0})} K(t)S^{-1}K(t) \Bigr] \vecg(\eta),
\end{split}
\end{align}
where on the right-hand side we assume that $K(t)$ is an approximation of the vectorized operator $Z(G^{\ast}(t),t).$ This is the reason why we replace the approximation by an equality sign in \eqref{eq:odeK}.
With a change of variables (cf. \cite{saccon2013second}) 
\begin{equation}\label{eq:defP}
P(t):= e^{-\alpha(t-t_{0})}K(t)^{-1}\,,
\end{equation}
and the formula for the derivative of the inverse of a matrix \cite{petersen2008matrix},
we obtain
\begin{align}\label{eq:odeP}
\dot{P}(t) =&  - \alpha e^{-\alpha(t-t_{0})}K(t)^{-1} - e^{-\alpha(t-t_{0})}K(t)^{-1} \dot{K}(t) K(t)^{-1} \nonumber \\
= & -\alpha P(t) -  e^{\alpha(t-t_{0})}P(t) \dot{K}(t) P(t)\,.
\end{align}
Insertion of \eqref{eq:odeK} (after omitting the direction $\vecg(\eta)$ that was chosen arbitrarily) into \eqref{eq:odeP} leads to the differential equation \eqref{eq:optimalStateP} in Theorem~\ref{thm:minEnFilter_Acc}.
Therefore, we also find that  
\begin{equation}\label{eq:connection_C_B}
C(G^{\ast},t) =  \bsm
\mathbf{0}_{6\times 6} & \eins_{6} \\
\mathbf{0}_{6 \times 6} & \mathbf{0}_{6\times 6}
\esm  - B(t) \,. 
\end{equation}

The differential equation of the optimal state \eqref{eq:optimalStateG} follows from inserting \eqref{eq:defP} into \eqref{eq:op-state-withK}, which completes the proof. \qed
\end{proof}

\subsection{Generalization to Higher-order Models}
In the previous section, we discussed minimum energy filters to estimate \egomotion under the assumption of \emph{constant acceleration}.
We saw that changing the assumption of \emph{constant velocity} to \emph{constant acceleration} requires extending the Lie group and adopting the functions $f(G)$ and $C(G)$.

The generalization to higher polynomial models regarding camera motion, where we assume that the $m$-th order derivative of the \egomotion should be zero, i.e.\
\begin{equation}
\frac{d^{m}}{d t^{m}} E(t) =  0,
\end{equation}
is straightforward.

Again, the approach can be described by a system of first-order ODEs as follows. Note that in the constant acceleration model (second-order), only the first-order model 
needs to respect manifold structures, whereas all the other derivatives are trivial since they evolve on Euclidean spaces:
\begin{align}
\begin{split}\label{eq:state-general}
\dot{E}(t) = & E(t)\bigl(\matse{v}_{1}(t) + \delta_{1}(t) \bigr), \\
\dot{v}_{1}(t) = & v_{2}(t) + \delta_{2}(t), \\
\vdots &  \\
\dot{v}_{m-2}(t) = & v_{m-1} + \delta_{m-1}(t), \\
\dot{v}_{m-1}(t) = & \delta_{m}(t)
\end{split}
\end{align}
To achieve a unique solution we require initial values, i.e.\ $v_{1}(0) = v_{1}^{0}, \dots, v_{m-1}(0) = v_{m-1}^{0} \in \R^{6}.$
Again, the observation equations \eqref{eq:observation-eq} stay unchanged. The minimum energy filter for this model is provided by the following theorem. By using once again
\begin{equation}
G = (E,v_{1},\dots,v_{m-1}) \in \G_{m} := \SE \times \R^{6} \times \cdots \times \R^{6}\,,
\end{equation}
 the corresponding minimum energy filter can be obtained easily from Theorem~\ref{thm:minEnFilter_Acc}. 

\begin{theorem}[Minimum energy filter for $m-$th order state equation]\label{thm:MEF-higher-order}
The differential equations of the second-order Minimum Energy Filter for the state equation \eqref{eq:state-general} and the observation equations \eqref{eq:observation-eq} are given by the equations \eqref{eq:optimalStateG} and
{\small
\begin{align}
 \begin{split}
& \dot{P}(t) =   -\alpha \cdot P + S^{-1} + C P + P C^{\T}\\
& - P 
\begin{pmatrix}
\sum_{k=1}^{n} (\tilde{\Gamma}_{\vecse(\on{Pr}(A_{k}(E^{\ast})))} + D_{k}(E^{\ast})) & \mathbf{0}_{6\times (m-1)6} \\
\mathbf{0}_{(m-1)6\times 6} & \mathbf{0}_{(m-1)6\times (m-1)6}
\end{pmatrix} P \,, \\
& P(t_{0}) = \eins_{6m}\,,
\end{split}
\end{align}} 
where we assume that the expressions $G^{\ast}$ and $P$ lie in the spaces $\G_{m}$ and $\R^{6m\times 6m},$ respectively. The appearing expressions in Theorem~\ref{thm:minEnFilter_Acc} are replaced by
\begin{align*}
f(G) := &  (\matse(v_{1}),v_{2},\dots,v_{m-1},\mathbf{0_{6\times 1}}), \\
r_{t}(G^{\ast}) := & \bigl(\sum_{k=1}^{n} \on{Pr}(A_{k}(E^{\ast}), \mathbf{0}_{(m-1)6 \times 1}\bigr), \\
C(G^{\ast},t) := & \begin{pmatrix}
\bsm -\Psi(G^{\ast},t) \\ \mathbf{0}_{6(m-2)\times 6} \esm & \eins_{6(m-1)} \\
 \mathbf{0}_{6\times 6} & \mathbf{0}_{6\times 6(m-1)}
 \end{pmatrix} \,.
 \end{align*}
All the other expressions from Theorem~\ref{thm:minEnFilter_Acc} stay unchanged.
\end{theorem}
\begin{proof}
Since product Lie groups are simply Lie groups with the product topology, we can still apply the general minimum energy filter of Saccon {\em et al.}~\cite{saccon2013second}. The Lie group $\G_{m}$ has dimension $6m$ such that the vectorized bilinear operator~$Z$ from~\eqref{eq:defZ}, i.e.\ $P$ results in a $6m \times 6m$ matrix. The definition of the function~$f$ follows from the differential equations in~\eqref{eq:state-general}.
Similarly to Theorem~\ref{thm:minEnFilter_Acc}, the observations do not depend on the whole state $G=(E,v_{1},\dots,v_{m-1})$, but only on~$E$. This leads to the fact that~$r_{t}$, which is essentially the left-trivialized differential of the Hamiltonian (i.e. $G^{-1} \D_{1} \mc H^{-}(G,\mathbf{0},t)$), vanishes after calculating the differentials regarding $v_{1},\dots,v_{m-1}$. Similarly, the Hessian $G^{-1} \Hess_{1} \mc H^{-}(G,\mathbf{0},t)[G\eta]$ in Lemma~\ref{lemma2} can be extended  by zeros.
Furthermore, components $v_{1},\dots,v_{m-1} \in \R^{6}$ have a trivial geometry and do not contribute to curvature and thus the corresponding connection functions in Lemma~\ref{lemma2} also do not influence curvature.
Finally, we can compute the expression
\begin{align*}
\begin{split}
\D_{1}(\D_{2} \mc H^{-}(G,& \mathbf{0}, t))[G \eta] = -\D f(G)[\eta] \\
&  = -(\matse(v_{2}),v_{3},\dots,v_{m-1},\mathbf{0})
\end{split}
\end{align*}
and thus
\begin{align*}
\begin{split}
\vecg\bigl(\D_{1}(\D_{2}& \mc H^{-}(G,\mathbf{0}, t))[G \eta] \bigr) = \bsm \mathbf{0}_{6(m-1)\times 6} & \eins_{6(m-1)} \\ \mathbf{0}_{6\times 6} & \mathbf{0}_{6\times 6(m-1)}  \esm ,
\end{split}
\end{align*}
as we did in Lemma~\ref{lemma2} for the special case. Together with the adjoint operator in $\Psi(G,t)$, we obtain the expression~$C$. \qed
\end{proof}

\section{Comparison with Extended Kalman Filters}

As an alternative to the proposed approach, we also suggest considering extended Kalman filters. For this purpose, we will compare our approach to a state-of-the-art {\em discrete~/ continuous extended Kalman filter on Lie groups} \cite{bourmaud2015continuous} in Section~\ref{sec:exp}.
The Kalman filter approach is valid in a more generalized scenario compared to ours because the state space as well as the observation space are matrix Lie groups, whereas we only consider real-valued observations in $\R^{n}.$ On the other hand, one needs to know that the covariance matrices of the model and observation noise and the {\em a posteriori} distribution are assumed to be Gaussian, which is in general not true for non-linear observation dynamics. 

\begin{algorithm}[th!]
\caption{Extended Kalman Filter for Lie Groups}\label{Alg:ExtendedKalman}
\begin{algorithmic}[1]
\Require State $G(t_{l-1})$,  Covariance $P(t_{l-1})$,  Observations $y_{k}(t_{l}),k=1,\dots,n$
\Procedure{Propagation on $[t_{l-1},t_{l}]:$}{} \em{Integrate the following differential equations}
 \State $\dot{G}(t) = G(t)f(G(t)) $
 \State \begin{varwidth}[t]{\linewidth}
      $\dot{P}(t) = J(t) P(t) + P(t)(J(t))^{\T} + S$ \par
        \hskip\algorithmicindent $+ \tfrac{1}{4}\EE(\ad_{\g}(\epsilon(t))S \ad_{\g}(\epsilon(t))^{\T})$\par
        \hskip\algorithmicindent $+  \tfrac{1}{12}\EE \bigl(\ad_{\g} (\epsilon(t))^{2}\bigr)S + \tfrac{1}{12}S\EE \bigl(\ad_{\g} (\epsilon(t))^{2}\bigr)^{\T} $
      \end{varwidth}  \label{line:dyn-P}
 \State $G^{-}(t_{l}) = G(t_{l})$, $P^{-}(t_{l}) = P(t_{l})$
 \EndProcedure
\Procedure{Update:}{}
 \State $K_{l} = P^{-}(t_{l}) H_{l}^{\T} \bigl(H_{l} P^{-}(t_{l}) H_{l}^{\T} + Q_{l} \bigr)^{-1}$
 \State $m_{l \vert l}^{-} = K_{l} \sum_{k=1}^{n}\bigl(y_{k}(t_{l})-h_{k}(G^{-}(t_{l}))\bigr)$ \label{line:residual}
 \State $G(t_{l}) = G^{-}(t_{l})\Exp( \matg(m_{l\vert l}^{-}))$
 \State $P(t_{l}) = \Phi(m_{l\vert l}^{-})\bigl(\eins_{12} - K_{l}H_{l}\bigr)P^{-}(t_{l})\Phi(m_{l\vert l}^{-})^{\T}$ \label{line:Phi}
 \EndProcedure
\end{algorithmic}
\end{algorithm}

The extended Kalman Filter  from~\cite{bourmaud2015continuous} is summarized in Algorithm~\ref{Alg:ExtendedKalman} and has already been adapted to our problem for real-valued observations. In line~\ref{line:residual} the residual is expressed as direct difference which is a special case of  \cite{bourmaud2015continuous}.  The function $\Phi$ in line~\ref{line:Phi} on $\G$ is shown in Appendix~\ref{sec:app-ext-kalman}. 

In the next section, we will adapt the Algorithm~\ref{Alg:ExtendedKalman} to different scenarios: to a filtering problem with linear observations as well as to our nonlinear filtering problem with a projective camera (cf.~\eqref{eq:state-eq}, \eqref{eq:observation-eq}). 

\begin{remark}
Note that the extended Kalman filter from~\cite{bourmaud2015continuous} requires a differential equation (that is not only driven by noise) in order to propagate the state, i.e.\ $\dot{E}(t) = E(t)\bigl(f(E) + \delta(t)\bigr)$, where~$f$ is non-trivial. Otherwise the update step of the extended Kalman filter is not significant because update and correction steps in the extended Kalman filter are separated. This is the reason why we only compare it to the second-order model where~$f \not\equiv 0$.
\end{remark}

\subsection{Derivations for Linear Observations}
\label{sec:linear-observations}

In the scenario of linear observations the state equation stays unchanged, i.e.\ is identical to \eqref{eq:state-eq}. Similarly to~\cite{zamani2012} we use the following linear observation equations:
\begin{equation}\label{eq:linear-observer}
y_{k}(t) = E(t) a_{k} + \epsilon_{k}(t), \quad k \in [n],
\end{equation}
where $E(t) \in \SE$ is the first component of $G(t) \in \G$ and $a_{k} \in \R^{4}$ are vectors that model the linear transformation of the state $G.$  Again, $\epsilon_{k}(t) \in \R^{4}$ are the observation noise vectors. 

In this case, the Minimum Energy Filter can be derived much more easily than in the non-linear case. Thus, for the compactness of presentation, we will skip the proof of the following propositions.
\begin{proposition}\label{prop:MEF-linear}
The Minimum Energy filter for the constant acceleration model \eqref{eq:state-eq} and linear observation equations \eqref{eq:linear-observer} is given by the equations \eqref{eq:optimalStateG} and \eqref{eq:optimalStateP}
where the function $A_{k}$ for $G=(E,v)$ is replaced by
\begin{align}
A_{k}(G) = & E^{\T} Q (Ea_{k} - y_{k}) a_{k}^{\T}, 
\end{align}
and the components $(i,j), i,j=1,\dots,6$ of the matrix $D_{k}(G)\in \R^{6\times 6}$ are given by
\begin{equation}
(D_{k}(G))_{i,j} = \zeta_{i}^{k}(E)(E^{j}), \quad E^{j}:= \matse(e_{j}^{6})\,,
\end{equation}
with $\zeta^{k}(E)(\cdot): \se \rightarrow \R^{6}$ given by
\begin{align}
\begin{split}
 \matse(&\zeta^{k}(E)(\eta_{1})) \\
 & := \on{Pr}\bigl(\eta_{1}^{\T} Q(Ea_{k}-y_{k})a_{k}^{\T} + E^{\T}Q \eta_{1} a_{k}a_{k}^{\T}\bigr)\,.
 \end{split}
\end{align}
Here,  $Q \in \R^{4\times 4}$ is a symmetric and positive definite matrix (cf.~\eqref{eq:penalty-function}). All other expressions from Theorem~\ref{thm:minEnFilter_Acc} stay unchanged.
\end{proposition}

Since the linear observation model is a special case of the approach in~\cite{bourmaud2015continuous} we only need to modify the corresponding expressions in Algorithm~\ref{Alg:ExtendedKalman} which we summarize in the following proposition.

\begin{proposition}\label{prop:EKF-linear}
The Extended Kalman Filter for the constant acceleration model~\eqref{eq:state-eq} and linear observation equations~\eqref{eq:linear-observer} is given by Algorithm~\ref{Alg:ExtendedKalman} where the matrix $H_{l}:= \sum_{k=1}^{n} H_{l}^{k}$ is given by
\begin{align}
H_{l}^{k} = \bpm  \vecse(\Pr(E(t_{l})^{\T} e_{1}^{4}a_{k}^{\T}))^{\T} & \mathbf{0}_{1\times 6} \\
\vecse(\Pr(E(t_{l})^{\T} e_{2}^{4}a_{k}^{\T}))^{\T} & \mathbf{0}_{1\times 6} \\
\vecse(\Pr(E(t_{l})^{\T} e_{3}^{4}a_{k}^{\T}))^{\T} & \mathbf{0}_{1\times 6} \\
\vecse(\Pr(E(t_{l})^{\T} e_{4}^{4}a_{k}^{\T}))^{\T} & \mathbf{0}_{1\times 6} \\ \epm \in \R^{4\times 12} \label{eq:hk-linear-case}
\end{align}
and the function $J(t)$~(\cite[Eq. (52)]{bourmaud2015continuous})
is provided 
by~\eqref{eq:def-J-ext-Kalman} in Appendix~\ref{sec:app-ext-kalman}.
\end{proposition}

\begin{remark}
Note that \eqref{eq:hk-linear-case} is different from~\cite[Eq.~(111)]{bourmaud2015continuous} because of the additive instead of multiplicative noise term, and consequently is not consistent with the group structure of $\SE$.
\end{remark}

\subsection{Derivations for Nonlinear Observations}
The adaption of the extended Kalman Filter~\cite{bourmaud2015continuous} to our state~\eqref{eq:state-eq} and observation~\eqref{eq:observation-eq} equation is provided by the following proposition:
\begin{proposition}
The extended Kalman filter from \cite{bourmaud2015continuous} for our state \eqref{eq:state-eq} and observation \eqref{eq:observation-eq} equation is given by Algorithm~\ref{Alg:ExtendedKalman} where the expressions $J(t)$ and $H_{l}$ are provided in the equations~\eqref{eq:def-J-ext-Kalman} and~\eqref{eq:Hl-nonlinear}, respectively, see Appendix~\ref{sec:app-ext-kalman}.
\end{proposition}

\section{Numerical Geometric Integration}\label{sec:numInt}

The numerical integration of the optimal state differential equation \eqref{eq:optimalStateG} requires respecting the geometry of the Lie group.
We use the implicit Lie midpoint rule for integration of the differential equation of the optimal state $G^{\ast}$  \eqref{eq:optimalStateG} as proposed in \cite{hairer2006}. We need to modify the method since we defined state space $\G$ as left invariant Lie group. Instead, in~\cite{hairer2006}, only right-invariant Lie groups are investigated. The adaption to left-invariant Lie groups is straightforward and leads to the following integration schemes: for a discretization $t_{0}<t_{1}<\cdots<t_{n}$ with equidistant step size $\delta = t_{k}-t_{k-1}$ for all $k$, we integrate the differential equation of the optimal state~\eqref{eq:optimalStateG} using the scheme
\begin{align}
G(t_{k+1}) & = G(t_{k}) \Exp( \Xi )\,,\\
\begin{split}\label{eq:fixedPointOmega}
\mbox{with } \Xi  & = \delta \bigl(f(G(t_{k})\Exp(\Xi/2)))\\
& \quad - \matg(P(t_{k})\vecg(r_{t}(G(t_{k})\Exp(\Xi/2)))\bigr)\,.
\end{split}
\end{align}
For each $k$ the matrix $\Xi$ is received by a fixed point iteration of \eqref{eq:fixedPointOmega}. For the integration of equation \eqref{eq:optimalStateP}, we need to consider that this is a special kind of the {\em matrix Riccati differential equation} for which methods exist that ensure that the solution is positive definite. As shown in \cite{dieci1994positive}, a numerical integration method will preserve positive definiteness if and only if the order of the method is one. 
By taking down \eqref{eq:optimalStateP} as general Riccati differential equation
\begin{equation}
\dot{P}(t) = A(t) P(t) + P(t) A(t)^{\T} - P(t)B(t)P(t) + C(t)\,,
\end{equation}
with symmetric matrices $B(t)$ and $C(t),$ the implicit Euler integration method is given by
\begin{align}
\begin{split}
P(t_{k+1}) = & P(t_{k}) + \delta\bigl( AP(t_{k+1})  + P(t_{k+1}) A^{\T}  \\
& - P(t_{k+1})BP(t_{k+1}) + C \bigr),
\end{split}
\end{align}
which can be expressed by the \emph{algebraic Riccati equation} for which an unique solution exists \cite{lancaster1980existence} that can be found by standard solvers, e.g.\ CARE.

\section{Experiments}\label{sec:exp}
\begin{figure}[tb]
\includegraphics[scale=0.18]{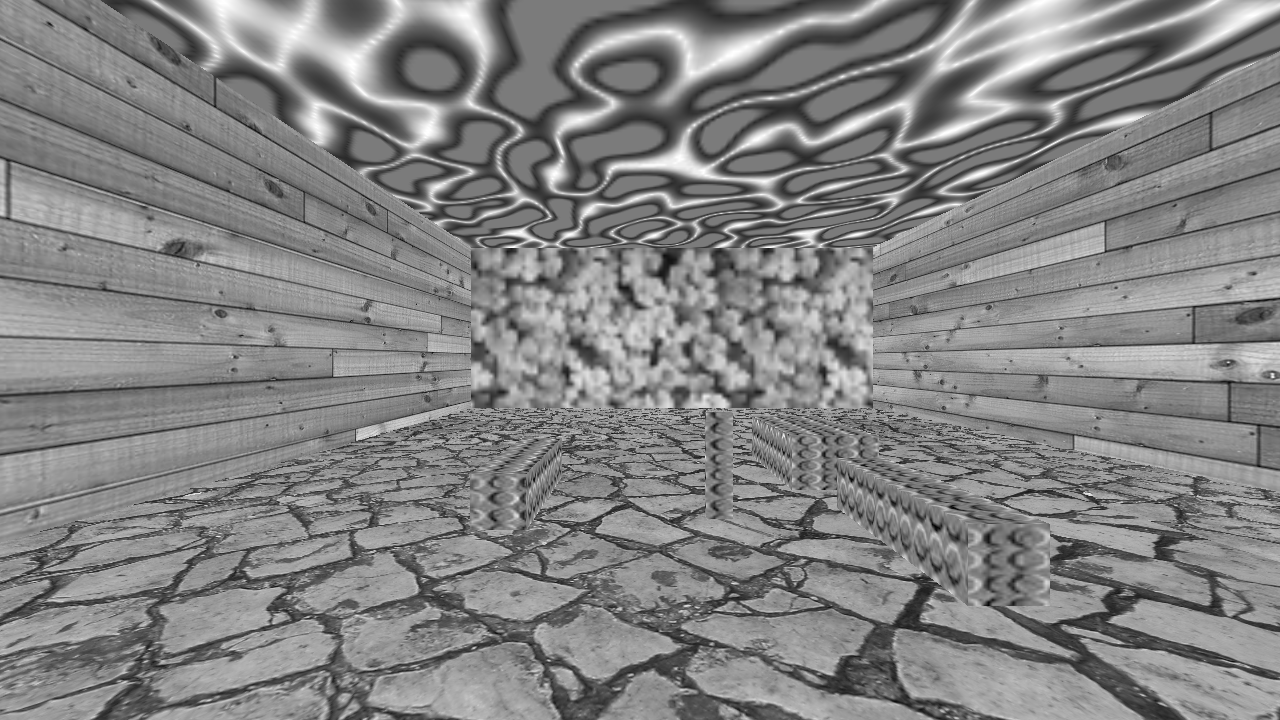}\\
\includegraphics[scale=0.116]{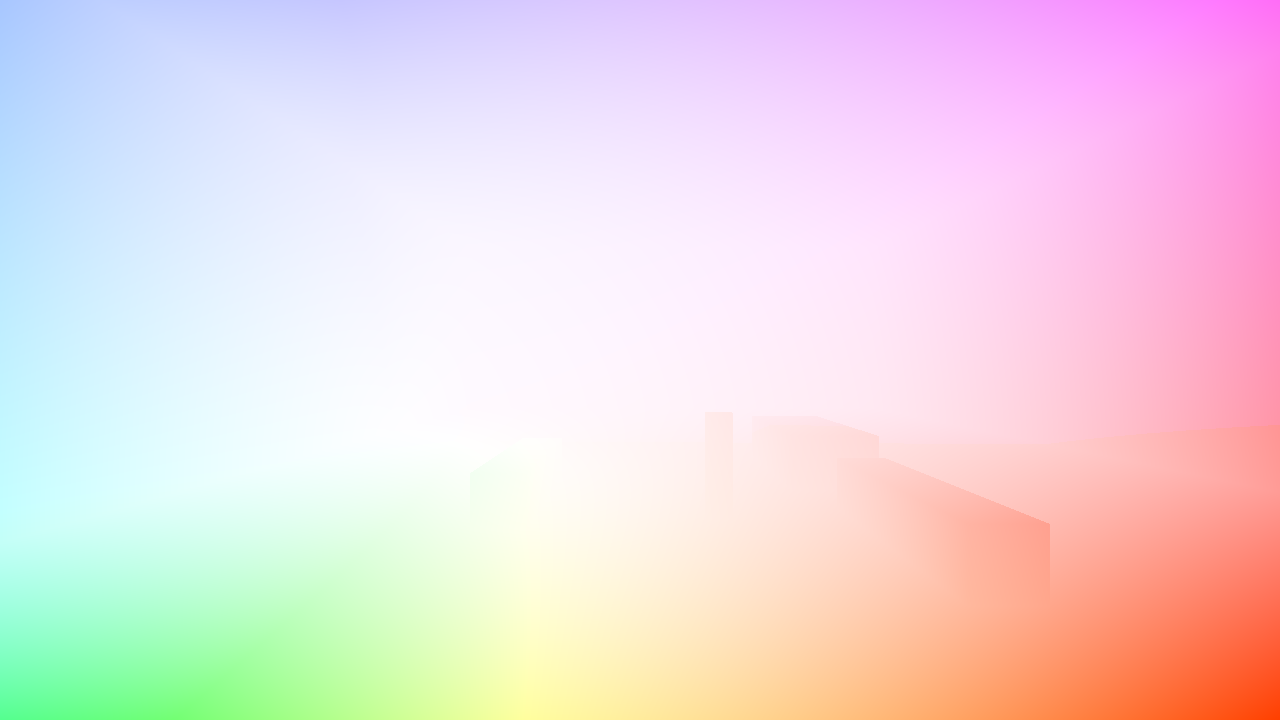}\hspace{-1mm} \includegraphics[scale=0.552]{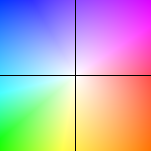}
 \\
\includegraphics[scale=0.116]{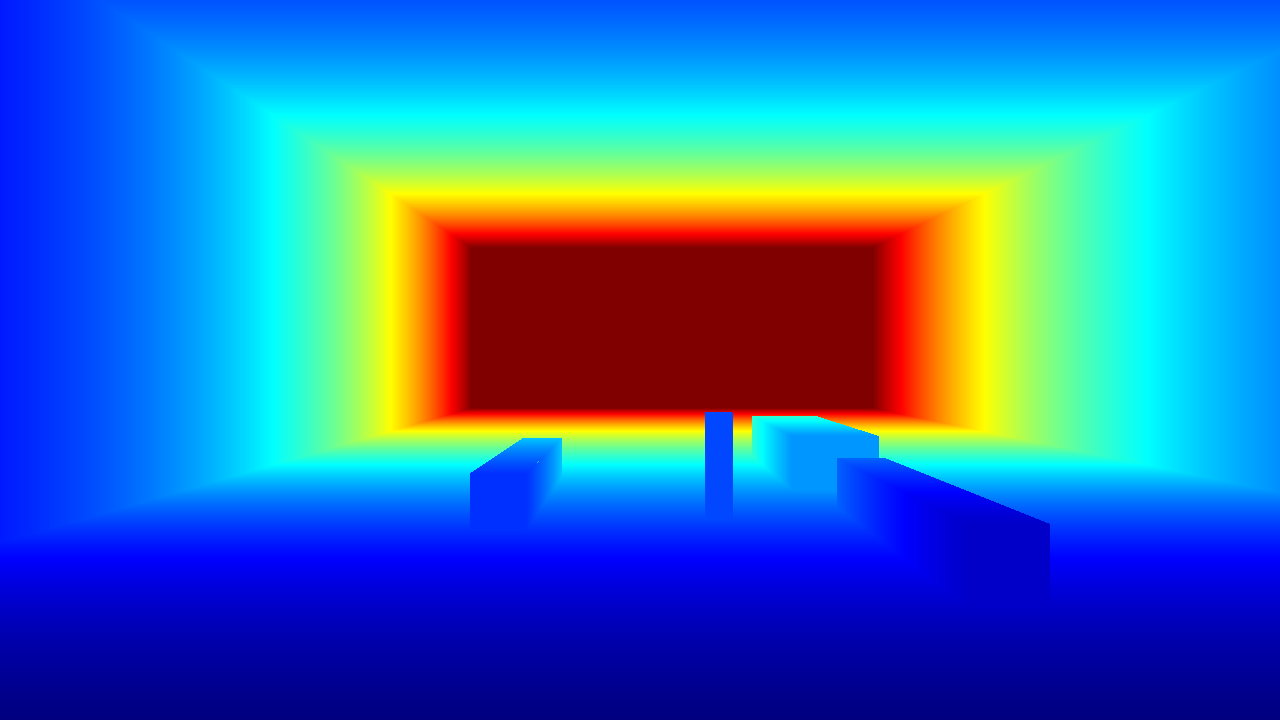}\includegraphics[scale=0.37]{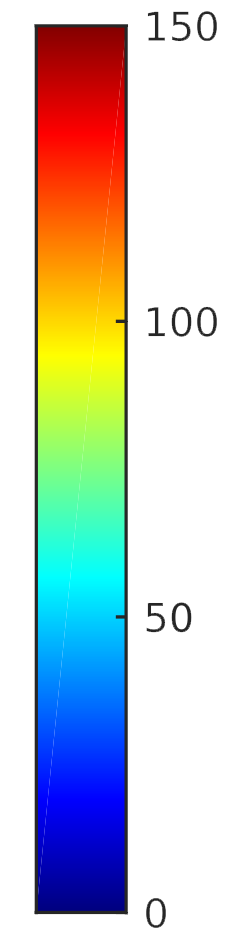}
\caption{Synthetic sequence (top) generated by a simple ray tracer. To provide realistic camera tracks we used ground truth trajectories from the KITTI odometry benchmark and computed the corresponding induced optical flow (mid) and the depth map (bottom). The corresponding color encodings for direction of optical flow and depth map are on the right hand side.}
\label{fig:synth}
\end{figure}

In this experimental section, we will evaluate the accuracy of the proposed minimum energy filter for ego-motion estimation. First we will provide experiments on synthetic data to exclude external influences and to show robustness against measurement noise. Then we will consider real world experiments on the challenging KITTI benchmark and compare our method with a state-of-the-art method~\cite{geiger2011stereoscan}. Finally, to evaluate the theoretical performance of the filter, we will also compare to the state-of-the-art extended Kalman filter~\cite{bourmaud2015continuous} in a controlled environment.

\subsection{Synthetic Data}
Before considering real-life sequences, we first evaluate synthetic scenes to have full control on the regularity on the camera track. We generate 3D scenes by raytracing simple geometric objects (cf.~Fig.~\ref{fig:synth}), which also enables us to acquire correctly induced optical flow and depth maps. In order to gain a realistic camera behavior, we use the tracks from the KITTI visual odometry training benchmark which were determined by an inertial navigation system in a real moving car. We start with considering the case of perfect measurements (Section~\ref{sec:synth-perfect-measurements})
and demonstrating robustness against different kinds of noise in Section ~\ref{sec:synth-noisy-measurements}.

\subsubsection{Evaluation on Noiseless Measurements}\label{sec:synth-perfect-measurements}
First, we evaluate the proposed filter on the true optical flow.
To avoid overfitting, we set a relatively small weight onto the weighting matrix for the data term, i.e.~$Q = 0.1/n$, where~$n$ is the number of observations. We set the weighting matrix $S$ to the block diagonal matrix containing the matrices $S_{i}$, i.e.
\begin{equation}\label{eq:def-blockdiag}
S = \operatorname{blockdiag}(S_{1},\dots,S_{m})\,,
\end{equation}
where~$m$ denotes the order of the kinematic model and the $S_{1} = \diag(s_{1},s_{1},s_{1},s_{2},s_{2},s_{2})$ with $s_{1}=10^{-2}$ and $s_{2}=10^{-5}$. The decay rate is set to $\alpha=2$ and the integration step size to $\delta=1/50$.

As demonstrated in Fig.~\ref{fig:synth_without_noise}, the proposed filters of different order show a similar rotational error since the ground truth rotation is often constant and influenced
by (physical) noise. That is possibly caused by the low temporal resolution of 10 Hz, not being able to give sufficient information on the kinematics. On the contrary, in the translational part we can see that the higher-order models work significantly better than our first-order model~\cite{berger2015second}, but that third- and fourth-order methods perform fairly the same. 
From this we can conclude that kinematics of fifth- or even higher-order will not improve performance regarding this kind of camera tracks.

\begin{figure*}[htbp]
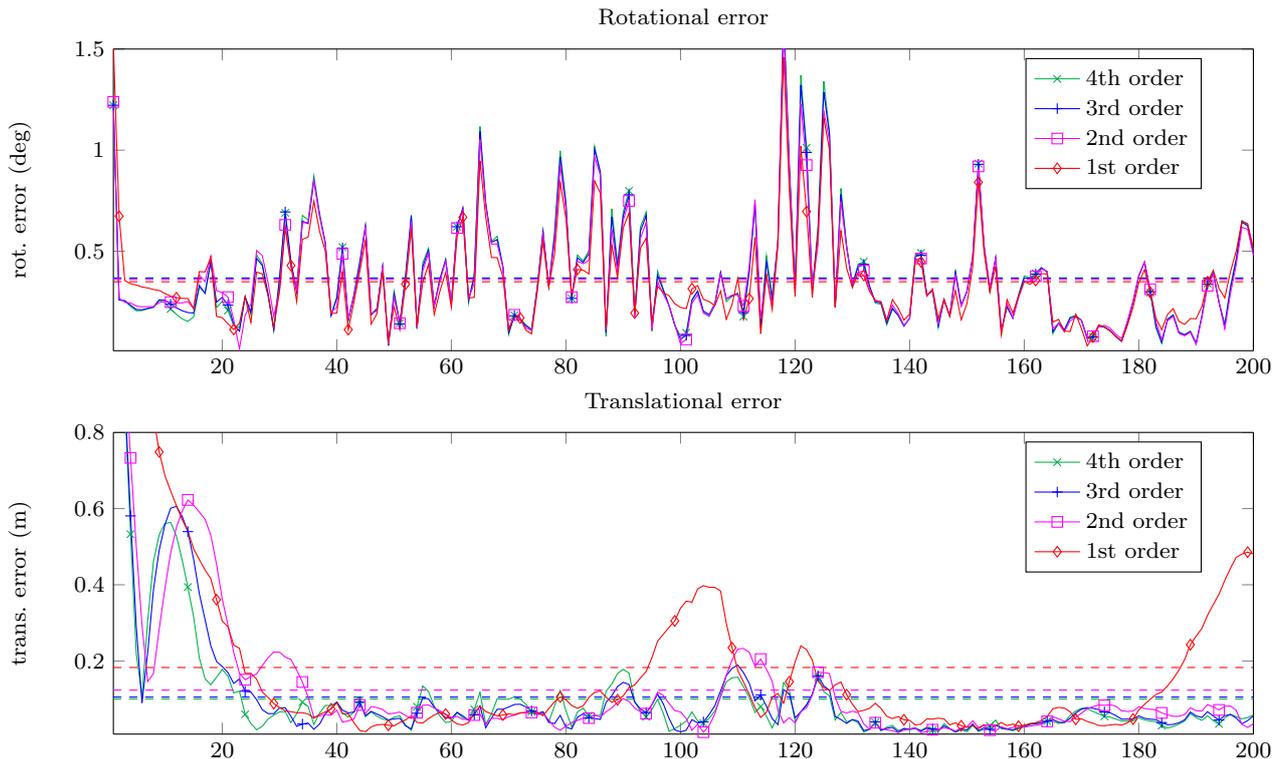

\begin{center}
%
%
\definecolor{mycolor1}{rgb}{0.,0.7,0.3}%
\definecolor{mycolor2}{rgb}{1.00000,0.00000,1.00000}%

\caption{Comparison of the rotational error in degree (top) and the translational error in meters (bottom) of the proposed minimum energy filters with kinematic state equations of orders one~(see \cite{berger2015second}) and two, three and four (this work). The dotted lines show the error averaged over all frames. We used a real camera track from sequence~0 of the KITTI visual odometry benchmark and generated synthetic sequences with induced depth maps and optical flow. The rotational errors are similar through all frames although the higher-order methods converge faster in the first iterations. In frames 20--90, the motion of the camera is almost constant and the filters perform similarly. However, the translational error of the first order method significantly changes in frames 90--150 and 175--200 because the constant velocity assumption is violated by curves in the trajectory.}
\label{fig:synth_without_noise}
\end{center}
\end{figure*}

\subsubsection{Evaluation on Noisy Measurements}\label{sec:synth-noisy-measurements}
To evaluate the robustness against noise, we altered the true optical flow measurements by 
multiplicative and additive noise, each being distributed uniformly or Gaussian, see Fig.~\ref{fig:synth_noise}.
The proposed method determines camera motion using the same parameters as in Section ~\ref{sec:synth-perfect-measurements}.
Comparison to the ground truth is achieved using the geodesic distance on~$\SE$ in order to avoid two separate error measures for translation and rotation, i.e. 
\begin{equation}\label{eq:geo-error}
d_{\SE}(E_{1},E_{2}) := \lVert \vecse(\Log(E_{1}^{-1}E_{2})) \rVert_{2} \,.
\end{equation}

The results in Tab.~\ref{tab:noise-quantitative} show that higher-order models outperform the first-order model with the exception of very high noise levels where the data does not contain sufficient information to correctly estimate a higher-order kinematic.

\begin{figure}[tb]
\begin{center}
\includegraphics[scale=0.09]{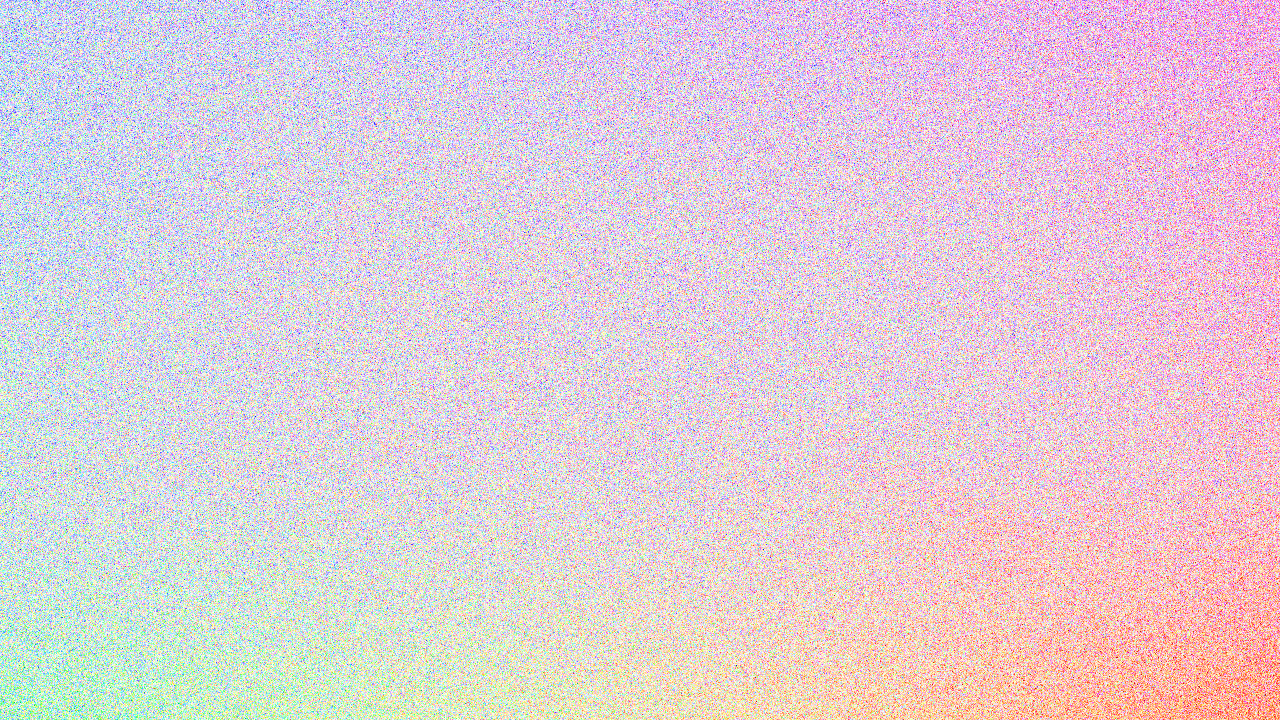}\includegraphics[scale=0.09]{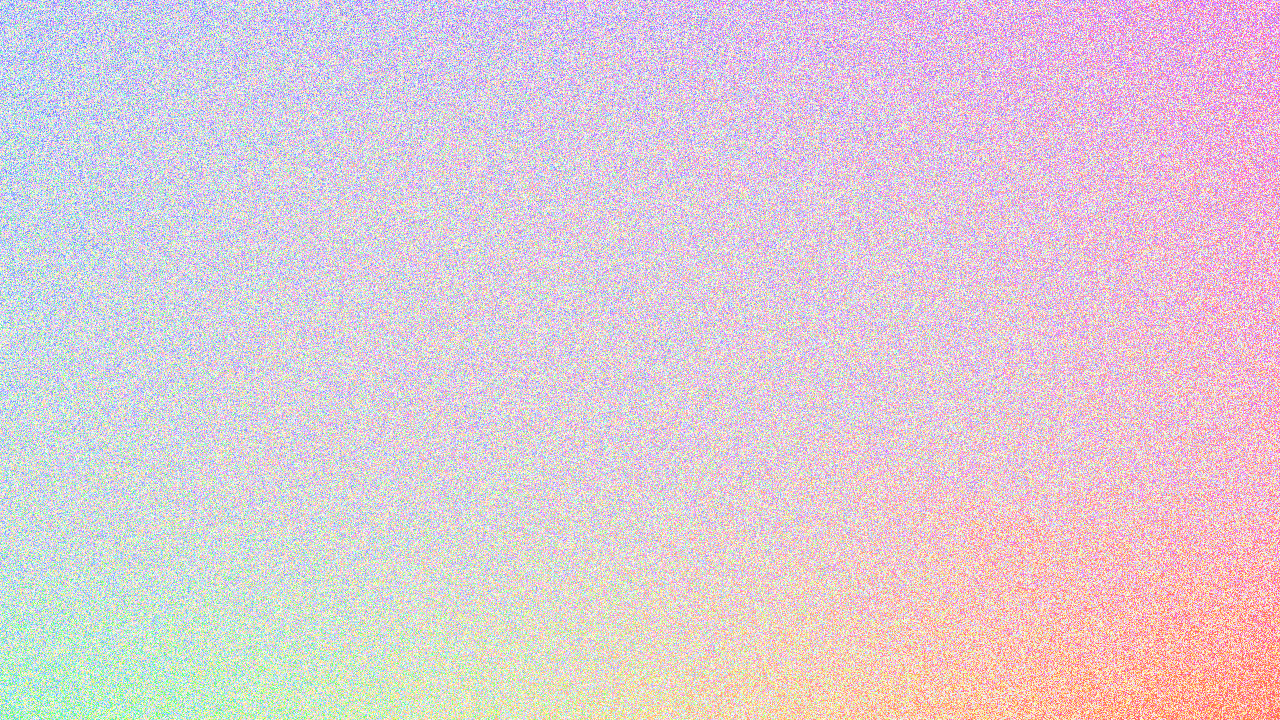}
\includegraphics[scale=0.09]{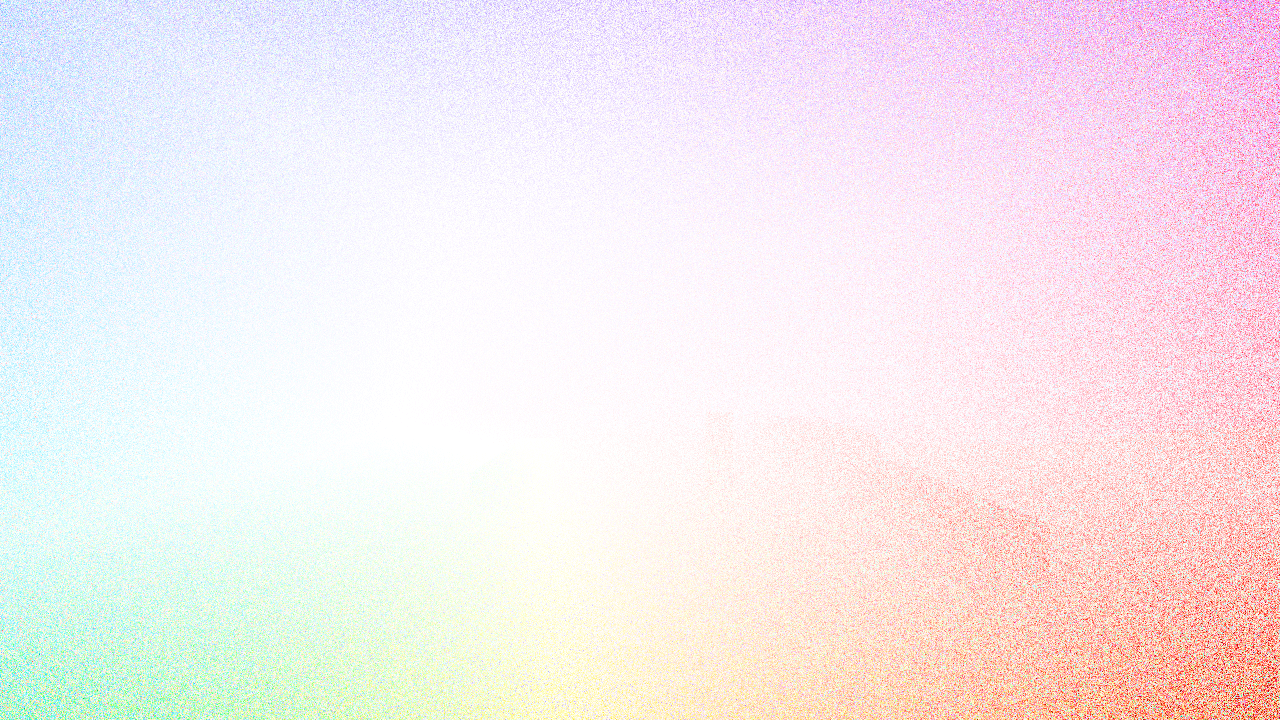}\includegraphics[scale=0.09]{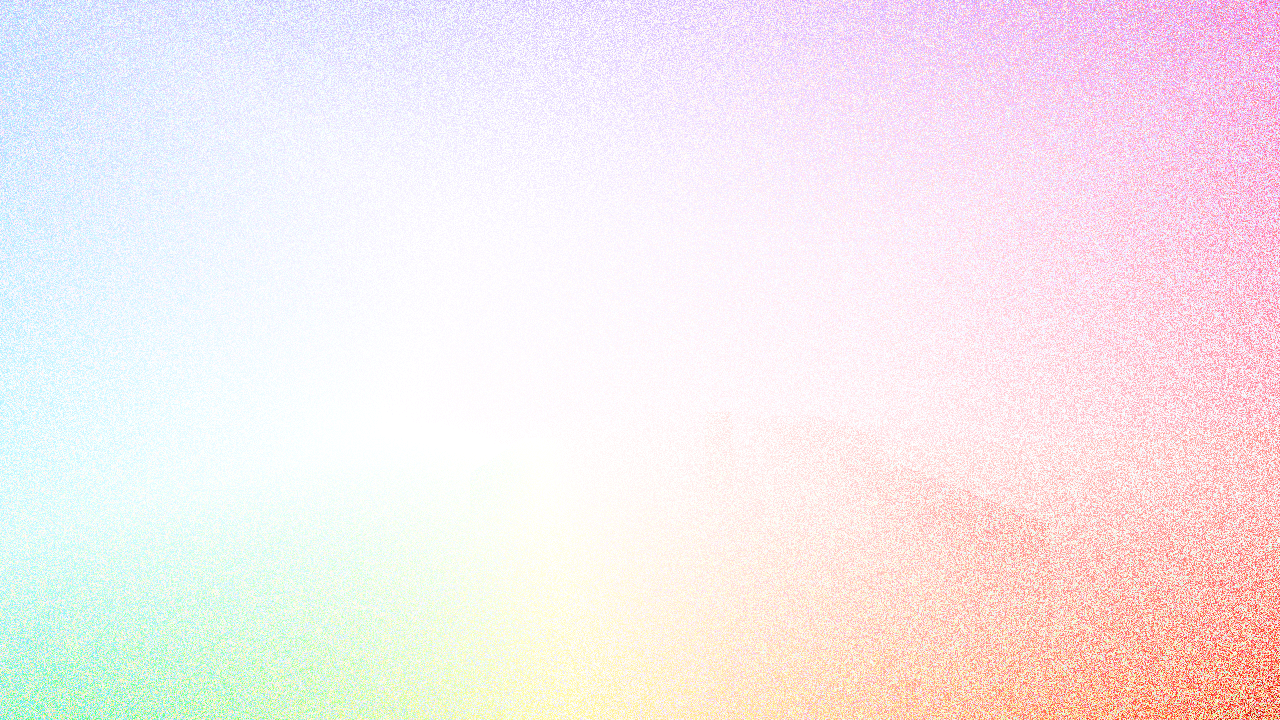}
\caption{Different noise models for the observed data (optical flow, cf.~Fig.~\ref{fig:synth}): top left: additive Gaussian noise ($\mu=0$, $\sigma^{2}=0.001$), top right: additive uniform noise ($\mu=0$, $\sigma^{2}=0.001$), bottom left: multiplicative Gaussian noise ($\mu=1$, $\sigma^{2}=1$), bottom right: multiplicative uniform noise ($\mu=1$, $\sigma^{2}=1$).}
\label{fig:synth_noise}
\end{center}
\end{figure}
\begin{table}[htdp]
\caption{Quantitative evaluation 
of proposed methods (order 1 to 4) 
measuring the 
geodesic error (cf.~\eqref{eq:geo-error})
w.r.t.\ ground truth camera motion.
As input data we used noisy flow observations with the following noise models: additive Gaussian (AG, $\mu=0$), additive uniform (AU, $\mu=0$), multiplicative Gaussian (MG, $\mu=1$) and multiplicative uniform (MU, $\mu=1$) for different variances~$\sigma^{2}$. For intense noise (multiplicative: $\sigma^{2}>10^{-1}$, additive: $\sigma^{2}>10^{-4}$), the first-order method performs better than higher-order models since it is more robust against noise. In contrast, for moderate noise levels, higher-order kinematics are more appropriate.
}
\begin{center}
\begin{tabularx}{\columnwidth}{l l @{\extracolsep{\fill}}llll} \toprule
noise		&   $\sigma^{2}$			&1st order 	   & 	2nd order &	3rd order & 	4th order\\ \midrule
MG		 	&  \multirow{2}{*}{$10^{0}$} 	& \textbf{0.2162}  & 	0.2759     &   	0.2821     & 	0.2866  \\
MU			&  						& \textbf{0.2856}  & 	0.3840     & 	0.3705     & 	0.3705    \\ \midrule
MG 			& \multirow{2}{*}{$10^{-1}$} 	& 0.1597 		    &	0.1644     &  	0.1485     & 	\textbf{0.1423} \\
MU 			&						& \textbf{0.2072}  & 	0.2596     &  	0.2367     &  	0.2287   \\ \midrule
MG 			& \multirow{2}{*}{$10^{-2}$}  	& 0.1417 		    &	0.1184     & 	0.1041     &	\textbf{0.1011}  \\
MU 			& 						& 0.1517  		    & 	0.1353     & 	0.1143     &  	\textbf{0.1082} \\ \midrule
MG			& \multirow{2}{*}{$10^{-3}$}       & 0.1283 		    & 	0.0987     & 	0.0844     &	\textbf{0.0808} \\
MU			& 						& 0.1300 		    &	0.0952 	& 	0.0808 	& 	\textbf{0.0777}  \\ \midrule
AG 			& \multirow{2}{*}{$10^{-3}$} 	& \textbf{0.2859}   & 	0.4355      & 	0.4318      & 	0.4385 \\
AU 			&						& \textbf{0.4835}   &	0.7431      &  	0.7175      & 	0.7071 \\ \midrule
AG			& \multirow{2}{*}{$10^{-4}$} 	& \textbf{0.1598}   & 	0.1695     & 	0.1688     & 	0.1701  \\
AU 			&						& \textbf{0.2176}   &	0.2341      & 	0.2216      & 	0.2193 \\ \midrule
AG			& \multirow{2}{*}{$10^{-5}$} 	& 0.1384 		     &	0.1157      & 	0.1010      &  	\textbf{0.0974} \\
AU 			& 						& 0.1263 		     &	0.1130      &  	0.1009      &  	\textbf{0.0968} \\ \midrule
w/o 			& $0$  					& 0.1264    	     &	0.0893      &  	0.0783      &	\textbf{0.0757} \\ \bottomrule
\end{tabularx}
\end{center}
\label{tab:noise-quantitative}
\end{table}

\begin{remark}
Please note that our model currently does not model noise on depth maps explicitly since it only allows additive noise on the {\em flow measurements} as introduced in~\eqref{eq:observation-eq}. However, we think that the noise term $\epsilon$ should also compensate small deviations of the depth.
\end{remark}

\subsubsection{Evaluation of Kinematics}\label{sec:eval-higher-order}
In the last section we showed that the proposed method is robust against different kinds of measurement noise. 
Now we evaluate the proposed minimum energy filters with higher-order kinematic model
for camera tracks of different complexity. For this purpose, we generate camera tracks for the kinematic models (first to fourth order) by (geometric) numerical integration of corresponding differential equation~\eqref{eq:state-general} for $m\in\{1,2,3,4\}$ where we set $v_{0}\equiv 0$. In order to obtain reasonable paths we use non-trivial initializations for $(E_{0},v_{1}^{0}, v_{2}^{0},v_{3}^{0})$. Then we generate synthetic sequences for the different kinematic tracks and use the ground truth optical flow and depth maps as input for the proposed filters. 

The proposed method uses the parameters $Q=0.1 n^{-1} \eins_{2}$ with $n=1000$; and $S$ was chosen as in \eqref{eq:def-blockdiag}, whereas $s_{1}=1$, $s_{2}=0.001$ and $\alpha=0$.

In Fig.~\ref{fig:eval-kinematics} we visualize the geodetical error~\eqref{eq:geo-error} as well as the camera track reconstructions.
It becomes apparent that for a camera track with constant velocity (Fig.~\ref{fig:geo1}) the minimum energy filter with first-order kinematics~\cite{berger2015second} performs best and reaches the highest accuracy. For the other tracks with higher-order kinematics  (cf.~Figures~\ref{fig:geo2}, \ref{fig:geo3} and \ref{fig:geo4}), the proposed filters with higher-order kinematic model work superiorly to~\cite{berger2015second}.

\begin{figure*}[htbp]
\begin{center}
\subfloat[][reconstruced track: first-order kinematics]{
%
%
\definecolor{mycolor2}{rgb}{0.,0.7,0.3}%
\definecolor{mycolor1}{rgb}{1.00000,0.00000,1.00000}%

%
\label{fig:reconst1}} 
\subfloat[][geodetical error on a first-order kinematic track]{
%
%
\definecolor{mycolor1}{rgb}{0.,0.7,0.3}%
\definecolor{mycolor2}{rgb}{1.00000,0.00000,1.00000}%
%
%
\label{fig:geo1}}
\\
\subfloat[][reconstruced track: second-order kinematics]{
%
%
\definecolor{mycolor2}{rgb}{0.,0.7,0.3}%
\definecolor{mycolor1}{rgb}{1.00000,0.00000,1.00000}%
%
%
\label{fig:reconst2}}
\subfloat[][geodetical error on a second-order kinematic track]{
%
%
\definecolor{mycolor1}{rgb}{0.,0.7,0.3}%
\definecolor{mycolor2}{rgb}{1.00000,0.00000,1.00000}%
%
%
\label{fig:geo2}}
 \\
 \subfloat[][reconstruced track: third order kinematics]{
%
%
\definecolor{mycolor2}{rgb}{0.,0.7,0.3}%
\definecolor{mycolor1}{rgb}{1.00000,0.00000,1.00000}%
%
%
\label{fig:reconst3}} 
\subfloat[][geodetical error on a third order kinematic track]{
%
%
\definecolor{mycolor1}{rgb}{0.,0.7,0.3}%
\definecolor{mycolor2}{rgb}{1.00000,0.00000,1.00000}%
%
%
\label{fig:geo3}}
\\
\subfloat[][reconstruced track: fourth order kinematics]{
%
%
\definecolor{mycolor2}{rgb}{0.,0.7,0.3}%
\definecolor{mycolor1}{rgb}{1.00000,0.00000,1.00000}%
%
%
\label{fig:reconst4}}
\subfloat[][geodetical error on a fourth order kinematic track]{
%
%
\definecolor{mycolor1}{rgb}{0.,0.7,0.3}%
\definecolor{mycolor2}{rgb}{1.00000,0.00000,1.00000}%
%
%
\label{fig:geo4}}

\caption{Reconstruction of the camera tracks (left column) and evaluation of the geodetical error w.r.t.\ ground truth (left column) as computed by the proposed filter with kinematics of order 1, 2, 3 and 4.
We evaluated the performance on simulated camera tracks with kinematic models of different orders: 
constant velocity \protect\subref{fig:geo1},\protect\subref{fig:reconst1}, constant acceleration \protect\subref{fig:geo2},\protect\subref{fig:reconst2} as well as third \protect\subref{fig:geo3},\protect\subref{fig:reconst3} and fourth \protect\subref{fig:geo4},\protect\subref{fig:reconst4} order kinematics. In the constant velocity scenario \protect\subref{fig:geo1}, the first-order filter performs best. On the other scenarios \protect\subref{fig:geo2}, \protect\subref{fig:geo3}, \protect\subref{fig:geo4}, the higher-order methods are superior and lead to the best path reconstructions.}
\label{fig:eval-kinematics}
\end{center}
\end{figure*}

\subsection{Evaluation with Realistic Observations}
In order to demonstrate that the minimum energy filter with higher-order state equations also works under real world conditions, we evaluate our approach on the challenging KITTI odometry benchmark~\cite{Geiger2012CVPR}. This benchmark does not contain ground truth data for optical flow, and depth maps can only be obtained from external laser scanners. Thus, we compute optical flow and depth maps in a preprocessing step using the freely available method by Vogel {\em et al.}~\cite{vogel2013d} which only requires image data. 
Although this method is the top ranked method on the KITTI optical flow benchmark, its results still contain relevant deviations from the true solution and thus provide realistic observation noise to evaluate the performance of our proposed filter. As the preprocessed data of~\cite{vogel2013d} is dense, it causes a high computational effort. Therefore, we only use a sparse subset of data points which are selected randomly. In Section~\ref{sec:numberOfObservations} we will show that a small number of observations is sufficient for good reconstructions.

\subsubsection{Quantitative Evaluation of First and higher-order Models}
\label{sec:Quantitative-KITTI}

For our quantitative evaluations on the KITTI benchmark in Table~\ref{tab:kitti-quantitative}, we initialize our 
first~\cite{berger2015second}
and higher-order approaches with the corresponding identity element on the Lie group, i.e.\ $G_{0} = \Id$, and set the corresponding matrices~$P_{0}$ to the identity matrices.
The quadratic forms of the penalty term of the model noise~$\delta$ are set as shown in~\eqref{eq:def-blockdiag} with $s_{1}=10^{-2}$ and $s_{2} = 10^{-5}$.
To increase the influence of the data term, we set the weighting matrix to
\begin{equation}
Q := \tfrac{1}{n}\eins_{2}\,, \quad n=1000\,.
\end{equation}
On the one hand side, this high-weighting leads to less smoothed camera trajectories, but on the other hand side minimizes the observation error, which is desirable for visual odometry applications. For comparison we also present in Table~\ref{tab:kitti-quantitative} the performance measures of the odometry method~\cite{geiger2011stereoscan}.

We emphasize that the first-order approach~\cite{berger2015second} and second-order method from Theorem~\ref{thm:minEnFilter_Acc} perform better in the case of camera motion reconstruction than the proposed higher-order ($>2$) models with generalized kinematics from Theorem~\ref{thm:MEF-higher-order}. The reason for that is that the real camera motion is influenced by model noise, induced by jumps 
 of the camera, to which the first-order method can adapt faster.
Higher-order models smooth the camera trajectories, which in this case is unfortunate. However, they will be beneficial if the actual camera motion behaves according to the models, as shown in the experiments in Section ~\ref{sec:eval-higher-order}.

Please note that our method currently is not designed to be robust against outliers in the observation. 
In contrast, the approach of Geiger {\em et al.}~\cite{geiger2011stereoscan} uses additional precautions to eliminate violation of the assumption of a single rigid body motion, see e.g.\ sequence~3 in Table~\ref{tab:kitti-quantitative}.

\begin{table*}[htdp]
\caption{Quantitative evaluation of rotational (in degrees) and translational (in meters) error on the first 200 frames of the training set of the KITTI odometry benchmark. We compared the proposed higher-order method (i.e.\ 2nd to 4th) with our first-order method from~\cite{berger2015second}. As a reference method, we also evaluated the approach by Geiger {\em et al.}~\cite{geiger2011stereoscan}. The first and second-order methods outperform the higher-order methods since they can fit more easily to the non-smooth ego-motion data.
}
\begin{center}
\begin{tabularx}{\textwidth}{l l @{\extracolsep{\fill}}ccccccccccc} \toprule
& sequence &  00 & 01 & 02 & 03 & 04 & 05 & 06 & 07 & 08 & 09 & 10   \\ \midrule
 \multirow{5}{*}[-2pt]{\begin{sideways}trans.~error\end{sideways}} & (Geiger~\cite{geiger2011stereoscan}) & \textbf{0.0272} & \textbf{0.0572} & 0.0255 & \textbf{0.0175} & \textbf{0.0161} & \textbf{0.0185} & \textbf{0.0118} & \textbf{0.0160} & 0.1166 & \textbf{0.0175} & \textbf{0.0147} \\
& 1st order \cite{berger2015second} & 0.0284  & 0.0759 & \textbf{0.0188} & 0.0804 & 0.0165 & 0.0188 & 0.0122 & 0.0174 & \textbf{0.1142} & 0.0193  & 0.0205 \\
& 2nd order& 0.0356 & 0.0786 & 0.0289 & 0.0938 & 0.0210 & 0.0288 & 0.0153 & 0.0284 & 0.1153  & 0.0293 & 0.0417 \\
& 3rd order&  0.0358 & 0.0784 & 0.0290 &  0.0924 & 0.0216 & 0.0286 & 0.0175 & 0.0268 & 0.1153  & 0.0258 & 0.0342 \\
& 4th order&  0.0347 & 0.0782 & 0.0275 & 0.0918  & 0.0211 & 0.0277 & 0.0140 & 0.0257 & 0.1155 & 0.0240 & 0.0317 \\ \midrule
 \multirow{5}{*}[-2pt]{\begin{sideways}rot.~error\end{sideways}}  & (Geiger~\cite{geiger2011stereoscan}) & \textbf{0.1773} & \textbf{0.1001} & 0.1552 & \textbf{0.1829} & 0.0970 & 0.1539 & 0.0829 & 0.1770 & 0.1589 & 0.1166 & 0.2001 \\
& 1st order \cite{berger2015second} & \textbf{0.1773} & 0.1139 & 0.1504 & 0.2246 & 0.0836 & \textbf{0.1454} & 0.0765 & \textbf{0.1654} & \textbf{0.1444} & \textbf{0.0911} & 0.1829 \\
& 2nd order&  0.1996 & 0.1183 & \textbf{0.1430} & 0.2448 & \textbf{0.0805} & 0.1566 & \textbf{0.0703} & 0.2113 & 0.1676 & 0.1167 & 0.2388 \\
& 3rd order&  0.2402 & 0.1348 & 0.1872& 0.2719 & 0.1090 & 0.1971 & 0.0875 & 0.2362 & 0.2053 & 0.1335 & 0.2628 \\
& 4th order& 0.2795 & 0.1466 & 0.2223 & 0.3120 & 0.1479 & 0.2335  & 0.1045 & 0.2709 & 0.2318 & 0.1630 & 0.2956 \\ \bottomrule
\end{tabularx}
\end{center}
\label{tab:kitti-quantitative}
\end{table*}%

\subsubsection{Determination of Optimal Number of Observations}\label{sec:numberOfObservations}
Since the evaluation of the functions $A_{k}$ and $D_{k}$ in Theorem~\ref{thm:minEnFilter_Acc} as well as the accurate numerical integration in Section~\ref{sec:numInt} are expensive, we are looking for a good trade-off between the number of required measurements and accuracy. In Table~\ref{tab:evalNumObs} we evaluate the geodetical error for a different number of observations $n$. For $n=1$, our proposed filters do not converge since they are numerically instable. For $n=5,\dots,20$, the geodetical error is fairly small but reaches a minimum for $n=50$.  For $n<5$, the error increases because the ego-motion cannot be reconstructed uniquely (cf.\ Five-point-algorithm~\cite{nister2004efficient}). Likewise, for $n>50$, the error rises due to noisy measurements averaged by the filter.

\begin{table}[tbp]
\caption{Determination of the optimal number of measurements $n$. We evaluated the mean geodetical error our filter with different kinematic models (first to fourth order) on a short sequence (10 frames) for different numbers $n$ of observations. Since the $n$ observations are selected randomly, we repeated the experiment 50 times and averaged finally, to find a representative value. We found an optimal number of measurements for $n=50$.
}
\begin{center}
\begin{tabularx}{\columnwidth}{l @{\extracolsep{\fill}}cccc} \toprule
$n$ 	& 1st order      &	2nd order & 	3rd order & 	4th order \\ \midrule
1000	&	0.1205    &  	0.1361     & 	0.1311     & 	0.1290 \\
500  &	0.1070    &	0.1174     & 	0.1116     & 	0.1096 \\
200  &	0.0915    &	0.0945     & 	0.0902    &	0.0890 \\
100  &	0.0764    &	0.0764     &	0.0739    & 	0.0733 \\
\textbf{50}    &	\textbf{0.0667}    &	\textbf{0.0651}     &	\textbf{0.0638}    &	\textbf{0.0637} \\
20    & 	0.0715    &	0.0703     &	0.0687    &	0.0684 \\
15    & 	0.0709    &	0.0691     &	0.0674    &	0.0672 \\
12    & 	0.0718    &	0.0720     &	0.0702    &	0.0699 \\
10    & 	0.0749    &	0.0735     &	0.0716    &	0.0712 \\
9      &	0.0751    &	0.0747     &	0.0726    &	0.0722 \\
8      & 	0.0772    &	0.0762     &	0.0742    &	0.0738 \\
7      &	0.0735    &	0.0733     &	0.0717    &	0.0714 \\
6      &	0.0786    &	0.0776     &	0.0757    &	0.0753 \\
5      &	0.0789    &	0.0797     &	0.0778    &	0.0774 \\
4      &	0.0856    &	0.0859     &	0.0837    &	0.0831 \\
3      &	0.0917    &	0.0951     &	0.0928    & 	0.0921 \\
2      &	0.1005    &	0.1085     &	0.1058    &  	0.1051 \\ \bottomrule
\end{tabularx}
\end{center}
\label{tab:evalNumObs}
\end{table}

\subsubsection{Influence of the Decay Rate $\alpha$}
In real sequences, the motion is usually not uniform and changes due to acceleration and curves. As demonstrated earlier, higher-order state equations that model accelerations, jerks, etc.\ usually converge faster and yield a better accuracy. However, higher-order models are delayed since it takes some time until the information from the observation is transported to the lowest layer. Furthermore, if the motion changes quickly, then higher-order models will still propagate wrong kinematics. For this reason, in \cite{saccon2013second} a decay $\alpha>0$ rate is introduced and also adopted to our model. For $\alpha=0$, all past information is preserved in the propagation within the filter. For larger values of $\alpha$,  old information about the trajectory has lower influence on the filter and is less respected in future.

For the experiments we use the weighting matrix $Q =  n^{-1} \eins_{2}$, where $n$ is the number of measurements. Furthermore, we use $S$ as in \eqref{eq:def-blockdiag} with the values $s_{1}=5\cdot 10^{-2}$, $s_{2}=5 \cdot 10^{-4}$. The integration step size is set to $\delta=1/50$.

In Fig.~\ref{fig:eval-alpha}, we visualize the influence of different values of $\alpha$ on the minimum energy filters of order 1 to 4. For small decay rates $\alpha$, the filters will converge faster over time, but will also cause errors if the kinematics change. On the other hand, large decay rates adapt more easily to spontaneous changes of kinematics. The filters take longer to converge, however.

\begin{figure*}[htbp]
\begin{center}
\subfloat[][$\alpha=1$]{
%
%
\definecolor{mycolor1}{rgb}{0.,0.7,0.3}%
\definecolor{mycolor2}{rgb}{1.00000,0.00000,1.00000}%
\begin{tikzpicture}

\begin{axis}[%
width=3in,
height=1.5in,
at={(0.771875in,0.483542in)},
scale only axis,
xmin=0,
xmax=50,
ymode=log,
ymin=0.001,
ymax=1,
yminorticks=true,
]
\addplot [color=mycolor1,solid,mark=x,mark options={solid},mark repeat={10}]
  table[row sep=crcr]{%
1	0.819110806857404\\
2	0.686737278925937\\
3	0.450375331578959\\
4	0.178088719915083\\
5	0.0445051328366336\\
6	0.105123625299282\\
7	0.090079625121361\\
8	0.0778716286090705\\
9	0.0573013762453419\\
10	0.0467564854000374\\
11	0.075494937792353\\
12	0.0646064543147729\\
13	0.0444846219757573\\
14	0.0271538025019272\\
15	0.0172840972789903\\
16	0.0292432427346622\\
17	0.0179015918754259\\
18	0.0389292942970995\\
19	0.0288289930132862\\
20	0.0282816974477703\\
21	0.012648905568619\\
22	0.0209869539537162\\
23	0.020254868243973\\
24	0.0295631366204318\\
25	0.0337231074551585\\
26	0.0384296290540214\\
27	0.0257764112354786\\
28	0.0157940801261315\\
29	0.0482453038076423\\
30	0.0509367748133315\\
31	0.0160075510509888\\
32	0.0343046592947919\\
33	0.0689331959535537\\
34	0.0709459968509088\\
35	0.0459678918213854\\
36	0.0451308362042751\\
37	0.0487298000590443\\
38	0.0132267288837568\\
39	0.0530093954200023\\
40	0.0766538433156658\\
41	0.0693573942023762\\
42	0.062539246900583\\
43	0.0370367418934619\\
44	0.0383286854657516\\
45	0.027072881734635\\
46	0.0295234181948668\\
47	0.0262641285878232\\
48	0.0261215031960973\\
49	0.0444671768383698\\
50	0.0320363309453775\\
};
\addplot [color=blue,solid,mark=+,mark options={solid},mark repeat={10}]
  table[row sep=crcr]{%
1	0.819139681761788\\
2	0.689713730028489\\
3	0.47443678339541\\
4	0.233184120767334\\
5	0.0452043597659059\\
6	0.0900868143561742\\
7	0.114592823396447\\
8	0.13249635039855\\
9	0.118351232687633\\
10	0.0866668964897817\\
11	0.0231749377055107\\
12	0.0475724911626145\\
13	0.06400052865852\\
14	0.0470691573998915\\
15	0.0259442164182929\\
16	0.027999479005179\\
17	0.0210150586329033\\
18	0.0278775710833492\\
19	0.0217181578617736\\
20	0.0356209431322434\\
21	0.0304104642625171\\
22	0.00859922688178914\\
23	0.0189233297017372\\
24	0.0423379883465118\\
25	0.0204804703507964\\
26	0.0376069991433635\\
27	0.0463110928607881\\
28	0.0268713252964637\\
29	0.0173888054131146\\
30	0.0380335929071881\\
31	0.0498048151547414\\
32	0.0193505209320186\\
33	0.0254485002842277\\
34	0.095738978233309\\
35	0.0408179055148061\\
36	0.0205495414901375\\
37	0.0568212586940754\\
38	0.0210540330551759\\
39	0.0391734576372261\\
40	0.0624330128889323\\
41	0.0658434387195693\\
42	0.0614481119418851\\
43	0.0415280037571739\\
44	0.0435225476929117\\
45	0.0426041717518201\\
46	0.028434752127796\\
47	0.0207151896189023\\
48	0.0295403001625965\\
49	0.0300471014224594\\
50	0.0341339106541562\\
};
\addplot [color=mycolor2,solid,mark=square,mark options={solid},mark repeat={10}]
  table[row sep=crcr]{%
1	0.81967533282233\\
2	0.704951131926489\\
3	0.540955608393818\\
4	0.362925619576802\\
5	0.196985619056068\\
6	0.06438884540078\\
7	0.0516166033912834\\
8	0.120105730665451\\
9	0.175348896392852\\
10	0.21340267192264\\
11	0.185891182230674\\
12	0.152312731776152\\
13	0.101596492617606\\
14	0.0737600852979877\\
15	0.0521727864714981\\
16	0.0220961047290534\\
17	0.0210006062766515\\
18	0.0440119603583738\\
19	0.0475228630075631\\
20	0.0371424423495869\\
21	0.0283409445749512\\
22	0.0294459557620823\\
23	0.013669473773798\\
24	0.0240145688269931\\
25	0.0226393166439324\\
26	0.0277251173371002\\
27	0.0417602731885803\\
28	0.047300015999846\\
29	0.054759816834935\\
30	0.0275064232002275\\
31	0.0294576380031186\\
32	0.0489989690199059\\
33	0.0462934093542422\\
34	0.02451588970659\\
35	0.048255245475286\\
36	0.0388892914317642\\
37	0.0196160040633376\\
38	0.0359720799239718\\
39	0.0345003553458923\\
40	0.0496068339395049\\
41	0.0464639746513192\\
42	0.0461290685911007\\
43	0.0392513274446485\\
44	0.0412470298954152\\
45	0.0580692354282929\\
46	0.0475449352982883\\
47	0.0214892720391329\\
48	0.0257497237894068\\
49	0.027731039625738\\
50	0.0206707462899135\\
};
\addplot [color=red,solid,mark=diamond,mark options={solid},mark repeat={10}]
  table[row sep=crcr]{%
1	0.826048540424457\\
2	0.760930396670387\\
3	0.693365790651757\\
4	0.628499682030286\\
5	0.570871340851576\\
6	0.51727412023786\\
7	0.457169321682266\\
8	0.41370443140578\\
9	0.375659946561482\\
10	0.336353410295913\\
11	0.29876088239943\\
12	0.266604874881828\\
13	0.231857835906883\\
14	0.199793776578307\\
15	0.171348339467651\\
16	0.157362096385055\\
17	0.101402757189702\\
18	0.0890013378486751\\
19	0.0867022283329267\\
20	0.0820714543867389\\
21	0.0625843185266124\\
22	0.0517680081798019\\
23	0.0413186701214014\\
24	0.0291526966664221\\
25	0.0322436327979921\\
26	0.0315941710793751\\
27	0.02689512878013\\
28	0.0307457946731469\\
29	0.0324537720502575\\
30	0.0321193739888374\\
31	0.0244421903533814\\
32	0.0186029010988193\\
33	0.0270267348391422\\
34	0.0199541050577214\\
35	0.0272903220575576\\
36	0.034310418553293\\
37	0.0241830028092413\\
38	0.0152830078169105\\
39	0.0254337566729754\\
40	0.0721015717232633\\
41	0.051976241281535\\
42	0.0525962200313343\\
43	0.0328697751827286\\
44	0.0367251223797231\\
45	0.0355982875506123\\
46	0.0381045907867074\\
47	0.0231247013925955\\
48	0.0170389797267715\\
49	0.0117155187945456\\
50	0.00979094030144879\\
};
\end{axis}
\end{tikzpicture}%
\label{fig:alpha1}}
\subfloat[][$\alpha=2$]{
%
%
\definecolor{mycolor1}{rgb}{0.,0.7,0.3}%
\definecolor{mycolor2}{rgb}{1.00000,0.00000,1.00000}%
\begin{tikzpicture}

\begin{axis}[%
width=3in,
height=1.5in,
at={(0.771875in,0.483542in)},
scale only axis,
xmin=0,
xmax=50,
ymode=log,
ymin=0.001,
ymax=1,
yminorticks=true,
]
\addplot [color=mycolor1,solid,mark=x,mark options={solid},forget plot,mark repeat={10}]
  table[row sep=crcr]{%
1	0.82054875275133\\
2	0.7251407501345\\
3	0.584237074543593\\
4	0.421322315274644\\
5	0.258147532337068\\
6	0.108533778097832\\
7	0.0342206944394199\\
8	0.122073204357002\\
9	0.171622647697326\\
10	0.195196457031112\\
11	0.19885701368615\\
12	0.180417794560627\\
13	0.143495555913327\\
14	0.108393963512945\\
15	0.0716778471429441\\
16	0.0375719302768269\\
17	0.0255928502447108\\
18	0.049981078288944\\
19	0.0624971918724686\\
20	0.0696830327587358\\
21	0.0674201369591682\\
22	0.0704481136856109\\
23	0.0570224466301947\\
24	0.0207049359322808\\
25	0.00843287978907918\\
26	0.0193844430162578\\
27	0.0363828844086903\\
28	0.037764808181194\\
29	0.0430121487582023\\
30	0.0349787286231239\\
31	0.0168313049738186\\
32	0.0130723697039522\\
33	0.0215387008974324\\
34	0.00929956567178212\\
35	0.023972917162039\\
36	0.0253496434111823\\
37	0.0111404714894789\\
38	0.0209127058162947\\
39	0.0264230685786356\\
40	0.0477172903749697\\
41	0.0345866980616225\\
42	0.0374968178265437\\
43	0.0362063498190782\\
44	0.0346231431130673\\
45	0.0413660347047617\\
46	0.0461195418044706\\
47	0.0300579380422413\\
48	0.0243901595775093\\
49	0.0261675704987216\\
50	0.0227773380028314\\
};
\addplot [color=blue,solid,mark=+,mark options={solid},forget plot,mark repeat={10}]
  table[row sep=crcr]{%
1	0.820567609949148\\
2	0.726329219203199\\
3	0.592073772087995\\
4	0.441638766346897\\
5	0.292380189411363\\
6	0.152152794974472\\
7	0.0388246661314053\\
8	0.0811050292606888\\
9	0.146401656748815\\
10	0.188497496391609\\
11	0.212979698739648\\
12	0.214073775622815\\
13	0.196544736052822\\
14	0.171080808141849\\
15	0.140892094239172\\
16	0.102242877608284\\
17	0.0584997569891227\\
18	0.0397793497628917\\
19	0.0303286610381144\\
20	0.0407415919891817\\
21	0.0504148038256358\\
22	0.0639905456338454\\
23	0.0626396824250312\\
24	0.0352432720754799\\
25	0.0225570569227884\\
26	0.00389298761289134\\
27	0.0156117633858739\\
28	0.0226888285035945\\
29	0.0314891560579702\\
30	0.0323040696553813\\
31	0.0218754202429536\\
32	0.00948727304707528\\
33	0.00408396232879531\\
34	0.00310815860577346\\
35	0.0223437621111718\\
36	0.0210377118444324\\
37	0.013723798982162\\
38	0.0185716907602088\\
39	0.0261405458070838\\
40	0.0483394067106549\\
41	0.0321137322965443\\
42	0.0338353072527438\\
43	0.0326384229179884\\
44	0.0303060101030926\\
45	0.0362164634418762\\
46	0.0395866010680628\\
47	0.0275623515243054\\
48	0.0251154632738956\\
49	0.0264781760541987\\
50	0.0240037518559576\\
};
\addplot [color=mycolor2,solid,mark=square,mark options={solid},forget plot,mark repeat={10}]
  table[row sep=crcr]{%
1	0.820950410038627\\
2	0.733737109885417\\
3	0.620231980274768\\
4	0.498471757699783\\
5	0.378360457312768\\
6	0.260384570930473\\
7	0.150894549714868\\
8	0.0515933180417824\\
9	0.0504215951654827\\
10	0.117368179084302\\
11	0.173048722900871\\
12	0.209246557068345\\
13	0.229304272773005\\
14	0.236756698833052\\
15	0.235628221896628\\
16	0.216861050555034\\
17	0.189224988095362\\
18	0.161899015421563\\
19	0.133468136411242\\
20	0.110968266836159\\
21	0.0916247357761498\\
22	0.0721880626461516\\
23	0.0578650186859703\\
24	0.0469634837560893\\
25	0.0413810688978336\\
26	0.0461182422006141\\
27	0.0393464740528962\\
28	0.0363194428567054\\
29	0.0300542980947672\\
30	0.0294107351001904\\
31	0.0277252817388412\\
32	0.0298911682088355\\
33	0.03363541616867\\
34	0.0343449049277374\\
35	0.0311516570250982\\
36	0.0305837744383147\\
37	0.0138830462692071\\
38	0.00505803022133815\\
39	0.0231928644305537\\
40	0.0562438008289699\\
41	0.0340622967863464\\
42	0.0352605923987297\\
43	0.0373426767556177\\
44	0.0341465993823011\\
45	0.0437216187851699\\
46	0.0418854043812565\\
47	0.037196615955612\\
48	0.0320856613517993\\
49	0.0220535202161765\\
50	0.0183043136357404\\
};
\addplot [color=red,solid,mark=diamond,mark options={solid},forget plot,mark repeat={10}]
  table[row sep=crcr]{%
1	0.826258392908405\\
2	0.770533873954496\\
3	0.713767367062766\\
4	0.661873699499507\\
5	0.617153906447673\\
6	0.572972760442228\\
7	0.529226459821335\\
8	0.484028942715512\\
9	0.426945072840575\\
10	0.382201985683166\\
11	0.347822839829709\\
12	0.318359906100377\\
13	0.287818425987235\\
14	0.257359195736416\\
15	0.226225459141243\\
16	0.206026625159003\\
17	0.187104559411535\\
18	0.170082118268259\\
19	0.160835793129858\\
20	0.151766855222247\\
21	0.142365683349141\\
22	0.137350691639366\\
23	0.12321217143741\\
24	0.0892252394983445\\
25	0.0711148079068049\\
26	0.0551455896431689\\
27	0.0439588407302388\\
28	0.0415274051324592\\
29	0.0382793934535437\\
30	0.038222646848741\\
31	0.0355544995130547\\
32	0.0369106014940716\\
33	0.0418111391030943\\
34	0.0399974002988465\\
35	0.040428194735016\\
36	0.0395325396717174\\
37	0.0329289742216682\\
38	0.0334539236002852\\
39	0.0510219894425423\\
40	0.0929027410275445\\
41	0.0666553757315326\\
42	0.0690258403340913\\
43	0.0509460644437312\\
44	0.0507737643777646\\
45	0.0337607047055368\\
46	0.0315400139863261\\
47	0.0176374949751009\\
48	0.0161552566029074\\
49	0.0161301753792549\\
50	0.0129456205031701\\
};
\end{axis}
\end{tikzpicture}%
\label{fig:alpha2}} \\
\subfloat[][$\alpha=4$]{
%
%
\definecolor{mycolor1}{rgb}{0.,0.7,0.3}%
\definecolor{mycolor2}{rgb}{1.00000,0.00000,1.00000}%
\begin{tikzpicture}

\begin{axis}[%
width=3in,
height=1.5in,
at={(0.771875in,0.483542in)},
scale only axis,
xmin=0,
xmax=50,
ymode=log,
ymin=0.001,
ymax=1,
yminorticks=true,
]
\addplot [color=mycolor1,solid,mark=x,mark options={solid},forget plot,mark repeat={10}]
  table[row sep=crcr]{%
1	0.812262608170916\\
2	0.740100314764533\\
3	0.664533659829572\\
4	0.575493052290553\\
5	0.483880412294086\\
6	0.394296883634613\\
7	0.3095859588845\\
8	0.225420108073038\\
9	0.148549480666781\\
10	0.0783471331180806\\
11	0.0170458423716538\\
12	0.0561146532701682\\
13	0.107396339432735\\
14	0.153730722149856\\
15	0.198454794198863\\
16	0.22286999602373\\
17	0.244928888202313\\
18	0.262600335028934\\
19	0.272445297666711\\
20	0.279019420928776\\
21	0.286152746453286\\
22	0.292405020192784\\
23	0.298728251803779\\
24	0.304523022325284\\
25	0.30112138915933\\
26	0.302573670122416\\
27	0.302610843397952\\
28	0.289818153151261\\
29	0.278593711167513\\
30	0.264028764331439\\
31	0.251666266013372\\
32	0.231039878237248\\
33	0.208676026553144\\
34	0.183504298424277\\
35	0.161985007862305\\
36	0.134890569421065\\
37	0.118946281995042\\
38	0.0968367703972041\\
39	0.0672412503313942\\
40	0.0555281204986606\\
41	0.0390117208913724\\
42	0.0453564763080756\\
43	0.0347666485993879\\
44	0.0440361809089731\\
45	0.0388853474389784\\
46	0.0488409555499643\\
47	0.0541125413487427\\
48	0.0540833055015042\\
49	0.0595658187556218\\
50	0.066626773979346\\
};
\addplot [color=blue,solid,mark=+,mark options={solid},forget plot,mark repeat={10}]
  table[row sep=crcr]{%
1	0.81227198737438\\
2	0.74031975146175\\
3	0.665337066060263\\
4	0.577198530893916\\
5	0.486692137862809\\
6	0.398316935739736\\
7	0.314798494798819\\
8	0.231740013907677\\
9	0.155832760344676\\
10	0.0859382750256029\\
11	0.021930733065823\\
12	0.0480907227652365\\
13	0.0998589365279242\\
14	0.147057925652409\\
15	0.19270932150622\\
16	0.217970666489612\\
17	0.240745636795724\\
18	0.259049991206363\\
19	0.269641057526688\\
20	0.277010185667853\\
21	0.285089199378427\\
22	0.292408184461259\\
23	0.299816096799195\\
24	0.306638203860973\\
25	0.304211503934226\\
26	0.30657763809327\\
27	0.30760716259559\\
28	0.295833367688699\\
29	0.285610821073338\\
30	0.272061306326062\\
31	0.260667015387482\\
32	0.240851148869616\\
33	0.21927484392953\\
34	0.19479019828716\\
35	0.173860282128234\\
36	0.14722696235563\\
37	0.131495392561457\\
38	0.109532373292824\\
39	0.0789403640137068\\
40	0.0590260889055768\\
41	0.0424816901849826\\
42	0.0435605403421907\\
43	0.031640511378123\\
44	0.0381577971255416\\
45	0.0341649364129648\\
46	0.0450294674531067\\
47	0.0511090220189535\\
48	0.0519505038906713\\
49	0.0584980989532495\\
50	0.0669392888939931\\
};
\addplot [color=mycolor2,solid,mark=square,mark options={solid},forget plot,mark repeat={10}]
  table[row sep=crcr]{%
1	0.812506760453429\\
2	0.742570855588335\\
3	0.671455483951161\\
4	0.58855986159265\\
5	0.504187914640476\\
6	0.4223920051059\\
7	0.345317464373859\\
8	0.268244598503202\\
9	0.197574972302452\\
10	0.130016564839725\\
11	0.0651568516676631\\
12	0.0171058134791875\\
13	0.057837829248681\\
14	0.108049973729363\\
15	0.157849468446575\\
16	0.186723599178537\\
17	0.21258294878189\\
18	0.233461390007247\\
19	0.247285257110833\\
20	0.258221399574248\\
21	0.270847759666741\\
22	0.283386050111785\\
23	0.296140808341775\\
24	0.307991629353779\\
25	0.310369712299216\\
26	0.317241907340791\\
27	0.323291519763002\\
28	0.316806660723373\\
29	0.311952820246825\\
30	0.304000861793499\\
31	0.297920138958197\\
32	0.282847883297547\\
33	0.266256033666404\\
34	0.246163884680512\\
35	0.229042796945214\\
36	0.20583280828096\\
37	0.192513026517442\\
38	0.172449747133129\\
39	0.140052886989856\\
40	0.101272844538974\\
41	0.0873899541171036\\
42	0.0705465314380256\\
43	0.0560895883392862\\
44	0.0374378784573322\\
45	0.0335703233494281\\
46	0.0315471771531577\\
47	0.0331161340777865\\
48	0.0340334084246041\\
49	0.0433497638599961\\
50	0.0571952968164948\\
};
\addplot [color=red,solid,mark=diamond,mark options={solid},forget plot,mark repeat={10}]
  table[row sep=crcr]{%
1	0.817092464066622\\
2	0.764789148538259\\
3	0.720829672839145\\
4	0.672788557513289\\
5	0.629468930452707\\
6	0.592669198240352\\
7	0.561398987202121\\
8	0.529525167723832\\
9	0.501092102145967\\
10	0.464294054225102\\
11	0.429566325928266\\
12	0.402701112347911\\
13	0.364724227807375\\
14	0.327931511709988\\
15	0.301979924199167\\
16	0.28583409127962\\
17	0.27767875880802\\
18	0.270169222784454\\
19	0.268437030589675\\
20	0.263861292160952\\
21	0.256935677470148\\
22	0.252350972122303\\
23	0.244923225202359\\
24	0.237180469290726\\
25	0.235583375664656\\
26	0.228214421221037\\
27	0.217816554921346\\
28	0.213087476623036\\
29	0.205265670436045\\
30	0.199656410283111\\
31	0.192219549728363\\
32	0.188838746652902\\
33	0.1932102685335\\
34	0.190574052856902\\
35	0.185467691358691\\
36	0.18388566578741\\
37	0.17366112411167\\
38	0.166866135525155\\
39	0.172106065162304\\
40	0.202556353341781\\
41	0.170882230449981\\
42	0.165858177288003\\
43	0.138624163881847\\
44	0.132352360041179\\
45	0.102655649058435\\
46	0.0912493331873553\\
47	0.0769535014902978\\
48	0.0654404849575724\\
49	0.0622358545835309\\
50	0.0593786071845636\\
};
\end{axis}
\end{tikzpicture}%
\label{fig:alpha4}}
\subfloat[][$\alpha=8$]{
%
%
\definecolor{mycolor1}{rgb}{0.,0.7,0.3}%
\definecolor{mycolor2}{rgb}{1.00000,0.00000,1.00000}%
\begin{tikzpicture}

\begin{axis}[%
width=3in,
height=1.5in,
at={(0.771875in,0.483542in)},
scale only axis,
xmin=0,
xmax=50,
ymode=log,
ymin=0.001,
ymax=1,
yminorticks=true,
legend style={at={(0.,0.51)},anchor=north west,legend cell align=left,align=left,draw=black, font=\scriptsize}
]
\addplot [color=mycolor1,solid,mark=x,mark options={solid},mark repeat={10}]
  table[row sep=crcr]{%
1	0.832623791440766\\
2	0.802323235869548\\
3	0.778103630765886\\
4	0.753466182700423\\
5	0.729339272475603\\
6	0.702026001607519\\
7	0.676183755076265\\
8	0.646879823121408\\
9	0.619540724208343\\
10	0.587137312208294\\
11	0.551797812534385\\
12	0.518735735968859\\
13	0.486190467017471\\
14	0.45377668725291\\
15	0.420624751291729\\
16	0.396625775046353\\
17	0.374862629150222\\
18	0.351771621327539\\
19	0.336032273552392\\
20	0.317843193351455\\
21	0.299283903869907\\
22	0.283584699141193\\
23	0.268700242822041\\
24	0.252445083436409\\
25	0.24138349542356\\
26	0.224233409572084\\
27	0.210605061213806\\
28	0.20258624648941\\
29	0.194686587335282\\
30	0.190189305525783\\
31	0.180493644063988\\
32	0.18021020173335\\
33	0.180265198031783\\
34	0.177431231455203\\
35	0.17087772439609\\
36	0.16740262184618\\
37	0.151106050448369\\
38	0.118507326196071\\
39	0.0844866539626363\\
40	0.0584804054509082\\
41	0.0518260968119949\\
42	0.0553680547846321\\
43	0.0575649228231918\\
44	0.0568698121076194\\
45	0.0680385735478866\\
46	0.075614306999014\\
47	0.094347655376415\\
48	0.109954443224017\\
49	0.12072731399021\\
50	0.130358202447446\\
};
\addlegendentry{4th order model};
\addplot [color=blue,solid,mark=+,mark options={solid},mark repeat={10}]
  table[row sep=crcr]{%
1	0.832624760243145\\
2	0.802331170113454\\
3	0.778123545152627\\
4	0.753501950019592\\
5	0.729393728144084\\
6	0.702101607029759\\
7	0.676282976246462\\
8	0.647003980349095\\
9	0.619691751281511\\
10	0.587315606646853\\
11	0.552004968461208\\
12	0.51897731261682\\
13	0.486462344594112\\
14	0.454076654018573\\
15	0.420945932513415\\
16	0.39697156225734\\
17	0.375229392022843\\
18	0.352164975893307\\
19	0.336441977123378\\
20	0.318271135903195\\
21	0.299721790812475\\
22	0.284006490019851\\
23	0.26910450933949\\
24	0.252836134545054\\
25	0.241768639859624\\
26	0.224617241199364\\
27	0.210960359534485\\
28	0.202913496062142\\
29	0.194956291987999\\
30	0.190411043592671\\
31	0.180684112701108\\
32	0.180390123847892\\
33	0.180444012233147\\
34	0.17760954202683\\
35	0.171047306223787\\
36	0.167583163838201\\
37	0.151276694789948\\
38	0.118661001130575\\
39	0.0846905675443677\\
40	0.0587356478428279\\
41	0.0520516783958106\\
42	0.055578250801452\\
43	0.0577377708026918\\
44	0.0569894726216338\\
45	0.0680540756930299\\
46	0.0756097079142863\\
47	0.0943453163066872\\
48	0.109910142538215\\
49	0.120696083191111\\
50	0.130326842917554\\
};
\addlegendentry{3rd order model};
\addplot [color=mycolor2,solid,mark=square,mark options={solid},mark repeat={10}]
  table[row sep=crcr]{%
1	0.832664508117982\\
2	0.80254113084118\\
3	0.778595444707966\\
4	0.754308653583176\\
5	0.73058562614239\\
6	0.703722740927881\\
7	0.678378605340203\\
8	0.649595776968036\\
9	0.622815520440381\\
10	0.590975668304964\\
11	0.556231229143101\\
12	0.523882959693335\\
13	0.491960725267931\\
14	0.460123123229867\\
15	0.427405933389921\\
16	0.403913493195287\\
17	0.382583139446542\\
18	0.360041671873522\\
19	0.344640520846644\\
20	0.326820463631354\\
21	0.308454554486454\\
22	0.292430538692517\\
23	0.277191885148575\\
24	0.260671072597744\\
25	0.249489848546782\\
26	0.23231724103735\\
27	0.218105209123364\\
28	0.209498946896481\\
29	0.200390438464671\\
30	0.194866080270288\\
31	0.18448374905592\\
32	0.183922330395997\\
33	0.183906279739992\\
34	0.181020973294178\\
35	0.174259240031499\\
36	0.171009210665704\\
37	0.154524208265185\\
38	0.121637549802337\\
39	0.0887081323495109\\
40	0.063874903824477\\
41	0.0566447061749278\\
42	0.0599423968395681\\
43	0.0614527701029971\\
44	0.0597698674248046\\
45	0.068777191544816\\
46	0.075924105256094\\
47	0.0946037019840375\\
48	0.109280659884613\\
49	0.120292680491716\\
50	0.129872506123243\\
};
\addlegendentry{2nd order model};
\addplot [color=red,solid,mark=diamond,mark options={solid},mark repeat={10}]
  table[row sep=crcr]{%
1	0.834212353618096\\
2	0.808521899250334\\
3	0.79094174230545\\
4	0.774615100722044\\
5	0.759875400681348\\
6	0.742935322115652\\
7	0.728486162319626\\
8	0.711043723202766\\
9	0.696460644471381\\
10	0.677197700225641\\
11	0.656189697146332\\
12	0.641525852255456\\
13	0.625015027870824\\
14	0.607879367854536\\
15	0.588310910951465\\
16	0.579721717900424\\
17	0.57260574043025\\
18	0.566737453058771\\
19	0.564555480112627\\
20	0.556641080935473\\
21	0.542981049712469\\
22	0.532571337280807\\
23	0.521865118533387\\
24	0.511450322573372\\
25	0.50786373300195\\
26	0.50026975562665\\
27	0.486209869983323\\
28	0.473904835445385\\
29	0.447222843236083\\
30	0.419458093145913\\
31	0.39151921961896\\
32	0.370874798842511\\
33	0.35508371171999\\
34	0.337864197623142\\
35	0.324486430152998\\
36	0.318886248229518\\
37	0.304013685701537\\
38	0.294950964612427\\
39	0.302033464253978\\
40	0.331470497630286\\
41	0.307879551857246\\
42	0.303837639538664\\
43	0.28493318632312\\
44	0.2824397504388\\
45	0.255447607768064\\
46	0.24720377026837\\
47	0.234003532767656\\
48	0.224740038533723\\
49	0.222174139460211\\
50	0.217448792990116\\
};
\addlegendentry{1st order model};
\end{axis}
\end{tikzpicture}%
\label{fig:alpha8}}
\caption{Evaluation of the translational error (in meters) of the minimum energy filter regarding the first, second, third and fourth order state equation on the first 50 frames of sequence~0 of the KITTI odometry sequence. For small values of $\alpha,$ the filter memorizes past information and converges fast, see~Fig.~(\ref{fig:alpha1}). Although higher-order filters converge faster, they cause oscillation due to the time delay that is required to propagate information into higher-order derivatives of the kinematics. Since for large values of $\alpha$ past information is neglected, the filters converge slower and the difference between second, third and fourth order models become smaller, while the oscillations disappear. Please note that for this experiments the weighting matrices $S$ and $Q$ are kept fixed. To further reduce the error for large $\alpha$ we propose to adapt the weights.
}
\label{fig:eval-alpha}
\end{center}
\end{figure*}
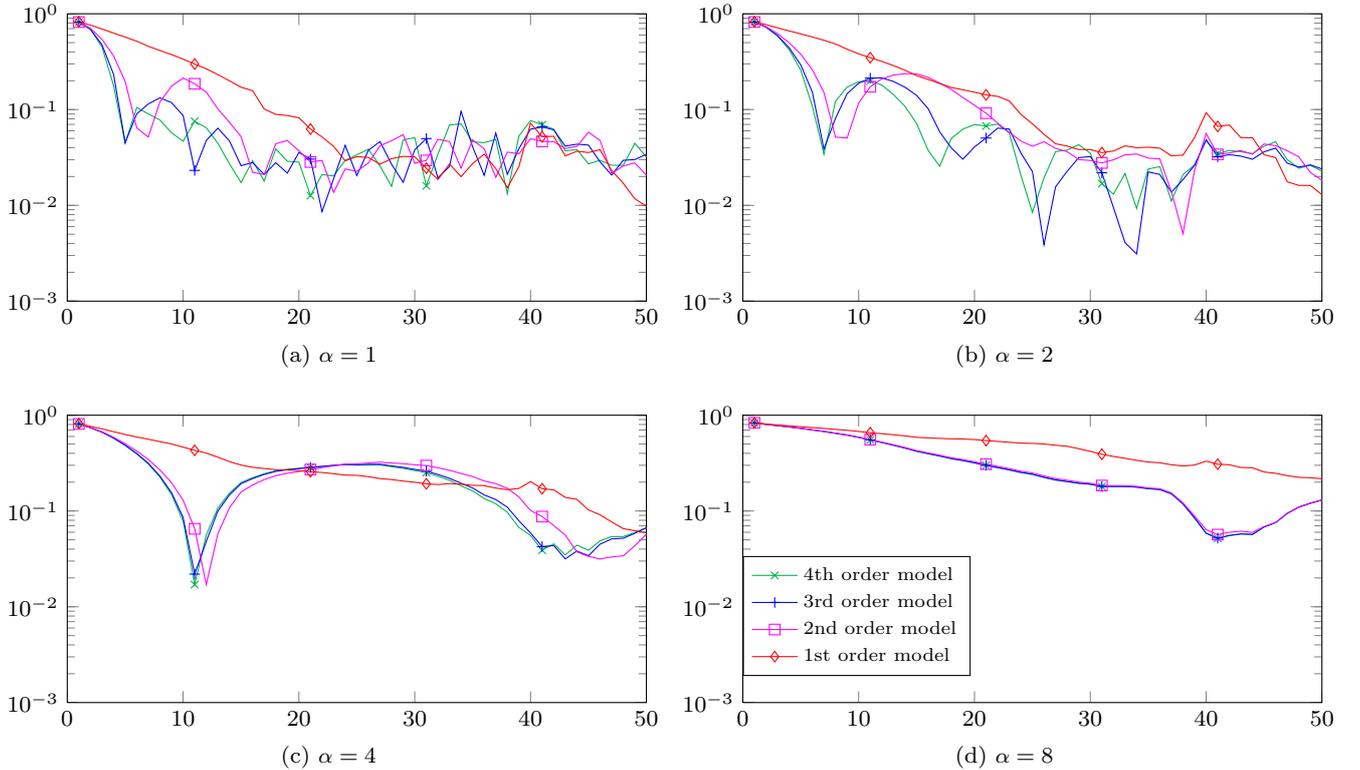

\subsection{Comparison with the Extended Kalman Filters}

\subsubsection{Experiments with Linear Observation Equation}

For the experiments in Fig.~\ref{fig:EKF-vs-MEF} we use four observation equations ($n=4$), and the vectors $a_{k}$ in \eqref{eq:linear-observer} are chosen as
\begin{equation}
a_{k} = e_{k}^{4}\,, \quad k \in [4]\,,
\end{equation}
to extract information from all directions. We generate the ground truth from an arbitrary initialization by integration of \eqref{eq:state-eq}  with multivariate Gaussian noise with mean $\mathbf{0}_{12}$ and diagonal covariance matrix $S = \eins_{12}$. As shown in \cite{bourmaud2015continuous}, we integrate the ground truth with ten times smaller step sizes than the filtering equations of extended Kalman and minimum energy filter. Afterwards we generate the observations with \eqref{eq:hk-linear-case} and Gaussian noise with covariance $Q = 10^{-8} \eins_{4}$ and set the covariance matrices $S$ and $Q$ in Algorithm~\ref{Alg:ExtendedKalman} to the same values. However, the matrix $Q$ for the minimum energy filter in Proposition~\ref{prop:MEF-linear} is set to $Q = 100 \,\eins_{4}$ to give more weight to the observations for faster convergence. Note that for the extended Kalman Filter the choice $Q = 100 \, \eins_{4}$  leads to a worse performance, which is why we use the true covariance instead.

As a reference, we apply our own implementation of the method by Bourmaud {\em et al.}~\cite{bourmaud2015continuous} adapted to our model. The results are demonstrated in Fig.~\ref{fig:EKF-vs-MEF}.
We suppose that the main reason for the different performances is that we compare a \emph{second-order} (minimum energy filter) with a \emph{first-order} (extended Kalman) filter.

\begin{figure*}[htbp]
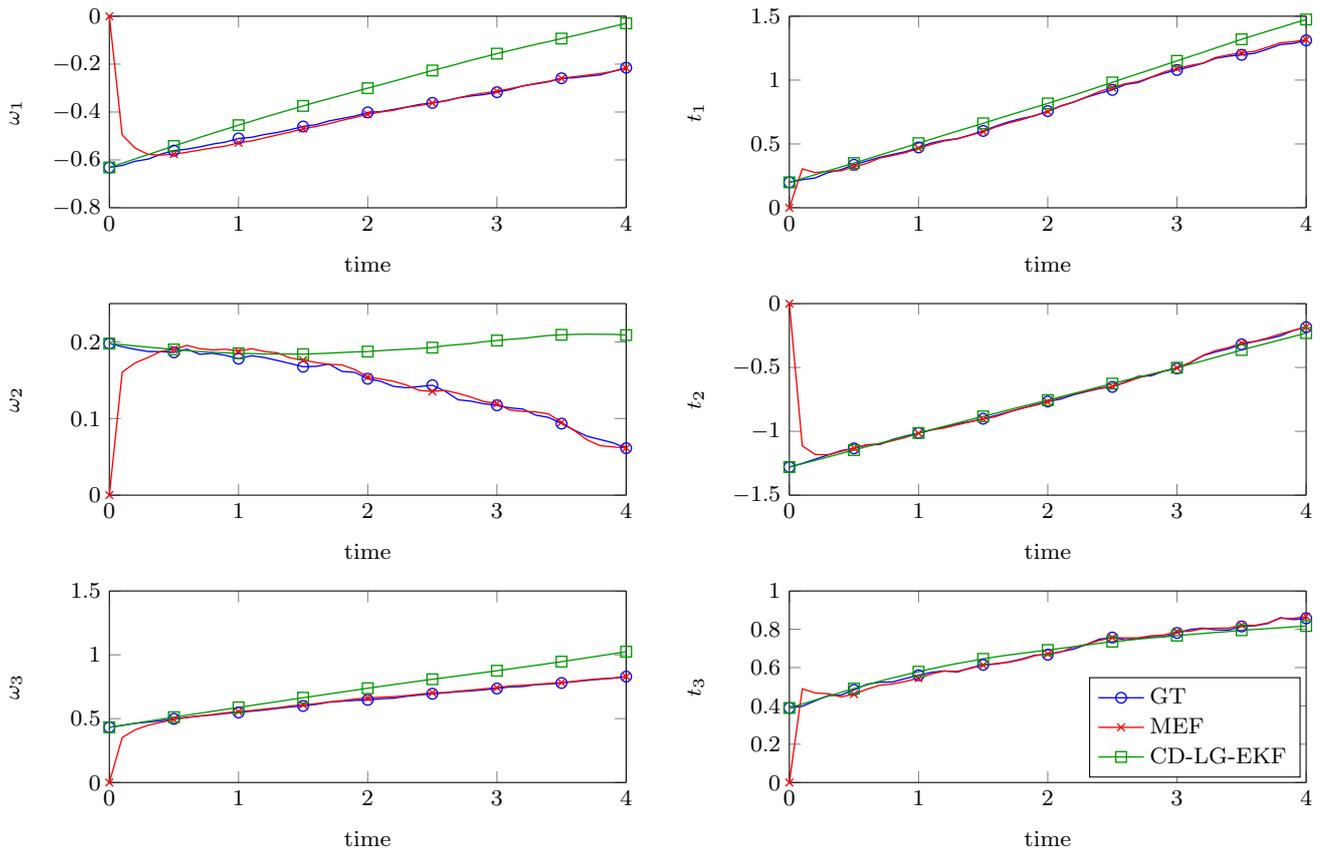

\begin{center}
%
%
\definecolor{mycolor1}{rgb}{0.00000, 0.60000,0.00000}%
\def\h{1}
\def\distI{0.1}
\def\distII{1.6}
\def\distIII{3.1}
\def\t{0.5}

%
\caption{Comparison between minimum energy filter with second-order kinematics (MEF) (red, cross) and extended Kalman filter (CD-LG-EKF) \cite{bourmaud2015continuous} (green, square) with state equation \eqref{eq:state-eq} and observation equation \eqref{eq:linear-observer} as derived in Properties~\ref{prop:MEF-linear} and \ref{prop:EKF-linear}, respectively. We plotted the six components of the rigid motion of the ground truth (GT) (blue, circle), the extended Kalman filter, and the minimum energy filter, i.e. $(\omega_{1}, \omega_{2}, \omega_{3}, t_{1},t_{2},t_{3})^{\T}:= (\vecg(\Log_{\G}(G))))_{1:6}$. Here, $G$ is the corresponding element of the Lie group $\G$. Further, we set the discretization step size to $\delta = 0.1$.
Although we initialized the extended Kalman filter with the ground truth solution and added only little observation noise, it diverges after a few steps whereas the minimum energy filter converges from a wrong initialization to the correct solution within a few steps. The reason for that is that the approach \cite{bourmaud2015continuous} only uses first-order approximation, whereas the minimum energy filter also includes second-order derivatives of the observation function. }
\label{fig:EKF-vs-MEF}
\end{center}
\end{figure*}

\subsubsection{Discussion on Extended Kalman Filter for Non-linear Observations}

We were not able to obtain convergence of this filter from a trivial (chosen as identity element of the Lie group) or ground truth  initialization.
Since the extended Kalman did not converge for linear observations~\eqref{sec:linear-observations} from wrong initializations, we presume that the non-linearities of our observation equations are intractable for the approach from~\cite{bourmaud2015continuous}.

\section{Limitations}
Our proposed method requires good measurements in terms of optical flow and depth maps in order to reconstruct the camera motion correctly. Although we showed on synthetic data that the proposed method is robust against different kinds of noise, it is not robust against outliers, caused by independently moving objects that violate the static scene assumption, or simply wrong computations of optical flow and depth maps. Making our approach robust as component of a superordinate processing stage, however, is beyond the scope of this paper and left for future work.

In addition to optical flow, the proposed method requires depth information which is expensive to obtain if not available anyway, e.g.\ in stereo camera setups.

\section{Conclusion \& Future Work}
We generalized the camera motion estimation approach~\cite{berger2015second} from a model with constant velocity assumption to a more realistic model with constant acceleration assumption as well as to a kinematic model which respects derivatives of any (fixed) order. To the authors' knowledge, this has not been done so far in the fields of image processing and computer vision.
For the resulting second-order minimum energy filter with higher-order kinematics, we provided all necessary derivations and demonstrated that our approach is superior to our previous method~\cite{berger2015second} for both synthetic and real-life data.
We also compared our approach to the state-of-the-art {\em continuous-discrete extended Kalman filter on connected unimodular matrix Lie groups}~\cite{bourmaud2015continuous} and showed that in both cases the minimum energy filters is superior since it converges from imperfect initializations to the correct solutions.

In the future, we want to investigate how to reconstruct the camera motion (with constant acceleration) {\em jointly with the camera's depth map} from {\em monocular} optical flow observations. 
\bigskip

\appendix

\section*{Appendix}

\section{Kronecker products on $\se$}\label{sec:kron}

The Kronecker products $\kronse, \kronse^{\T}: \R^{4\times 4} \times \R^{4\times 4} \rightarrow \R^{6\times 6}$ on $\se$  are defined for  matrices $A,B \in \R^{4\times 4}$ and $\eta \in \se$ through $\vecse(A\eta B)=: (A \kronse B)\vecse(\eta)$ and $\vecse(A \eta^{\T} B)=: (A \kronse^{\T} B)\vecse(\eta)$. Since the explicit formulas for $\kronse, \kronse^{\T}$ are quite uninformative, we do not provide them here.

\section{Properties of $\SE$ and $\mc{G}$}

\subsection{Projection onto $\se$}
The projection $\on{Pr}: \R^{4\times 4} \rightarrow \se$ is given by

\begin{align}
\begin{split}\label{eq:proj-se}
 \on{Pr}(A) := & \tfrac{1}{2}\diag((1,1,1,0)^{\T}) \bigl(A \diag((1,1,1,2)^{\T})  \\
  & \quad - A^{\T} \diag((1,1,1,0)^{\T}) \bigr)\,.
  \end{split}
\end{align}

\subsection{Adjoints, exponential and logarithmic map}

The adjoint operator $\ad_{\mf{se}} (\matse(v))$ can be computed for a vector $v\in \R^{6}$ as follows

\begin{align}
\vecse & \bigl(\ad_{\mf{se}} (\matse(v)) \eta) =  \ad_{\mf{se}}^{\on{vec}}(\matse(v)) \vecse(\eta)	\nonumber \\
  := &\bpm \matso(v_{1:3}) &  \mathbf{0}_{3\times 3} \\ \matso(v_{4:6}) & \matso(v_{1:3}) \epm \vecse(\eta) \,,\label{eq:adjoint_representaiton-se} 
\end{align}
where $\matso(v_{1:3}):=(\matse(v))_{1:3,1:3}$.
This directly follows from the definition of the adjoint as Lie bracket, i.e. $\ad_{\mf{se}}(\xi)\eta := [\xi, \eta]$ where the Lie bracket $[\cdot, \cdot]:\se \times \se \rightarrow \se$ is simply the matrix commutator on $\se.$ 
\begin{align}
\vecse & (\ad_{\se}(\matse(v)) \eta) = \vecse([ \matse(v), \eta])  \\
= & \vecse(\matse(v)\eta \eins_{4} -  \eins_{4} \eta \matse(v)) \\
= &  \bigl(\matse(v) \kronse \eins_{4} - \eins_{4} \kronse\matse(v)\bigr) \vecse(\eta) \,.\label{eq:adjoint_derivation}
\end{align}
A componentwise evaluation of \eqref{eq:adjoint_derivation} leads to \eqref{eq:adjoint_representaiton-se}. Since $\R^{6}$ is trivial, the adjoint representation on $\g$ parametrized by a vector $v \in \R^{12}$ is

\begin{equation}\label{eq:def-ad-g}
\ad_{\g}^{\on{vec}} (\matg(v)) = \bpm \ad_{\mf{se}}^{\on{vec}}(v_{1:6}) & \mathbf{0}_{6\times 6} \\ \mathbf{0}_{6\times 6} & \mathbf{0}_{6\times 6} \epm.
\end{equation}

The exponential map $\Exp_{\SE}: \se \rightarrow \SE$ and the logarithmic map on $\SE$ can be computed by the matrix exponential and matrix logarithm or more efficiently by the \emph{Rodrigues' formula} as in \cite[p.~413f ]{murray1994mathematical}.

Then the exponential map $\Exp_{\G}:\se \rightarrow \SE$ for a tangent vector $\eta = (\eta_{1}, \eta_{2})\in \g$ and the logarithmic map $\Log_{G}:\SE \rightarrow \se$  for $G =(E,v) \in \G$ are simply

\begin{align}
\Exp_{\G}(\eta) = & (\Exp_{\SE}(\eta_{1}), \eta_{2}) \in \G\,, \\
\Log_{\G}(G) = & (\Log_{\SE}(E), v) \in \g\,,
\end{align}
and similar for higher-order state spaces.

\subsection{Vectorization of connection function}
Following \cite[Section ~5.2]{Absil2008}, we can vectorize the connection function $\omega$ of the Levi-Civita connection $\nabla$ for {\em constant} $\eta,$ $\xi \in \g$ in the following way:
\begin{align}
\vecg(\omega_{\eta}\xi)= &\vecg( \omega(\eta,\xi)) = \vecg(\nabla_{\eta} \xi) =  \tilde{\Gamma}_{\vecg(\xi)}\vecg(\eta) \,, \label{eq:vec-connection-function}
\end{align}
where $\tilde{\Gamma}_{x}$ is the matrix whose $(i,j)$ element is the real-valued function
\begin{equation}
(\tilde{\Gamma}_{\gamma})_{i,j} := \sum_{k} (\gamma_{k} \Gamma_{jk}^{i}) \,,
\end{equation}
and $\Gamma_{jk}^{i}$ are the \emph{Christoffel symbols} of the connection function $\omega$ for a vector $\gamma \in \R^{12}.$ Similarly, permuting indices, we can define the adjoint matrix $\tilde{\Gamma}_{\gamma}^{\ast}$ whose $(i,j)$-th element is given by
\begin{equation}
(\tilde{\Gamma}_{\gamma}^{\ast})_{i,j} := \sum_{k} (\gamma_{k} \Gamma_{kj}^{i})\,.
\end{equation}
This leads to the following equality:
\begin{equation} \label{eq:vectorize-connection-function}
\vecg(\omega_{\eta}\xi) = \tilde{\Gamma}_{\vecg(\eta)}^{\ast}\vecg(\xi)\,.
\end{equation}
If the expression $\xi$ in \eqref{eq:vec-connection-function} is {\em non-constant}, we obtain the following vectorization from~\cite[Eq.~(5.7)]{Absil2008}, for the case of the Lie algebra $\se$, i.e.
\begin{align}
& \vecse(\nabla_{\eta_{x}} \xi(x) ) \nonumber \\
 = & \tilde{\Gamma}_{\vecg(\xi(x))}\vecse(\eta_{x}) + \D \vecse(\xi(x))[\vecse(\eta_{x})] \nonumber \\
= & \tilde{\Gamma}_{\vecg(\xi(x))}\vecse(\eta_{x}) + \sum_{i} (\eta_{x})_{i} \vecse(\D \xi(x))[E^{i}]) \nonumber \\
= & \tilde{\Gamma}_{\vecg(\xi(x))}\vecse(\eta_{x}) +    D  \vecse(\eta_{x})\,, \label{eq:vec-Levi-Civita-general}
\end{align}
where the entries of the matrix $D \in \R^{6\times 6}$ can be computed as
\begin{equation}
(D)_{i,j} = (\vecse(\D \xi(x)[E^{j}]))_{i}\,, \quad E^{j} = \matse(e_{j}^{6})\,, \label{eq:comp-D-levi-civita}
\end{equation}
where $e_{j}^{6}$ denotes the $j$-th unit vector in $\R^{6}$.

\section{Proofs}
\label{appendix:proofs}

\begin{proof}[of Proposition \ref{prop1}] The tangent map is simply the differential or directional derivative. 
For $G_{1}=(E_{1},v_{1}) ,G_{2}=(E_{2},v_{2})\in \G$ it holds $T_{G_{2}}L_{G_{1}}:T_{G_{2}}\G \rightarrow T_{L_{G_{1}}(G_{2})}\G.$ Thus, we can compute it for a $\eta = (E_{2}\eta_{1},\eta_{2})\in T_{G_{2}}\G = T_{E_{2}}\SE \times \R^{6}$ as follows
\begin{align*}
 T_{G_2}& L_{G_{1}} \circ \eta =  \D L_{G_{1}}(G_{2}) [\eta] \\
= & \lim_{\tau \rightarrow 0^{+}} \tau^{-1} \bigl( L_{G_{1}}(G_{2}+\tau \eta) - L_{G_{1}}(G_{2})  \bigr)   \\
= & \lim_{\tau \rightarrow 0^{+}} \tau^{-1} \bigl( L_{(E_{1},v_{1})}((E_{2}+\tau E_{2}\eta_{1}, v_{2}+\tau \eta_{2}))  \\ & \quad- (E_{1}E_{2},v_{1}+v_{2})\bigr) \\
= & \lim_{\tau \rightarrow 0^{+}} \tau^{-1} \bigl( (E_{1}E_{2} + \tau E_{1}E_{2}\eta_{1}, v_{1}+ v_{2} + \tau \eta_{2} )  \\ & \quad- (E_{1}E_{2},v_{1}+v_{2})\bigr) \\
= & (G_{1}G_{2}\eta_{1}, \eta_{2}) \in T_{G_{1}G_2}\G = T_{L_{G_{1}}(G_2)}\G \,.
\end{align*}
For $G_{2}=\Id = (\eins_{4}, \mathbf{0}_{6})$ and $\eta = (\eta_{1},
\eta_{2}) \in \g$, it follows 
\begin{equation*}
T_{\Id}L_{G_{1}} \circ \eta = (E_{1}\eta_{1}, \eta_{2}) = L_{(E_{1}, \mathbf{0}_{6})}(\eta_{1},\eta_{2}) =: G_{1} \eta  \in T_{G_{1}}\G \, .
\end{equation*}
Note that the adjoint of the tangent map of $L_{G}$ at identity can be expressed as inverse of $G=(E,v)$, i.e. for $\eta=(\eta_{1},\eta_{2}) \in T_{G}\G$ and $\xi=(\xi_{1}, \xi_{2}) \in \g$
\begin{align*}
\la T_{\Id}L_{G}^{\ast} \eta ,\xi \ra_{\Id} = & \la \eta, T_{\Id}L_{G} \xi \ra_{G} \\
= & \la \eta_{1}, E\xi_{1} \ra_{E} + \la \eta_{2} , \xi_{2} \ra  \\
= & \la E^{-1} \eta_{1}, \xi_{1} \ra_{\Id_{\se}} + \la \eta_{2}, \xi_{2} \ra \\
= & \la L_{(E^{-1},\mathbf{0}_{6})}\eta, \xi \ra_{\Id} \,.
\end{align*}
Thus, $T_{\Id}L_{G}^{\ast} \eta = L_{(E^{-1}, \mathbf{0}_{6})}\eta$. We will use the shorthand $G^{-1}\eta:= T_{\Id}L_{G}^{\ast}$ for the dual of the tangent map of $L_{G}$ at identity. \qed
\end{proof}

\begin{proof}[of Lemma \ref{lemma1}] Since $\mu = (\mu_{1}, \mu_{2}),v$ are independent of $E$ the gradient $\D_{1} \mc H^{-}(G=(E,v), \mu, t)$ can be computed separately in terms of $E$, i.e.  for $\eta = (E \eta_{1}, \eta_{2}) \in T_{G}\G$
\begin{align*}
\D_{1} \mc H^{-}(G,\mu,t)[\eta] = & \Bigl( \D_{E} \tfrac{1}{2}e^{-\alpha(t-t_{0})} \Bigl( \sum_{k=1}^{n} \lVert y_{k}-h_{k}(E)\rVert_{Q}^{2} \Bigr)[\eta_{1}], \\ & \quad  - \D_{v}\la \mu_{1}, \matse(v)\ra[\eta_{2}] \Bigr) \,.
\end{align*}
The directional derivative regarding $v$ can be computed by the usual gradient on $\R^{6}$ which is given by
\begin{align}\label{eq:directional-der-hamilton-v}
\begin{split}
- \D_{v} \la \mu_{1}, \matse(v) \ra [\eta_{2}] = & - \la \vecse(\mu_{1}), \D_{v} v[\eta_{2}] \ra \\
= &  \la -\vecse(\mu_{1}), \eta_{2} \ra\,,
\end{split}
\end{align}
such that $\D_{v} \la \mu_{1}, \matse(v) \ra = -\vecse(\mu_{1})$. For the directional derivative of $\mc H^{-}$ we first consider the directional derivative of $h_{k}(E).$ Since $h_{k}(E)$ can also be written as
\begin{equation}\label{def:hk}
h_{k}(E):= ((e_{3}^{4})^{\T} E^{-1}g_{k}(t))^{-1} \hat{I}E^{-1} g_{k}(t)\,, \quad \hat{I}:=\bsm 1 & 0 & 0 & 0 \\ 0 & 1 & 0 & 0 \esm\,,
\end{equation}
the directional derivative (into direction $\xi$) can be derived by the following matrix calculus.

\begin{align}
\D & h_{k}(E) [\xi] \\
 = & \D \bigl(((e_{3}^{4})^{\T} E^{-1}g_{k})^{-1}\bigr)[\xi]\hat{I}E^{-1} g_{k} \nonumber \\
  & \quad +  ((e_{3}^{4})^{\T} E^{-1}g_{k})^{-1} \D \bigl(\hat{I}E^{-1} g_{k}\bigr)[\xi] \nonumber \\
= & - \kappa_{k}^{-1} \D ((e_{3}^{4} )^{\T}E^{-1}g_{k})[\xi] \kappa_{k}^{-1} \hat{I} E^{-1} g_{k}  \nonumber \nonumber \\
& \quad + \kappa_{k}^{-1} \hat{I} \D (E^{-1})[\xi] g_{k} \nonumber  \\
= & - \kappa_{k}^{-1}(e_{3}^{4})^{\T}  \D ( E^{-1})[\xi]g_{k} \kappa_{k}^{-1} \hat{I} E^{-1} g_{k}  \nonumber  \\
&  + \kappa_{k}^{-1} \hat{I} \D (E^{-1})[\xi] g_{k} \nonumber  \\
= &  - \kappa_{k}^{-1}(e_{3}^{4})^{\T} (-1)E^{-1} \D(E)[\xi]E^{-1}g_{k} \kappa_{k}^{-1} \hat{I} E^{-1} g_{k} \nonumber  \\
& \quad   + \kappa_{k}^{-1} \hat{I}(-1)E^{-1} \D(E)[\xi]E^{-1} g_{k}  \nonumber  \\
\begin{split}
= &   \kappa_{k}^{-2}(e_{3}^{4})^{\T}E^{-1}\xi E^{-1}g_{k}\hat{I} E^{-1} g_{k} - \kappa_{k}^{-1} \hat{I}E^{-1} \xi E^{-1} g_{k}\,, \label{eq:Dhk} 
\end{split}
\end{align}

where $\kappa_{k} = \kappa_{k}(E):= (e_{3}^{4})^{\T}E^{-1} g_{k}.$  Then  for the choice $\xi = E\eta_{1}$ we find that
\begin{align}
& e^{\alpha(t-t_{0})} \D_{1} \mathcal{H}^{-}(G, \mu, t)[E\eta_{1}] \\
=&  - \sum_{k=1}^{n} \tr \bigl( \D h_{k}(E)[E\eta_{1}] (y_{k} - h_{k}(E))^{\T} Q \bigr) \nonumber \\
 =&  - \sum_{k=1}^{n} \tr \bigl( \bigl( \kappa_{k}^{-2} ((e_{3}^{4})^{\T} \eta_{1} E^{-1} g_{k}) \hat{I}E^{-1} g_{k}-\kappa_{k}^{-1} \hat{I} \eta_{1} E^{-1} g_{k}\bigr) \nonumber\\
 & \quad \cdot (y_{k} - h_{k}(E))^{\T} Q \bigr)\nonumber \\
 = &\sum_{k=1}^{n} \tr \bigl( \bigl(\kappa_{k}^{-1} \hat{I} \eta_{1} E^{-1} g_{k}-  \kappa_{k}^{-2} ((e_{3}^{4})^{\T} \eta_{1} E^{-1} g_{k}) \hat{I}E^{-1} g_{k}\bigr)\nonumber \\
 & \quad \cdot  (y_{k} - h_{k}(E))^{\T}Q  \bigr)\nonumber  \\
 = &  \sum_{k=1}^{n} \tr \bigl( \bigl(\kappa_{k}^{-1} \hat{I} \eta_{1} E^{-1} g_{k}-  \kappa_{k}^{-2} \hat{I}E^{-1} g_{k} (e_{3}^{4})^{\T} \eta_{1} E^{-1} g_{k} \bigr) \nonumber \\
 & \quad \cdot  (y_{k} - h_{k}(E))^{\T} Q \bigr)\nonumber  \\
 = & \sum_{k=1}^{n} \tr \bigl( \bigl(\kappa_{k}^{-1}\hat{I} -  \kappa_{k}^{-2} \hat{I}E^{-1} g_{k} (e_{3}^{4})^{\T}\bigr) \eta_{1} E^{-1} g_{k}  (y_{k} - h_{k}(E))^{\T} Q \bigr) \nonumber \\
 = & \sum_{k=1}^{n} \tr \bigl(E^{-1} g_{k}  (y_{k} - h_{k}(E))^{\T} Q \bigl(\kappa_{k}^{-1}\hat{I} -  \kappa_{k}^{-2} \hat{I}E^{-1} g_{k} (e_{3}^{4})^{\T}\bigr)\eta_{1}  \bigr)\nonumber  \\
 = & \sum_{k=1}^{n}  \big \la \underbrace{\bigl(\kappa_{k}^{-1}\hat{I}  -  \kappa_{k}^{-2}  \hat{I}   E^{-1}   g_{k}(e_{3}^{4})^{\T} \bigr)^{\T}Q(y_{k} - h_{k}(E))   g_{k}^{\T}  E^{-\T}}_{=:A_k(E)},\eta_{1}  \big \ra \,_{\Id} .\label{eq:proof_lemma1:DH}
\end{align}
Here we used that the trace is cyclic. We obtain the Riemannian gradient on~$\SE$ by projecting (cf.~\cite[Section ~3.6.1]{Absil2008}) the left hand side of the Riemannian metric in \eqref{eq:proof_lemma1:DH} onto $T_{E}\SE$, which is for $G=(E,v)$
\begin{align}\label{eq:gradE_Hamiltonian1}
\begin{split}
\D_{E} \mathcal H^{-}(G,\mu,t) = & e^{-\alpha(t-t_{0})} \on{Pr}_{E}\bigl(E A_{k}(E)  \bigr)\\
= &  e^{-\alpha(t-t_{0})} \sum_{k} E\on{Pr}\bigl( A_{k}(E)  \bigr) \,,
\end{split}
\end{align}
with $A_{k}(E):=  
\bigl(\kappa_{k}^{-1}\hat{I}  -  \kappa_{k}^{-2}  \hat{I}   E^{-1}   g_{k}(e_{3}^{4})^{\T} \bigr)^{\T}
Q(y_{k} - h_{k}(E))   g_{k}^{\T}  E^{-\T}$,  
and $\on{Pr}_{E}:\on{GL}_{4} \rightarrow T_{E}\SE$ denotes the projection onto the tangential space $T_{E}\SE$ that can be expressed in terms of $\on{Pr}_{E}(E \cdot) = E \on{Pr}(\cdot)$. Besides, $\on{Pr}:\on{GL}_{4} \rightarrow \se$ denotes the projection onto the Lie algebra $\se$ as given in \eqref{eq:proj-se}.

Putting together \eqref{eq:directional-der-hamilton-v} and \eqref{eq:gradE_Hamiltonian1} results in
\begin{align}
\begin{split}
& \D_{1}\mc H^{-}(G,\mu,t)  = \\
&   \Bigl(  e^{-\alpha(t-t_{0})} \sum_{k=1}^n E\on{Pr}\bigl( A_{k}(E)  \bigr), -\vecse(\mu_{1})\Bigr)  \in T_{G}\G \,.
\end{split}
\end{align}\qed
\end{proof}

\begin{proof}[of Lemma \ref{lemma2}] 
Eq.~\eqref{eq:vectorization_Z} can be easily found by considering a basis of $\se$ and the fact that $Z$ is a linear operator on the Lie algebra. Since the resulting matrix $K(t) \vecg(\eta) := Z(G^{\ast},t)\circ \eta$ depends only on $t$, the equation \eqref{eq:vec-temp-der-Z}. Eq.\eqref{eq:vec-inv-Z} is trivial since $Z$ is linear.
\begin{enumerate}
	\item[1.] With the symmetry of the Levi-Civita connection, i.e. 
	\begin{equation}\label{eq:symmetry-lc}
	[\eta,\xi] = \nabla_{\eta}\xi - \nabla_{\xi}\eta \,,
	\end{equation}
	 we gain the following equalities
	\begin{align}
	\vecg(& Z(G^{\ast},t) \circ \omega_{(G^{\ast})^{-1} \dot{G^{\ast}}}\eta + Z(G^{\ast},t) \circ  \omega_{\D_{2}\mc H^{-}(G^{\ast},0,t)}^{\leftrightharpoons}\eta) \nonumber  \\
	\stackrel{\eqref{eq:vectorization_Z}}{=} & K(t) \vecg(\omega_{(G^{\ast})^{-1} \dot{G^{\ast}}}\eta + \omega_{\D_{2}\mc H^{-}(G^{\ast},0,t)}^{\leftrightharpoons}\eta) \nonumber \\
\stackrel{\eqref{eq:ode-optimal-state1}}{=} & K(t)   \vecg( \nabla_{ -\D_{2} \mc H^{-}(G^{\ast},0,t)}\eta \nonumber \\
&  - \nabla_{e^{-\alpha(t-t_{0})} Z(G^{\ast},t)^{-1} \circ r_{t}(G^{\ast})}\eta + \nabla_{\eta}\D_{2}\mc H^{-}(G^{\ast},0,t)) \nonumber\\
\stackrel{\eqref{eq:symmetry-lc}}{=} & K(t) \vecg(-[\D_{2}\mc H^{-}(G^{\ast},0,t),\eta]) \nonumber   \\
& - \nabla_{e^{-\alpha(t-t_{0})} Z(G^{\ast},t)^{-1} \circ r_{t}(G^{\ast})}\eta \nonumber \\
\stackrel{\eqref{eq:vectorize-connection-function}}{=} & K(t) \bigl( \vecg([f(G^{\ast}),\eta]) \nonumber \\
&  - \tilde{\Gamma}_{\vecg(e^{-\alpha(t-t_{0})} Z(G^{\ast},t)^{-1} \circ r_{t}(G^{\ast}))}^{\ast}\vecg(\eta) \bigr) \nonumber \\
\stackrel{\eqref{eq:vectorization_Z}}{=} & K(t) \bigl( \vecg([f(G^{\ast}),\eta]) \nonumber \\
&  + \tilde{\Gamma}_{-e^{-\alpha(t-t_{0})} K(t)^{-1}\vecg(r_{t}(G^{\ast}))}^{\ast}\bigr)\vecg(\eta\bigr)  \nonumber \\
\stackrel{ \eqref{eq:def-ad-g}}{=} & K(t) \bigl( \ad_{\g}^{\on{vec}} (f(G^{\ast})) \nonumber \\
& \qquad + \tilde{\Gamma}_{-e^{-\alpha(t-t_{0})} K(t)^{-1}\vecg(r_{t}(G^{\ast}))}^{\ast}  \bigr)\vecg(\eta) \nonumber \\
\begin{split}
=: & K(t) B \vecg(\eta)\,.
\end{split}
\label{eq:vectorization-connection-function}
	\end{align}
        The claim follows from the fact that the adjoints and the Christoffel symbols on $\R^{6}$ are zero.
	\item[2.] Since this expression is dual to the expression in 1. the claim follows by using its transpose.

	\item[3.]  Recall that the Hamiltonian in \eqref{eq:left-trivial-hamilton} is given by
	\begin{align*}
\mathcal{H}^{-}& ((E,V), \mu, t) =  \tfrac{1}{2}e^{-\alpha(t-t_{0})} \Bigl( \sum_{k=1}^{n} \lVert y_{k}-h_{k}(E)\rVert_{Q}^{2} \Bigr) \nonumber \\
& - \tfrac{1}{2}e^{\alpha(t-t_{0})}\Bigl( \la  \mu_{1}, \matse(S_{1}^{-1}\vecse(\mu_{1})) \ra_{\Id}  \\
& + \la \mu_{2}, S_{2}^{-1} \mu_{2} \ra\Bigr) - \la \mu_{1}, \matse(V)\ra_{\Id} \,.\nonumber
\end{align*}
The Riemannian Hessian w.r.t.\ the first component can be computed for $G=(E,v) \in \G$, $\eta=(\eta_{1},\eta_{2}) \in\g$ and the choice $\mu = (\mu_{1},\mu_{2})=(\mathbf{0}_{4\times 4}, \mathbf{0}_{6})$ as 
\begin{align}
& e^{\alpha(t-t_{0})}\vecg  ( G^{-1} \Hess_{1} \mc H^{-}(G,\mu,t)[G\eta])  \nonumber \\
& =e^{\alpha(t-t_{0})} \vecg\Bigl(G^{-1} \nabla_{G\eta} \D_{1} \mc{H}^{-}(G,\mathbf{0},t) \Bigr)   \label{eq:Hess1Hamilton-1} \\
& = e^{\alpha(t-t_{0})} \vecg\Bigl( \nabla_{\eta} G^{-1} \D_{1} \mc{H}^{-}(G,\mathbf{0},t) \Bigr)  \label{eq:Hess1Hamilton-2}  \\
& = \vecg\Bigl(\nabla_{\eta} \bigl(\sum_{k=1}^{n}\on{Pr}(A_{k}(E)), -e^{\alpha(t-t_{0})} \vecse(\mathbf{0}_{4\times 4}) \bigr)\Bigr) \label{eq:Hess1Hamilton-3}  \\
& =  \Bigl(\sum_{k=1}^{n}  \vecse\bigl(\nabla_{\eta_{1}} \on{Pr}\bigl(A_{k}(E)\bigr)\bigr), \mathbf{0}_{6}\Bigr)\nonumber \\ 
\begin{split}  \label{eq:Hess1Hamilton-5} 
& = \sum_{k=1}^{n}\Bigl( \tilde{\Gamma}_{\vecg\bigl(\on{Pr}(A_{k}(E))\bigr)}\vecse(\eta_{1}) \\
& \quad \quad \quad + \sum_{i} (\eta_{1})_{i} \vecse(\D \on{Pr}\bigl(A_{k}(E)\bigr))[E^{i}]) \Bigr)\,.
\end{split}
\end{align}

Here, line \eqref{eq:Hess1Hamilton-1} follows from the general definition of the Hessian (cf.~\cite[Def.~5.5.1]{Absil2008}). Line \eqref{eq:Hess1Hamilton-2} holds because of the linearity of the affine connection, the equation \eqref{eq:Hess1Hamilton-3} results from insertion of the expression in Lemma~\ref{lemma1} and 
\eqref{eq:Hess1Hamilton-5} can be achieved with \eqref{eq:vec-Levi-Civita-general}. 

As next we calculate the differential  $ \D \on{Pr}(A_{k}(E)) [\eta_{1}]$ in \eqref{eq:Hess1Hamilton-5} for an arbitrary direction $\eta_{1}$. Since the projection is a linear operation (cf.~\eqref{eq:proj-se}), i.e.~$\D \on{Pr}(A_{k}(E))[\eta_{1}] = \on{Pr}(\D A_{k}(E)[\eta_{1}])$, we require to calculate $\D A_{k}(E)[\eta_{1}]$. By using the product rule and the definition of $A_{k}$ from \eqref{eq:def-Ak} we obtain

\begin{align}
& \D A_{k}(E) [\eta_{1}] \nonumber \\
 = & \D \bigl(\bigl(\kappa_{k}^{-1}\hat{I}  -  \kappa_{k}^{-2}  \hat{I}   E^{-1}   g_{k}(e_{3}^{4})^{\T} \bigr)^{\T}Q(y_{k} - h_{k}(E))   g_{k}^{\T}  E^{-\T} \bigr)[\eta_{1}]  \nonumber \\
 \begin{split}\label{eq:DAk}
= &  \bigl(\D \bigl(\kappa_{k}^{-1}\hat{I}  -  \kappa_{k}^{-2}  \hat{I}   E^{-1}   g_{k}(e_{3}^{4})^{\T} \bigr)^{\T}[\eta_{1}]Q(y_{k} - h_{k}(E))   g_{k}^{\T}  E^{-\T} \bigr) \\
+  & \bigl(\kappa_{k}^{-1}\hat{I}  -  \kappa_{k}^{-2}  \hat{I}   E^{-1}   g_{k}(e_{3}^{4})^{\T} \bigr)^{\T}Q \Bigl(( - \D h_{k}(E)[\eta_{1}])   g_{k}^{\T}  E^{-\T} \bigr) \\
 + &  \bigl((y_{k} - h_{k}(E))   g_{k}^{\T}  \D E^{-\T}[\eta_{1}] \bigr) \Bigr) \,.
 \end{split}
\end{align}	

The directional derivative of $\bigl(\kappa_{k}^{-1}\hat{I}  -  \kappa_{k}^{-2}  \hat{I}   E^{-1}   g_{k}(e_{3}^{4})^{\T} \bigr)$ is
\begin{align}
& \D \bigl(\kappa_{k}^{-1}\hat{I}  -  \kappa_{k}^{-2}  \hat{I}   E^{-1}   g_{k}(e_{3}^{4})^{\T} \bigr)[\eta_{1}] \nonumber \\
= & -\kappa_{k}^{-2} (e_{3}^{4})^{\T} \D E^{-1}[\eta_{1}] g_{k}\hat{I} \nonumber  \\
& \quad  + 2 \kappa_{k}^{-3}(e_{3}^{4})^{\T} \D E^{-1}[\eta_{1}]g_{k} \hat{I}E^{-1}g_{k}(e_{3}^{4})^{\T} \nonumber  \\
& \quad  \quad - \kappa_{k}^{-2} \hat{I} \D E^{-1} [\eta_{1}] g_{k}(e_{3}^{4})^{\T} \nonumber \\
\begin{split}\label{eq:DAk-first}
= & \kappa_{k}^{-2} (e_{3}^{4})^{\T} E^{-1}\eta E^{-1} g_{k} \hat{I} \\
& \quad - 2\kappa_{k}^{-3}(e_{3}^{4})^{\T} E^{-1}\eta_{1} E^{-1} g_{k} \hat{I}E^{-1}g_{k}(e_{3}^{4})^{\T} \\
& \quad \quad + \kappa_{k}^{-2} \hat{I}E^{-1} \eta_{1} E^{-1} g_{k}(e_{3}^{4})^{\T} \,.
\end{split}
\end{align}

By inserting the directional derivatives \eqref{eq:DAk-first}, \eqref{eq:Dhk} and $\D E^{-\T}[\eta_{1}] = -(E^{-1}\eta_{1}E^{-1})^{\T}$ into \eqref{eq:DAk}, we obtain the vector-valued function $\zeta^{k}(E)(\cdot): \se \rightarrow \R^{6}$ defined as
		
\begin{align}		
& \matse( \zeta^{k}(E)(\eta_{1})) := \on{Pr}\bigl(\D A_{k}(E)[\eta_{1}] \bigr) \label{eq:def-zeta}\\ 
= & \on{Pr} \biggl(\Bigl(\kappa_{k}^{-2} (e_{3}^{4})^{\T} E^{-1}\eta_{1} E^{-1} g_{k} \hat{I} \nonumber \\
& \quad - 2\kappa_{k}^{-3}(e_{3}^{4})^{\T} E^{-1}\eta_{1} E^{-1} g_{k} \hat{I}E^{-1}g_{k}(e_{3}^{4})^{\T} \nonumber \\
& \quad + \kappa_{k}^{-2} \hat{I}E^{-1} \eta_{1} E^{-1} g_{k}(e_{3}^{4})^{\T}\Bigr)^{\T}Q\bigl(y_{k} - h_{k}(E)\bigr)   g_{k}^{\T}  E^{-\T}  \nonumber  \\
& + \bigl(\kappa_{k}^{-1}\hat{I}  -  \kappa_{k}^{-2}  \hat{I}   E^{-1}   g_{k}(e_{3}^{4})^{\T} \bigr)^{\T}Q \Bigl(\bigl(\kappa_{k}^{-1} \hat{I}E^{-1} \eta_{1} E^{-1} g_{k} \nonumber  \\
& - \kappa_{k}^{-2}(e_{3}^{4})^{\T}E^{-1}\eta_{1} E^{-1}g_{k}\hat{I} E^{-1} g_{k} \bigr) g_{k}^{\T}E^{-\T} \nonumber \\
&  - \bigl(y_{k}-h_{k}(E)\bigr)g_{k}^{\T} E^{-\T} \eta_{1}^{\T} E^{-\T} \Bigr)\biggr) \,. \nonumber
\end{align}	

Using the basis $\{E^{j}\}_{j=1}^{6}$ of $\se$, with $E^{j}:= \matse(e_{j}^{6})$ we define, as in \eqref{eq:comp-D-levi-civita}, the following matrix $D_{k}(E) \in \R^{6\times 6}$ with components
\begin{equation}\label{eq:def_Dk}
(D_{k}(E))_{i,j}:= \zeta_{i}^{k}(E^{j})\,.
\end{equation}
By using the equation \eqref{eq:vec-Levi-Civita-general} we find that
\begin{equation*}
\vecse \bigl(\nabla_{\eta_{1}} \on{Pr}(A_{k}(E))\bigr) = \bigl(\tilde{\Gamma}_{\on{Pr}(A_{k}(E))} + D_{k}(E)\bigr) \vecse(\eta_{1})\,.
\end{equation*}
Insertion of this expression into \eqref{eq:Hess1Hamilton-5} leads finally to the desired result, i.e.
\begin{align}
& e^{\alpha(t-t_{0})}\vecg  ( G^{-1} \Hess_{1} \mc H^{-}(G,\mu,t)[G\eta])  \nonumber \\
& =  \begin{pmatrix}
\sum_{k=1}^{n} ( \tilde{\Gamma}_{\vecse(\on{Pr}(A_{k}(E)))} + D_{k}(E)) &  \mathbf{0}_{6\times 6} \nonumber  \\
\mathbf{0}_{6\times 6} & \mathbf{0}_{6\times 6}
\end{pmatrix} \vecg(\eta) \,. \nonumber 
\end{align}

\item[4.] The Riemannian gradient of the Hamiltonian regarding the second component is at zero, thus we obtain 
\begin{equation} \label{eq:grad2-zero}
\D_{2} \mc H^{-} (G,\mathbf{0}, t) = \bigl(  - \matse(v) , \mathbf{0} \bigr) = - f(G)\,.
\end{equation}
Computation of differential regarding the first component at $\eta = (E\eta_{1},\eta_{2}) \in T_{G}\G$ results in
\begin{align*}
\D_{1}&(\D_{2} \mc H^{-}(G,\mathbf{0},t))[\eta] = -\D f(G)[\eta] \\
= & -\D_{(E,v)} (\matse(v), \mathbf{0})[\eta] \\
= & -(\matse(\eta_{2}) , \mathbf{0})\,.
\end{align*}
Finally, we compute the complete expression which is for $\eta=(\eta_{1},\eta_{2}) \in \g$ and $G^{\ast}=(E,v) \in \G$
\begin{align}
\vecg(Z(G^{\ast},t)& \circ \D_{1} (\D_{2} \mc H^{-})(G^{\ast},0,t) \circ T_{\on{Id}}L_{G^{\ast}}\eta) \nonumber  \\
\begin{split}\label{eq:vecD1D2H}
= & K(t) \vecg( \D_{1} (\D_{2} \mc H^{-})(G^{\ast},0,t) [E\eta_{1},\eta_{2}]) \\
= & -K(t) \vecg((\matse(\eta_{2}) , \mathbf{0})) \\
= & -K(t) \begin{pmatrix}
\mathbf{0}_{6\times 6} & \eins_{6} \\
\mathbf{0}_{6\times 6} & \mathbf{0}_{6\times 6}
\end{pmatrix}
\vecg(\eta)\,.
\end{split}
\end{align}
	\item[5.] The following duality holds
	 \begin{align}
	 \begin{split} \label{eq:D2D1H}
	 \D_{2}(\D_{1} \mc H^{-}(G^{\ast},0,t)) = & (\D_{1}(\D_{2} \mc H^{-}(G^{\ast},0,t)))^{\ast} \\ =&  - (\D_{G^{\ast}}f(G^{\ast}))^{\ast}\,,
	 \end{split}
	 \end{align}
as well as the following duality rule for linear operators $f,g: \g \rightarrow \g^{\ast}$ (i.e. $f^{\ast},g^{\ast}: \g \rightarrow \g^{\ast}$ by the identification $\g^{\ast \ast} = \g$) and $\eta, \xi \in \g$,
\begin{align}
\begin{split}
& \la (g^{\ast} \circ f^{\ast})(\eta) ,\xi \ra_{\Id} = \la f^{\ast}(\eta), g(\xi) \ra_{\Id} \\
= &  \la \eta, (f\circ g)(\xi) \ra_{ \Id} = \la (f\circ g)^{\ast}(\eta), \xi \ra_{\Id}\,,
\end{split}
\end{align}
from which follows 
\begin{equation}\label{eq:duality-permutation}
(g^{\ast} \circ f^{\ast}) = (f \circ g)^{\ast}\,.
\end{equation}	 
Note that for $\g = \se$ we replace the Riemannian metric $\la \cdot, \cdot \ra$ by the trace, and that the dual notation can be replaced by the transpose.
	 
Applying the $\vecg-$ operation for $\eta \in \g$ gives
\begin{align*}
\vecg(T_{\on{Id}} & L_{G^{\ast}}^{\ast} \circ \D_{2}(\D_{1} \mc H^{-})(G^{\ast},0,t) \circ Z(G^{\ast},t) \circ \eta) \\
\stackrel{\eqref{eq:D2D1H}}{=} & 	-\vecg(T_{\on{Id}} L_{G^{\ast}}^{\ast} \circ  (\D f(G^{\ast}))^{\ast}  \circ Z(G^{\ast},t)\circ \eta) \\
\stackrel{\eqref{eq:duality-permutation}}{=} &  	-\vecg((\D f(G^{\ast}) \circ T_{\on{Id}} L_{G^{\ast}})^{\ast}  \circ Z(G^{\ast},t)\circ \eta) \\
\stackrel{\eqref{eq:vecD1D2H}}{=} & - \begin{pmatrix}
\mathbf{0}_{6\times 6} & \mathbf{0}_{6\times 6} \\
\eins_{6} & \mathbf{0}_{6\times 6}
\end{pmatrix} \vecg(  Z(G^{\ast},t)\circ \eta)  \\
= &  - \begin{pmatrix}
\mathbf{0}_{6\times 6} & \mathbf{0}_{6\times 6} \\
\eins_{6} & \mathbf{0}_{6\times 6}
\end{pmatrix}K(t)\vecg(\eta)  \,.
\end{align*}
\item[6.] It holds for $\eta = (\eta_{1}, \eta_{2}) \in \mf g$ and the definition of the Riemannian Hessian that
\begin{equation}\label{eq:Hess2Hamilton}
\Hess_{2} \mc H^{-} (G,\mu,t) [ \eta] = \nabla_{(\eta_{1},\eta_{2})} \D_{2} \mc H^{-}(G,\mu,t) \,.
\end{equation}

The Riemannian gradient of the Hamiltonian regarding the second component can be computed for $G=(E,v)\in \G$ as

\begin{align}
& \D_{2}  \mc H^{-} (G,\mu, t) \nonumber  \\
=&  \bigl( - e^{\alpha(t-t_{0})} \matse(S_{1}^{-1}\vecse(\mu_{1})) - \matse(v), \label{eq:grad2-hamiltonian} \\
& \quad  \quad \quad \quad - e^{\alpha(t-t_{0})} S_{2}^{-1} \mu_{2} \bigr)\,. \nonumber
\end{align}

Inserting \eqref{eq:grad2-hamiltonian} into \eqref{eq:Hess2Hamilton} results in
\begin{align*}
& e^{-\alpha(t-t_{0})}\Hess_{2}  \mc H^{-} (G,\mu,t)  [\eta]  \\
= & - \nabla_{(\eta_{1}, \eta_{2})} \bigl( \matse(S_{1}^{-1}\vecse(\mu_{1})) + \matse(v) , S_{2}^{-1} \mu_{2} \bigr) \\
= & - \on{Pr}_{\mf g}\Bigl( \D_{\mu}(\matse(S_{1}^{-1}\vecse(\mu_{1})) + \matse(v))[\eta], \\
& \hspace{5.5cm}  \D_{\mu}(S_{2}^{-1}\mu_{2})[\eta]   \Bigr) \\
= &- \Bigl( \on{Pr} \bigl( \matse(S_{1}^{-1}\vecse(\eta_{1})) \bigr)  ,  S_{2}^{-1}\eta_{2} \Bigr) \\
= & -  \Bigl( \matse(S_{1}^{-1} \vecse(\eta_{1})) , S_{2}^{-1}\eta_{2} \Bigr)\,,
\end{align*}
where $\on{Pr}_{\g}: \R^{4\times 4} \times \R^{6} \rightarrow \g$ denotes the projection onto the Lie algebra $\g$. Note that the second component of the projection is trivial.

This result coincides with \cite{saccon2013second} where the Hessian of the Hamiltonian regarding the second component is computed directly. Applying the $\vecg-$operation leads to

\begin{align*}
\vecg(& \Hess_{2}  \mc H^{-} (G,\mu,t)  [T_{\on{Id}}L_{G} \eta]) \\
= & - e^{\alpha(t-t_{0})} \vecg\Bigl( \matse(S_{1}^{-1} \vecse(\eta_{1})) , S_{2}^{-1}\eta_{2} \Bigr) \\
= & -e^{\alpha(t-t_{0})} ((S_{1}^{-1} \vecse(\eta_{1}))^{\T}, (S_{2}^{-1}\eta_{2})^{\T})^{\T} \\
= & -e^{\alpha(t-t_{0})} \underbrace{\begin{pmatrix} S_{1}^{-1} & \mathbf{0}_{6\times 6} \\ \mathbf{0}_{6 \times 6} & S_{2}^{-1} \end{pmatrix}}_{=:S^{-1}} \vecg(\eta)\,.
\end{align*}

Now we apply the $\vecg$-operation to the expression $Z(G^{\ast},t) \circ  \Hess_{2} \mc H^{-}(G^{\ast},0,t) \circ Z(G^{\ast},t)$:

\begin{align*}
 \vecg& \Bigl(Z(G^{\ast},t) \circ  \Hess_{2} \mc H^{-}(G^{\ast},0,t)[Z(G^{\ast},t)(\eta)]\Bigr) \\ 
= & K(t) \vecg \Bigl(\Hess_{2} \mc H^{-}(G^{\ast},0,t)[Z(G^{\ast},t)(\eta)]\Bigr) \\
= & -e^{\alpha(t-t_{0})} K(t)S^{-1} \vecg(Z(G^{\ast},t)(\eta)) \\
= & -e^{\alpha(t-t_{0})} K(t)S^{-1}K(t)\vecg(\eta) \,. 
\end{align*}
\end{enumerate} \qed
\end{proof}

\section{Christoffel symbols}\label{app:christoffel}

The Christoffel symbols $\Gamma_{ij}^{k}, i,j,k \in \{1,\dots,6\}$ for the Riemannian connection on $\SE$ are given by

\begin{align*}
\Gamma_{12}^{3} = \Gamma_{23}^{1} = \Gamma_{31}^{2} = & \tfrac{1}{2}\, , \\
\Gamma_{13}^{2} = \Gamma_{21}^{3} = \Gamma_{32}^{1} = & -\tfrac{1}{2}\, , \\
\Gamma_{15}^{6} = \Gamma_{26}^{4} = \Gamma_{34}^{5} = & 1 \,, \\
\Gamma_{16}^{5} = \Gamma_{24}^{6} = \Gamma_{35}^{4} = & -1\,  .
\end{align*}
and zero otherwise. Note that this Christoffel symbols are similar to these of the {\em kinematic} connection in~\cite{Zefran1999}. However, for the {\em Riemannian} connection, we need to switch the indexes $i$ and $j$.

\section{Derivations for Extended Kalman Filter} \label{sec:app-ext-kalman}

The function $\Phi: \R^{12} \rightarrow \R^{12\times 12}$ in Alg.~\ref{Alg:ExtendedKalman} is
\begin{align*}
\Phi(v) =  & \bpm \Phi_{\SE}(v_{1:6}) & \mathbf{0}_{6\times 6} \\ \mathbf{0}_{6\times 6} & \eins_{6} \epm \,,
\end{align*}
whereas the function $\Phi_{\SE}$ is given in \cite[Section ~10]{selig2004lie} (cf.~\cite[Eq.~(17)]{bourmaud2015continuous}).

\subsection{Derivations for non-linear Observations}
The expression of $H_{l}$ that is defined in \cite[Eq.~(59)]{bourmaud2015continuous} 
is simply the Riemannian gradient of the observation function $h_{k}$, i.e. 
\begin{align*}
H_{l} := \sum_{k=1}^{n} \D h_{k}(G(t_{l}))\,,
\end{align*}
where $h_k$ is defined as in \eqref{def:hk}; and the $\D h_{k}$ can be computed component-wise (for $j=1,2$) for $G(t_{l})=(E(t_{l}),v(t_{l}))$ by the directional derivative for  a direction $G\eta\in T_{G}\mc G.$
\begin{align}
\D & h_{k}^{j}(G)[G\eta] =  \D \bigl((e_{3}^{4}E^{-1}g_{k})^{-1} e_{j}^{4} E^{-1} g_{k}\bigr)[(E\eta_{1}, \eta_{2})] \\
= & \kappa_{k}^{-2} e_{3}^{4} \eta_{1}E^{-1}g_{k}e_{j}^{4} E^{-1}g_{k} - \kappa_{k}^{-1} e_{j}^{4} \eta_{1} E^{-1}g_{k} \\
= & \la \bigl(\kappa_{k}^{-2} E^{-1} g_{k}e_{j}^{4} E^{-1}g_{k}e_{3}^{4} - \kappa_{k}^{-1}E^{-1} g_{k}e_{j}^{4}\bigr)^{\T},\eta_{1} \ra \\
=: & \la \rho_{k}^{j}(G), \eta_{1}\ra \,,
\end{align}
where the second last line follows from the definition of the Riemannian metric on $\SE$, i.e. $\la \eta, \xi\ra_{\Id} = \eta^{\T}\xi$, and the fact that the trace is cyclic.
By projection of $\rho_{k}^{1}(G(t_{l}))$ onto the Lie algebra $\se$ and by vectorization, we obtain the Riemannian gradient. Stacking the vectors leads to the Jacobian  $H_{l} \in \R^{2\times 12}$, which is provided through
\begin{equation}\label{eq:Hl-nonlinear}
H_{l} = \sum_{k=1}^{l} \bpm \vecse(\on{Pr}(\rho_{k}^{1}(t_{l})))^{\T} &  \mathbf{0}_{1\times 6} \\  \vecse(\on{Pr}(\rho_{k}^{2}(t_{l})))^{\T} & \mathbf{0}_{1\times 6}\epm\,. 
\end{equation}

Next, we consider the calculation of the function $J(t)$ in Alg.~\ref{Alg:ExtendedKalman} in line \ref{line:dyn-P}.
Following \cite{bourmaud2015continuous}, $J(t)$ can be calculated as

\begin{equation}\label{eq:def-J-ext-Kalman}
J(t) = F(t) - \ad_{\g}(f(G(t))) + \tfrac{1}{12} C(S)\,,
\end{equation}
where the differential of $F(t) = \D f(G(t))$ can be computed as

\begin{equation}\label{eq:def-F-ext-Kalman}
F(t) = \bpm  \mathbf{0}_{6\times 6} & \eins_{6} \\ \mathbf{0}_{6\times 6} & \mathbf{0}_{6\times 6}\epm\,.
\end{equation}
For a diagonal weighting matrix $S$, we find that in \eqref{eq:def-J-ext-Kalman} the function $C$ can be computed for diagonal weighting matrices $S$ as

\begin{equation}
C(S) = \bpm  \bsm   \Xi & \mathbf{0}_{3\times 3} \\ \mathbf{0}_{3\times 3} & \Xi \esm &  \mathbf{0}_{6\times 6} \\ \mathbf{0}_{6\times 6} & \mathbf{0}_{6\times 6} \epm\,,
\end{equation}
where $\Xi = -\diag((S_{22}+S_{33},S_{11}+S_{33}, S_{11}+S_{22} )^{\T})$,
and the adjoint in \eqref{eq:def-J-ext-Kalman} can be computed with \eqref{eq:def-ad-g}.


\bibliographystyle{spmpsci}      
\bibliography{paper}   

\end{document}